\documentclass[reqno]{amsart}
\usepackage[top=1.1in, bottom=1in, left=1in, right=1in]{geometry}
\usepackage{amsfonts}
\usepackage{amssymb}
\usepackage{amsthm}
\usepackage{amsmath}
\usepackage{mathrsfs}
\usepackage{color}
\usepackage{enumerate}
\usepackage[numbers,sort&compress]{natbib}
\usepackage{pdfsync}
\usepackage{esint}
\usepackage{graphicx}
\usepackage{float}
\usepackage{caption}
\usepackage{subfigure}
\usepackage{enumitem}
\usepackage{appendix}

\allowdisplaybreaks

\usepackage{tikz} 
\usetikzlibrary{calc}
\usetikzlibrary{intersections}
\usetikzlibrary{patterns}

\usepackage[colorlinks,
linkcolor=black,       
anchorcolor=blue,      
citecolor=blue,        
]{hyperref}

\makeatother

\numberwithin{equation}{section}

\newtheorem{theorem}{Theorem}[section]
\newtheorem{definition}[theorem]{Definition}

\newtheorem{lemma}[theorem]{Lemma}
\newtheorem*{notation}{Notations}

\newtheorem{proposition}[theorem]{Proposition}

\numberwithin{equation}{section}

\theoremstyle{remark}
\newtheorem{remark}[theorem]{Remark}

\def\wt{\widetilde}

\def\9{{\infty}}
\def\ve{{\varepsilon}}
\def\na{{\nabla}}

\def\bbr{{\mathbb{R}}}

\def\({\left(}
\def\){\right)}
\def\<{\left<}
\def\>{\right>}
\def\ol{\overline}

\renewcommand{\epsilon}{\varepsilon}

\newcommand{\X}{{\mathbb{X}}}

\begin{document}
	
	\title[multi-solitons to
	focusing mass-supercritical SNLS]{multi-solitons to
		focusing mass-supercritical stochastic nonlinear Schr\"odinger equations}

	\author{Michael R\"ockner}
	\address{Fakult\"at f\"ur Mathematik,
		Universit\"at Bielefeld, D-33501 Bielefeld, Germany, 
		Academy of Mathematics and Systems Science, 
		Chinese Academy of Sciences, 
		Beijing, China.}
	\email{roeckner@math.uni-bielefeld.de}
	\thanks{}

	\author{Yiming Su}
	\address{School of Mathematics,
		Hangzhou Normal University, 311121 Hangzhou, China.}
	\email[Yiming Su]{yimingsu@amss.ac.cn}
	\thanks{}

	\author{Yanjun Sun}
	\address{School of Mathematical Sciences, Shanghai Jiao Tong University, China.}
	\email[Yanjun Sun]{0830-syj@sjtu.edu.cn}
	\thanks{}

	\author{Deng Zhang}
	\address{School of Mathematical Sciences, MOE-LSC,
		CMA-Shanghai, Shanghai Jiao Tong University, China.}
	\email[Deng Zhang]{dzhang@sjtu.edu.cn}
	\thanks{}

	\keywords{Multi-solitons, mass-supercriticality, 
    stochastic nonlinear Schr\"odinger equations,  controlled rough path}
	
	\subjclass[2020]{60H15, 35C08, 35Q51, 35Q55.}
	
	\begin{abstract}
		We consider the stochastic nonlinear Schr\"odinger equation driven by linear multiplicative noise in the mass-supercritical case. 
		Given arbitrary $K$ solitary waves
		with distinct speeds,
		we construct stochastic multi-solitons 
        pathwisely in the sense of controlled rough path, 
		which  behave asymptotically as the sum of the $K$ prescribed solitons as
		time tends to infinity. 
        In contrast to the mass-(sub)critical case in \cite{RSZ23}, 
        the linearized Schr\"odinger operator around the ground state has more unstable directions in the supercritical case. 
        Our pathwise construction utilizes the rescaling approach
and the modulation method in \cite{CMM11}.  
We derive the  quantitative decay rates
dictated by the noise 
for 
the unstable directions, 
as well as the modulation parameters and remainder 
in the geometrical decomposition.  
They are important to close the key  bootstrap estimates  
and to implement topological arguments to control the unstable directions. 
As a result, 
the temporal convergence rate of stochastic multi-solitons, 
which can be of either exponential or polynomial type,  
is related closely to the spatial decay rate of the noise 
and reflects the noise impact on soliton dynamics.
	\end{abstract}
	
	\maketitle
	
	\tableofcontents

	
	\section{Introduction and formulation of main results}
	\subsection{Introduction} 	
	We consider the focusing stochastic nonlinear Schr\"odinger equations (SNLS for short) with linear multiplicative noise:
	\begin{equation}  \label{SNLS}
		\begin{cases}dX(t)=\text{$i$}\Delta{X(t)}dt+\text{$i$}|X(t)|^{p-1}X(t)dt-\mu(t)X(t)dt+\sum\limits_{k=1}^{N}X(t)G_k(t)dB_k(t),
			\\X(T_0)=X_0\in{H^1(\mathbb{R}^d)}.  \tag{SNLS}
		\end{cases}
	\end{equation}
	Here, $p\in(1+\frac{4}{d},1+\frac{4}{(d-2)_+})$, where
	$\frac{4}{(d-2)_+} = +\infty$ if $d=1,2$, or
	$\frac{4}{d-2}$ if $d\geq 3$,
	that is, the nonlinearity lies in the mass-supercritical regime.
	Note that $p=1+\frac 4d$ or $1+\frac{4}{d-2}$
	correspond to the mass-critical or energy-critical case, respectively.
	Moreover,
	$B_k$, $1\leq k\leq N$, are standard $N$-dimensional real-valued Brownian motions on a stochastic basis $(\Omega,\mathscr{F},\{\mathscr{F}_t\},\mathbb{P})$ with normal filtration $\{\mathscr{F}_t\}$,
	and $G_k(t,x)=i\phi_k(x)g_k(t)$,
	$x\in\mathbb{R}^d$, $t>0$, where $\{\phi_k\}\subseteq C_b^{\infty}(\mathbb{R}^d,\mathbb{R})$,  
    $\{g_k\}\subseteq C^{\alpha}(\mathbb{R}^+,\mathbb{R})$ with $\alpha\in(1/3, 1/2)$ 
    are controlled rough path with respect 
    to $\{B_k\}$, 
    and $\{g_k\}$ and their Gubinelli's derivative are     
    $\{\mathscr{F}_t\}$-adapted. 
	The last term $X(t)G_k(t)dB_k(t)$ is taken in the sense of controlled rough paths (see Definition \ref{def1.2} below),
	and the term $\mu$ is the Stratonovich correction term,
    which is of the form
	\begin{equation}
		\mu(t,x)=\frac{1}{2}\sum\limits_{k=1}^{N}\phi^2_k(x)g^2_k(t),~~x\in\mathbb{R}^d,~t>0,
		\label{1.2}
	\end{equation}
	to ensure the conservation law of mass as required in the physical context (\cite{BCIRG94,BCIRG95}).
	We note that
	the stochastic term $XG_kdB_k(t)$ can be viewed as a random potential acting on the quantum system. It  coincides with the It\^o integral 
	when  $\{X(t)\}$ is $\{\mathscr{F}_t\}$-adapted.
	In the special case where the noise is absent, i.e.,
	$B_k \equiv 0$, $1\leq k\leq N$,
	\eqref{SNLS} reduces to the classical nonlinear Schr\"odinger equation (NLS for short)
	\begin{equation}  \label{NLS}
		\begin{cases}
			i\partial_tu+\Delta u+|u|^{p-1}u=0,\\
			u(T_0)=u_0 \in H^1(\mathbb{R}^d).   \tag{NLS}
		\end{cases}
	\end{equation}

	The physical significance of SNLS is well-known. The 3D cubic NLS, which is a typical mass-supercritical model, 
     is of physical importance 
     in nonlinear optics 
     and 
	describes paraxial propagation of laser beams in a homogeneous Kerr medium, see \cite{Fibich}.
	In a crystal
	the noise corresponds to scattering of excitons by phonons,
	and the noise effect on the coherence of the ground state solitary waves
	was investigated in \cite{BCIRG95}.
	It also arises from the physical model
	of monolayer Scheibe aggregates \cite{BCIRG94}.
	Moreover, the noise effect on its collapse process was studied in \cite{RGBC95}.
	We also refer to  \cite{DD02,DD022}
	for numerical observations of noise effects
	on blow-up,
	and \cite{MRY21,MRRY21} where stochastic stable blow-up
	solutions have been investigated.

	Local well-posedness of \eqref{SNLS} and \eqref{NLS} is  known in the energy space $H^1$,
	see,  e.g.,  \cite{C03,dD03,BRZ16}.
	The main interest of this paper is to study the large-time
	dynamics, especially, the soliton dynamics
	of \eqref{SNLS} in the mass-supercritical case.
	
	According to the celebrated soliton resolution conjecture, 
	generic global solutions to NLS are expected to behave asymptotically as a superposition of solitons plus a dispersive decaying profile.
	In the last decades,
	significant progress
	has been achieved on the
	soliton resolution conjecture for energy-critical nonlinear wave equations,
	see \cite{DKM23, JL23} and references therein.
	Multi-solitons to NLS
	were constructed initially in the mass-critical case \cite{M90}.
	Afterwards,
	multi-solitons have been constructed in various settings,
	including the mass-subcritical case \cite{MM06},
	the mass-supercritical case \cite{CMM11, C14, N19},
	and the energy-critical case \cite{J17}.
	Non-pure multi-solitons
	(including their scattering profile),
	predicted by the soliton resolution conjecture,
	have been recently constructed in \cite{RSZ24},
	and uniqueness was proved in the solution class with $t^{-5-}$ decay rate.

	It should be mentioned that
	soliton dynamics in the mass-supercritical setting is much more complicated
	than in the (sub)critical case.
	One major obstruction is
	that
	the linearized Schr\"odinger operator in the supercritical case has more unstable directions
	than in the (sub)critical case. In fact, the eigenvalue of the linearized operator around the soliton $e^{it}Q$ in the (sub)critical case is exactly zero, 
    while in the supercritical case there exist two additional nonzero  real eigenvalues (see \cite{G90, S06, W85}).  
	As a result,
	the orthogonal conditions
	in the geometrical decomposition are
	insufficient to control all unstable directions.
	Moreover,
	in \cite{DR10}, Duyckaerts and Roudenko constructed global solutions $U(t)$ to the 3D focusing cubic NLS,
	such that $\lim_{t\rightarrow+\infty}\|U(t)-e^{it}Q\|_{H^1}=0$ whereas $U(t)\neq e^{it}Q$. However,
	in the (sub)critical case,
	no such special solutions $U(t)$ can exist, due to the
	variational property of the ground state Q and the corresponding linearized operator. In this spirit, a family of multi-solitons have been constructed 
    by Combet \cite{C14} with the same asymptotic behavior. This is in  contrast to the (sub)critical case, where multi-solitons are believed to be asymptotically unique, 
    see the conjecture raised by Martel \cite{M18}. Very recently, in the
	(sub)critical setting, the uniqueness of multi-solitons with polynomial asymptotic rate
	has been obtained by C\^ote and Friederich \cite{CF21},
	Cao, Su and Zhang \cite{CSZ23}.

	In the stochastic case, 
    more difficulties occur in 
	the study of large-time dynamics
	of SNLS.
	As a matter of fact,
	the presence of noise even destroys the basic conservation law of the energy.
	The energy of solutions to SNLS indeed evolves as a continuous semimartingale.
	Its evolution was carefully studied
	by numerical method in \cite{MRY21, MRRY21}.
	Furthermore,
	because solitons
	are unstable with respect to $H^1$ perturbations,
	it is a priori unclear whether the input of noise
	destroys the soliton dynamics.
	This is very different from the
	scattering dynamics in \cite{HRZ19, Z23},
	which is stable under $H^1$ perturbations.

	In \cite{dD02, dD05},
	it was first proved by de Bouard and Debussche that non-degenerate noise
	in the supercritical case
	can accelerate blow-up
	with positive probability.
	Afterwards,
	the small noise large deviation principle
	and the error in soliton transmission
	have been studied
	in \cite{DG08}. 
    Moreover, 
    for the 2-D Gross-Pitaevskii equation perturbed by a random quadratic potential, 
    it was proved in \cite{dF09} 
    that the solution 
    with initial condition of a standing wave 
    decomposes into the sum of a randomly modulated standing wave and a small remainder, 
    and the first order of the remainder converges to a Gaussian process. 
	See also \cite{dD07,dD09}
	for the soliton dynamics of stochastic Korteweg-de Vries equations.
	Recently,
	the quantitative construction
	of blow-up solutions
	have been obtained for the mass-critical  SNLS.
	We refer to   \cite{SZ23} for
	critical mass blow-up solutions, \cite{FSZ22} for
	stochastic log-log blow-up solutions,
	and \cite{RSZ24, SZ20}
	for multi-bubble
	(Bourgain-Wang type) blow-up solutions.
	
	In addition to the lack of energy conservation,
	the pseudo-conformal symmetry of mass-critical NLS is also destroyed
	by the noise. 
	As a result,
    unlike in the deterministic case, 
    stochastic multi-solitons cannot be directly derived from the aforementioned stochastic blow-up solutions. 
    This fact forces to
	construct stochastic multi-solitons
	on the soliton level.
	Recently,
	stochastic multi-solitons to mass-(sub)critical SNLS, i.e., \eqref{SNLS} with $1<p\leq 1+ 4/d$,
	have been constructed in \cite{RSZ23}.
	The construction of stochastic multi-solitons in the mass-supercritical case, 
    however, 
	remains open.

	The aim of the present work is to address this problem for  the mass-supercritical \eqref{SNLS}.
	More precisely,
	we construct stochastic multi-solitons to \eqref{SNLS} 
    in a pathwise way in the sense of controlled rough path. 
    The constructed stochastic solutions 
    behave asymptotically like a sum of solitary waves with distinct speeds,
	see Theorem \ref{th1.3} below. 
    This provides new examples for the soliton resolution conjecture in the stochastic case. 
    Our proof reveals that, 
    though solitons are unstable under $H^1$ perturbation of initial data, 
	the construction of multi-solitons
	in some sense has structural stability, 
    that is, 
    it is stable under perturbation
	of the NLS by first and zero order terms caused by the noise (see \eqref{RNLS} below).
	We construct
	the stochastic multi-solitons
	in  two scenarios of noise,
	which correspond to the exponential and polynomial spatial decay rates of noise, respectively. 
Quantitative decay rates of unstable directions, 
and modulation parameters and the remainder in the geometrical decomposition are derived. 
Interestingly,
	the temporal decay rate
	of stochastic multi-solitons
	is dictated by the spatial decay rate of the noise,
	which reveals the noise effect on soliton dynamics.

	\subsection{Main results}
	
	Before formulating the main results, 
    let us first recall 
	some basic notions
	in the theory of controlled rough paths from  \cite{FH14,G04}.  

    Fix $\alpha\in (1/3,1/2)$.  
    For $I=[S,T]\subseteq \mathbb{R}^+$, 
    let $\mathscr{C}^{\alpha}(I,\mathbb{R}^N)$ denote the space of $\alpha$-H\"older rough paths 
    $(X,\mathbb{X})$,  
 such that 
     $X\in C^{\alpha}(I; \mathbb{R}^N)$, 
      $\mathbb{X}\in C^{2\alpha}(I^2; \mathbb{R}^{N\times N})$, 
 and the Chen relation holds 
	\begin{equation*}
		\mathbb{X}(s,t) - \mathbb{X}(s,u) -\mathbb{X}(u,t)
        = \delta X_{su}\delta X_{ut} 
	\end{equation*}
    for all $S\le s\le u\le t\le T$. 
    For simplicity we write  
    \begin{align*}
     \|X\|_{\alpha,I}:=\sup_{s,t\in I,s\neq t}\frac{|\delta X_{st}|}{|t-s|^\alpha}<\infty, 
     \ \ 
      \|\mathbb{X}\|_{2\alpha,I}
      :=\sup_{s,t\in I,s\neq t}
      \frac{|\mathbb{X}(s,t)|}{|t-s|^{2\alpha}}
      <\infty, 
    \end{align*} 
    where 
    $\delta X_{st}:= X(t)-X(s)$.   
    We also set $\dot g:=\frac{dg}{dt}$ for any $C^1$ functions.

	Given a path $X\in C^{\alpha}([S,T],\mathbb{R}^N)$, 
    $0\leq S<T<\infty$, 
    we recall that a pair $(Y,Y')$ is a controlled rough path with respect to $X$, if $Y\in C^{\alpha}([S,T],\mathbb{R}^N)$, $Y'\in C^{\alpha}([S,T],\mathbb{R}^{N\times N})$, 
    and the remainder term $R^Y$, implicitly given by
	\begin{equation*}
		\delta Y_{k,st}=\sum\limits_{j=1}^{N}Y'_{kj}(s)\delta X_{j,st}+\delta R_{k,st}^Y,
	\end{equation*}
	satisfies $\|R_k^Y\|_{2\alpha, [S,T]}<\infty$, $1\le k\le N$.
	$Y'$ is the so-called Gubinelli derivative of $Y$.
	Let $\mathscr{D}_X^{2\alpha}([S,T],\mathbb{R}^N)$
	denote the space of controlled rough paths with respect to $X$.
	
	For any $\alpha\in(1/3, 1/2)$, 
    it is well-known that the $N$-dimensional Brownian motion $B=(B_j)_{j=1}^N$ can be enhanced to 
    the $\alpha$-H\"older rough path $\mathbf{B}=(B,\mathbb{B})$, where 
    $\mathbb{B}(s,t):=\int_s^t\delta B_{sr} \otimes dB(r) \in \mathbb{R}^N \times \mathbb{R}^N$  
    is taken in the sense of It\^o, 
    $0\leq s<t<\infty$. 
	For any $T\in (0,\infty)$, it holds that 
	$\mathbf{B} \in \mathscr{C}^{\alpha}([0,T],\mathbb{R}^N)$ almost surely,  
	see \cite[Proposition 3.4]{FH14}. 
	
	Given a controlled rough path $(Y,Y')\in\mathscr{D}_B^{2\alpha}([S,T],\mathbb{R}^N)$, 
    one can define the rough integration of $Y$ against $\mathbf{B}=(B,\mathbb{B})$ as follows 
    (see 
    \cite[Theorem 4.10]{FH14}, 
    \cite[Corollary 2]{G04}): 
    For each $1\le k\le N$,
	\begin{equation*}
		\int_S^TY_k(r)dB_k(r)
        :=\lim_{|\mathcal{P}|\to0}\sum\limits_{l=1}^N\Big(Y_k(t_l)\delta B_{k,t_lt_{l+1}}+\sum\limits_{j=1}^NY'_{kj}(t_l)\mathbb{B}_{jk}(t_l,t_{l+1})\Big),
	\end{equation*}
	where $\mathcal{P}:=\{t_l\}_{l=0}^n$ 
    is a partition of $[S,T]$ so that $t_0=S$, $t_n=T$, $|\mathcal{P}|:=\max_{0\le l\le n-1}|t_{l+1}-t_l|$.

	\medskip
	As in \cite{RSZ23},
	we assume that the noise in \eqref{SNLS} satisfies the following conditions:
	
	\begin{enumerate}
		\item[$(A0)$:]
		Asymptotic flatness. For every $1\le k\le N$, $\phi_k\subseteq C_b^{\infty}(\mathbb{R}^d,\mathbb{R})$ such that
		\begin{equation}
			\lim_{|x|\to\infty}|x|^2|\partial_x^\nu\phi_k(x)|=0, \ \nu\neq0.
			\label{1.13}
		\end{equation}
		
		\item[$(A1)$:]
		$\{g_k\}_{k=1}^N$ are $\{\mathscr{F}_t\}$-adapted continuous processes controlled by the Brownian motions $\{B_k\}$ with the Gubinelli derivative $\{g_{kj}'\}_{j,k=1}^N$,
		and $g_k\in L^2(\mathbb{R}^+)$, $1\leq k\leq N$, $\mathbb{P}$-a.s..
		Moreover, one of the following cases holds:
		
		\noindent Case (I): For every $1\leq k\leq N$, there exists $c_k>0$ such that for $|x|>0$, 
		\begin{equation}
			\sum\limits_{|\nu|\le4}|\partial^{\nu}_x\phi_k(x)|\le Ce^{-c_k|x|}.
			\label{1.14}
		\end{equation}
		
		\noindent Case (II): Let $\nu_*\in \mathbb{N}$.  
        For every $1\leq k\leq N$ and $|x|>0$,
		\begin{equation}
			\sum\limits_{|\nu|\le4}|\partial^{\nu}_x\phi_k(x)|\le C|x|^{-\nu_*}.
			\label{1.16}
		\end{equation}
		In addition,
		there exists a  
        random time $\sigma_*$ 
        and a 
        deterministic constant $c^*>0$ such that $\mathbb{P}$-a.s. 
        $\sigma_*\in [0,\infty)$ 
        and 
        for any $t\geq \sigma_*$,  
		\begin{equation}
\int_{t}^{\infty}g_k^2ds\log\left(\int_{t}^{\infty}g_k^2ds\right)^{-1}\le\frac{c^*}{t^2}, \quad 
			1\leq k\leq N. 
			\label{1.15}
		\end{equation}
	\end{enumerate}

	We note that
	{\rm Case (I)} and {\rm Case (II)}
	correspond to the exponential
	and polynomial spatial decay rates of noise, respectively.
	Without loss of generality,
	we consider in Assumptions 
    {\rm (A0)} and {\rm (A1)} 
    \begin{align*}
    \sum\limits_{|\nu|\le4}|\partial^{\nu}_x\phi_k(x)|\le  C 
    \phi(|x|) 
    \end{align*}
    with a decay  function $\phi$
	of the following form
	\begin{equation}
		\phi(|x|):=
		\begin{cases}
			e^{-|x|},&\text{in Case (I)};\\
			|x|^{-\nu_*},&\text{in Case (II)}.
		\end{cases}
		\label{1.17}
	\end{equation}
	As we shall see later,
	these spatial decay rates of noise
	indeed affect the temporal decay rate of
	stochastic solitary waves, see \eqref{1.20}
	and \eqref{1.22} below.
	
	The temporal condition \eqref{1.15}  relates  to the Levy H\"older continuity of Brownian motions,
	which permits to control the tail of the noise  $B_*(t)$
	in \eqref{B*-def} below 
    for $t$ large enough.
	It is worth noting that
	the $t^{-1}$ decay rate of $B_*(t)$ in \eqref{B*-bdd}
	is essential to close the bootstrap estimate of $\|\varepsilon\|_{H^1}$ in Case (II),
	see, e.g.,  \eqref{5.24} below.

	Let us also mention that
	the asymptotic flatness condition $(A0)$ ensures the local well-posedness of \eqref{SNLS},
	see \cite{BRZ16, Z23}.
	The smoothness condition of the spatial functions $\{\phi_k\}$ is assumed merely for simplicity.
	One can also treat infinitely many Brownian motions,
	i.e., $N=\infty$, 
    and noise of low spatial regularity in certain Sobolev and lateral Strichartz spaces,
	as in the context of Zakharov system  \cite{HRSZ24, HRSZ25}.

	\medskip
	The solutions to  \eqref{SNLS} are defined in the following controlled rough path sense. 
    
	\begin{definition}\label{def1.2}
		Let $p\in (1+\frac{4}{d},1+\frac{4}{(d-2)_+})$,
		$d\geq 1$. We say that X is a solution to \eqref{SNLS} on $\left[T_0,\tau^*\right)$,
		where $T_0$, $\tau^*$ $\in$ $\left(0,\infty\right]$ are random variables, 
        if for $\mathbb{P}$-a.e. $\omega\in \Omega$ 
        and for any $\varphi$ $\in$ $C_c^\infty$, 
        $t \mapsto \<X(t, \omega), \varphi\>$ 
      is continuous on 
      $[T_0(\omega),\tau^*(\omega) )$ and for any $T_0(\omega) \le s < t \le \tau^*(\omega)$,
		\begin{equation}
			\langle{X(t)-X(s)},\varphi\rangle-\int_{s}^{t}\langle{\text{$i$}X},\Delta\varphi\rangle+\langle{\text{$i$}|X|^{p-1}X},\varphi\rangle-\langle{\mu{X}},\varphi\rangle{dr}=\sum\limits_{k=1}^{N}\int_{s}^{t}\langle{\text{$i$}\phi_kg_k(r)X(r)},\varphi\rangle{dB_k(r)}.
			\label{1.18}
		\end{equation}
		Here the integral $\int_{s}^{t}\langle{\text{$i$}\phi_kg_kX},\varphi\rangle{dB_k(r)}$ is taken in the sense of controlled rough paths with respect to the 
        Brownian rough path 
        $\{(B, \mathbb{B})\}$, 
        and  $\langle{\text{$i$}\phi_kg_kX},\varphi\rangle 
        \in \mathscr{D}_B^{2\alpha}([s,t],\mathbb{R})$, satisfying 
		\begin{equation}
			\delta\langle{i\phi_kg_kX},\varphi\rangle_{st}=\sum\limits_{j=1}^{N}\langle{-\phi_j\phi_kg_j(s)g_k(s)X(s)+\text{$i$}\phi_kg_{kj}'(s)X(s)},\varphi\rangle\delta{B_{j,st}}+\delta{R^{\langle{i\phi_kg_kX},\varphi\rangle}_{k,st}},
			\label{1.19}
		\end{equation}
		and $\Vert\langle\phi_j\phi_kg_jg_kX,\varphi\rangle\Vert_{\alpha,[s,t]} + \Vert\langle\phi_kg_{kj}'X,\varphi\rangle\Vert_{\alpha,[s,t]} < \infty$, $\Vert{R^{\langle{i\phi_kg_kX},\varphi\rangle}_k}\Vert_{2\alpha,[s,t]} < \infty$, $\alpha$ $\in$ $\left(\frac{1}{3},\frac{1}{2}\right)$.
	\end{definition}

	In the characterization of soliton dynamics, 
    a key role is played by the ground state, which is the unique radial positive solution to the nonlinear elliptic equation
	\begin{equation}
		\Delta{Q}-Q+Q^p=0.
		\label{1.6}
	\end{equation}
	It is known (see, e.g., \cite{BL83})
	that the ground state decays  exponentially fast at infinity, i.e., there exist $C$, $\delta>0$ such that for any $|\nu|\le3$,
	\begin{equation}
		|\partial_x^\nu Q(x)|\le Ce^{-\delta |x|},\ \ x\in\bbr^d.
		\label{1.7}
	\end{equation} 
	For any $w>0$, 
	let $Q_{w}$ denote the rescaled ground state
	\begin{equation}
		Q_{w}(x):=w^{-\frac{2}{p-1}}Q(\frac{x}{w}), \quad x\in \mathbb{R}^d.
		\label{1.5}
	\end{equation}
	Note that by the ground state equation
	\eqref{1.6},
	$Q_{w}$ satisfies the equation
	\begin{equation}
		\Delta{Q}_{w}-w^{-2}Q_{w}+Q_{w}^p=0.
		\label{1.4}
	\end{equation}

	Given any $K\in \mathbb{N}$ 
	and any 
	$w_k\in\bbr^+$, $\alpha_k^0\in\bbr^d$, $\theta^0_k\in\bbr$, 
	$1\leq k\leq K$, 
	our aim is to construct stochastic multi-solitons
	which behave asymptotically as
	a sum of solitary waves
	\begin{equation}
		R_k(t,x):=Q_{w_k}(x-v_kt-\alpha_k^0)e^{\text{$i$}(\frac{1}{2}v_k\cdot{x}-\frac{1}{4}|v_k|^2t+(w_k)^{-2}t+\theta_k^0)} 
		\label{1.11}
	\end{equation}
	with distinct velocities 
	$$
	v_k\neq v_{k'} \quad  {\rm for\ any}\ k\neq k'. 
	$$

	\medskip
	The main result of this paper is the following:

	\begin{theorem}   \label{th1.3}
		Consider \eqref{SNLS} with $p\in{(1+\frac{4}{d},1+\frac{4}{(d-2)_+})}$, $d\geq 1$.
		Assume $(A0)$ and $(A1)$ with $\nu_*\ge\nu_0$ in Case (II),
		where $\nu_0$ is a deterministic constant given by \eqref{def-v0} below.
		Then, 
		there exists a positive random time $T_0$ such that for $\mathbb{P}$-a.e. $\omega\in\Omega$, there exist $X_*(\omega)\in H^1$ and an $H^1$-valued solution $X(t,\omega)$ to \eqref{SNLS} on $[T_0(\omega),\infty)$ satisfying $X(T_0,\omega)=X_*(\omega)$ and
		\begin{equation}
			\Vert e^{-W_*(t)}X(t)-\sum\limits_{k=1}^{K}R_k(t)\Vert_{H^1}\le Ct\phi^{\frac{1}{2}}(\delta t), \ t\ge T_0.
			\label{1.20}
		\end{equation}
		Here, the random phase function 
        $W^*$ is given by 
		\begin{equation}
			W_*(t,x)=-\sum\limits_{k=1}^{N}\int_{t}^{\infty}i\phi_k(x)g_k(s)dB_k(s),
			\label{1.21}
		\end{equation}
		the solitary waves $\{R_k\}$ are given by
		\eqref{1.11}, $\phi$ is the spatial decay function of the noise in (\ref{1.17}) and $C,\delta>0$  
        are deterministic positive constants. 
        In particular, for $t\ge T_0$,
		\begin{equation}
			\Vert X(t)-\sum\limits_{k=1}^{K}R_k(t)\Vert_{H^1}\le C\sum\limits_{k=1}^{N}\left(\int_{t}^{\infty}g_k(s)^2ds\log\left(\int_{t}^{\infty}g_k(s)^2ds\right)^{-1}\right)^{\frac{1}{2}}+Ct\phi^{\frac{1}{2}}(\delta t).
			\label{1.22}
		\end{equation}
	\end{theorem}

	\begin{remark} 
    Because the right-hand side of \eqref{1.22} tends to 0 as $t\to\infty$ due to
		$\{g_l\}\subseteq L^2(\mathbb{R}^+)$ in $(A1)$, 
		the constructed stochastic solution in Theorem \ref{th1.3} converges to the given $K$ solitary waves as time tends to infinity.  
		As a result, Theorem \ref{th1.3} provides new examples for the soliton resolution conjecture
		in the stochastic mass-supercritical case.

	We also note that the temporal decay rate of stochastic multi-solitons in \eqref{1.20}
are dictated by the spatial decay rate of the noise. 
The decay rate can be of either exponential or polynomial type 
in two scenarios of noise 
Cases (I) and (II), respectively.
	This reflects the noise impact on soliton dynamics.
		\end{remark}

\begin{remark} 
It should be mentioned that, 
compared to the (sub)critical case \cite{RSZ23}, 
many difficulties emerge in the present 
mass-supercritical case. 
One major difficulty is that 
the linearized Schr\"odinger operator has two extra unstable directions 
$\mathbf{a}^\pm$ (see \eqref{3.50} below) in the supercritical case, 
which are much harder to control.  

The strategy here employs the modulation method 
inspired by \cite{CMM11} 
to modulate the final data of approximating solutions.  
As an immediate technical issue, 
the control of modulated final data 
to prescribed vectors 
requires 
more delicate analysis of 
the non-degeneracy of Jacobian matrices 
than in the (sub)critical 
case \cite{RSZ23},  
see Proposition \ref{prop3.5} below. 
We also derive that the radius of the modulation parameter 
and the constants in Proposition \ref{prop3.5} 
can be chosen to be deterministic, 
which is important to 
take a large random time 
to close the bootstrap estimates 
of modulation parameters and remainder in the geometrical decomposition. 

We remark that the 
bootstrap estimates are crucial to obtain 
uniform estimates of approximating solutions. 
In contrast to the (sub)critical case in \cite{RSZ23}, 
the bootstrap estimates in the supercritical case require 
an a-priori estimate of the unstable direction $\mathbf{a}^-$, which, however, 
cannot be closed by Gronwall's argument.
In order to overcome this problem, 
a topological argument for $\mathbf{a}^-$, 
based on Brouwer’s fixed point theorem, 
is performed on a ball with radius 
dictated by the spatial decay rate of the noise.  
We refer to Subsection \ref{Subsec-Strat-Proof} below for more detailed explanations of the difference between 
the supercritical and (sub)critical cases, 
as well as the deterministic case. 
\end{remark}

\begin{remark}
We are not sure about the 
measurability of 
the solution $X$
constructed in Theorem \ref{th1.3}. 
The measurability issue arises from 
the choice of $\boldsymbol{a}_n^-(\omega)\in B_{\mathbb{R}^K}(\phi^{\frac{1}{2}+\frac{1}{4d}}(\widetilde{\delta}n))$ 
in Proposition \ref{prop5.4} 
which is based on a topological argument, 
as well as from the compactness argument 
used to obtain a subsequence of approximating solutions in the proof of Theorem \ref{th1.6} 
in Section \ref{Sec-Proof-Main}.   
It does not seem obvious here how to use 
measurable selection theorems 
for instance in \cite{W80} to select measurable versions of 
$\boldsymbol{a}_n^-$ and of $X$ 
in our situation. 
\end{remark}

	\subsection{Strategy of the proof} 
    \label{Subsec-Strat-Proof} 
	Our proof utilizes the rescaling approach and the modulation analysis, 
	including topological arguments to treat the new unstable directions of the linearized Schr\"odinger operator,
	which is different from \cite{RSZ23} in the mass-(sub)critical case.
	
	To be precise,
	we use the rescaling or Doss-Sussman type transformation
	\begin{equation}
		u(t):=e^{-W_*(t)}X(t),
		\label{1.28}
	\end{equation}
	where $W_*$ is given by \eqref{1.21},
	to transfer
	\eqref{SNLS} to a random NLS
	\begin{equation} \label{RNLS}
		\begin{cases}
			\text{$i$}\partial_tu+(\Delta+b_*\cdot\nabla+c_*)u+|u|^{p-1}u=0,
			\\
			u(T_0)=e^{-W_*(T_0)}X_0,
		\end{cases}
		\tag{RNLS}
	\end{equation}
	where the random coefficients $b_*$ and $c_*$
	have the expressions
	\begin{align}
		b_*(t,x)&=2\nabla{W_*(t,x)}=2\text{$i$}\sum\limits_{k=1}^{N}\int_{t}^{\infty}\nabla{\phi_k(x)}g_k(s)dB_k(s),
		\label{1.30}
		\\c_*(t,x)&=\sum\limits_{j=1}^{d}(\partial_jW_*(t,x))^2+\Delta{W_*(t,x)}\nonumber\\
		&=-\sum\limits_{j=1}^{d}\left(\sum\limits_{k=1}^{N}\int_{t}^{\infty}\partial_j\phi_k(x)g_k(s)dB_k(s)\right)^2
		+\text{$i$}\sum\limits_{k=1}^{N}\int_{t}^{\infty}\Delta{\phi_k(x)g_k(s)dB_k(s)}.
		\label{1.31}
	\end{align}

	The rescaled equation \eqref{RNLS} enables us to perform pathwise analysis
	in a sharp way,
	which is in general not possible for It\^o integral.
	This approach works successfully for
	well-posedness and optimal control problems,
	see \cite{BRZ14, BRZ16, BRZ17, Z20, Z23}.
	It is also comparable with the Fourier restriction method
	and permits to exploit the noise regularization effect on scattering.
	We refer the interested readers to \cite{HRSZ24,HRSZ25,HRZ19,SZZ25}.
	
	It should be mentioned that, 
    though the rescaling transform provides 
    a nice way to reveal the structure of SNLS, 
    it does not remove the difficulty 
    in the analysis.  
	One obstacle is to control the derivative term
	$b_*\cdot\nabla u$ caused by noise, 
	which is in general difficult for Schr\"odinger equations,
	due to the lack of global regularity of Schr\"odinger groups.
	See, for instance,
	the case of the Schr\"odinger map \cite{BIKT11}.
	This problematic term
	can be controlled here under the asymptotic flat  condition $(A0)$ of the noise,
	by using the local smoothing estimates in \cite{HRZ19,Z22}.
	An alternative way to control this term
	is to use lateral Strichartz estimates,
	as in the context of stochastic Zakharov systems
	\cite{HRSZ24, HRSZ25},
	which help
	to weaken the regularity condition on the noise.

	The $H^1$ local well-posedness of \eqref{RNLS}
	can be proved by using analogous arguments as in \cite{BRZ16, HRZ19}. 
    Then the $H^1$ well-posedness of \eqref{SNLS} 
	can be inherited from that of \eqref{RNLS} 
    via Theorem \ref{Thm-SNLS-H1} below.  
	This fact has been proved in \cite{RSZ23} 
	in the  mass-(sub)critical case.
	With slight modifications, 
	based on the Sobolev embedding $H^1(\mathbb{R}^d)\hookrightarrow L^{p+1}(\mathbb{R}^d)$ with $p\in (1+\frac{4}{d},1+\frac{4}{(d-2)_+})$, 
	the proof there also applies to 
	the present mass-supercritical case.

	\begin{theorem}  \label{Thm-SNLS-H1}
		Let $p\in (1+\frac 4d, 1+ \frac{4}{(d-2)_+})$, $d\geq 1$. Let u be the solution to \eqref{RNLS} on $[T_0, \tau^*)$ with $X_0\in H^1$,
		where $T_0, \tau^*\in (0,\infty]$ are random variables.
		Then, for $\mathbb{P}$-a.e. $\omega\in \Omega$,
		$X(\omega):= e^{W_*(\omega)} u(\omega)$ is the solution to \eqref{SNLS} on $[T_0(\omega), \tau^*(\omega))$ in the sense of Definition \ref{def1.2}.
	\end{theorem}

	As a result, the proof of Theorem \ref{th1.3} can be reduced to that of the following result for the rescaled random equation \eqref{RNLS}.

	\begin{theorem}\label{th1.6}
		Consider \eqref{RNLS} with $p\in (1+\frac{4}{d},1+\frac{4}{(d-2)_+} )$, 
        $d\geq 1$.
		Assume $(A_0)$ and $(A_1)$ with $\nu_*\ge\nu_0$ in Case (II),
		where $\nu_0$ is a deterministic constant given by \eqref{def-v0}.
		Then, 
		there exists a positive random time $T_0$ such that for $\mathbb{P}$-a.e. $\omega\in\Omega$, there exist $u_*(\omega)\in H^1$
		and a unique solution $u \in C([T_0(\omega),\infty];H^1)$ to  (\ref{RNLS}) satisfying $u(\omega,T_0)=u_*(\omega)$ and 
		\begin{equation}
			\Vert u(t)-\sum\limits_{k=1}^{K}R_k(t)\Vert_{H^1}\le Ct\phi^{\frac{1}{2}}(\delta t), \ t\ge T_0.
			\label{1.32}
		\end{equation}
		where the solitary waves $\{R_k\}$ are given by \eqref{1.11}, $\phi$ is the decay function in (\ref{1.17}) and $C,\delta>0$.
	\end{theorem}

	In the sequel, 
	we mainly prove Theorem \ref{th1.6}. 
	The proof utilizes the modulation method developed in \cite{CMM11,MM06} 
	and mainly proceeds in the following three steps.  
	The impact of noise
	on the construction of
	stochastic multi-solitons 
	is presented below as well.

	\medskip
	\paragraph{\bf $\bullet$ Geometric decomposition:}
	First in Section \ref{sec2},
	we derive the geometrical decomposition
	of approximating solutions into
	a sum of $K$ solitons plus a remainder term  
   $u=\sum_{k=1}^{K}\widetilde{R}_k+\varepsilon$, 
where $\widetilde{R}_k$ 
are modulated soliton profiles 
with modulation parameters $\{(\alpha_k, \theta_k)\}$, 
and $\ve$ is the remainder 
(see Proposition \ref{prop3.2} below).

    As mentioned above, 
    unlike in the mass-(sub)critical case \cite{RSZ23},
	the linearized Schr\"odinger operator 
   $\mathscr{L}$ (see \eqref{1.8})
    has four unstable directions, 
    reflected by the following 
    coercivity type estimate  
    \begin{equation*}
		\left(\mathscr{L}f,f\right)\ge C\Vert f\Vert_{H^1}^2-\frac{1}{C}\left(\!\left(\int\nabla Qf_1dx\right)^2\!+\!\left(\int Qf_2dx\right)^2\!+\!\left(\text{Im}\int Y^+\bar{f}dx\right)^2\!+\!\left(\text{Im}\int Y^-\bar{f}dx\!\right)^2\right) 
	\end{equation*} 
	for any $f=f_1+if_2\in H^1$, 
    where 
    $Y^\pm$ are two eigenfunctions 
    of the linearized 
    Schr\"odinger operator 
    in the mass-supercritical case.  
    
    The new eigenfunctions 
    give rise to two extra unstable directions 
    $a_k^\pm := {\rm Im} \int \wt{Y}_k^\pm \ol{\ve} dx$,
	where $\wt{Y}_k^\pm$ are the modulated eigenfunctions of $Y^\pm$ with speed $v_k$, 
	$1\leq k\leq N$. 
    The extra unstable directions 
	cannot be canceled by imposing
	orthogonal conditions in the geometric decomposition as in the (sub)critical case \cite{RSZ23}. 
	Instead, 
	they 
	are controlled by modulating the final data
	with the eigenfunctions. 
    That is, 
    instead of equation \eqref{RNLS},  
    we consider the following approximating random equation 
    \begin{equation}
		\begin{cases}\text{$i$}\partial_tu+(\Delta+b_*\cdot\nabla+c_*)u+|u|^{p-1}u=0,
			\\[5pt]u(T)=R(T)+\text{$i$}\sum\limits_{k,\pm}b_k^{\pm}Y_k^{\pm}(T)
		\end{cases}
	\end{equation}  
    (see \eqref{3.1} below). 
    One advantage to modulate the final data 
    is, that it allows to steer the unstable 
    directions $(a_k^+(T), a_k^-(T))$ at time $T$ to any prescribed 
    vector in 
    ${\bf 0}\times B_{\mathbb{R}^K}(r_0)$ 
    with $r_0 \ll 1$. 
    The  proof of this fact requires 
    careful analysis of the 
    {\it non-degeneracy}  
    of Jacobian matrices,  
    which do not appear 
	in the mass-(sub)critical case \cite{RSZ23}. 
    See the proof of Proposition \ref{prop3.5} below.

	Let us mention that 
	a detailed proof for the 
    {\it exponential decay} of the eigenfunctions $Y^\pm$ 
    is given in Lemma \ref{Y} below. 
	The exponential decay property of the eigenfunctions, 
    as well as of 
    the ground state, 
    allow us 
	to decouple the iterations between
	different soliton profiles and eigenfunctions,
	at the cost of exponential decay orders, 
	which are favorable in the bootstrap estimates.

	\medskip
	\paragraph{\bf $\bullet$ Bootstrap estimates:}
	The next step
	is to establish the crucial 
    {\it bootstrap estimates} of 
     the unstable direction $a_k^+$, 
     and 
	the modulation parameters 
    $(\alpha_k, \theta_k)$ and the remainder $\ve$ in the geometrical decomposition. 
	The bootstrap estimates are important to derive 
	uniform estimates for approximating solutions 
	up to a universal time, 
	and thus, allow to construct 
	stochastic multi-solitons by using compactness arguments.   
	This constitutes the main technical part
	of Sections \ref{subsec cm} and \ref{sec4}. 
    
	The remainder term in the geometrical decomposition  
    can be controlled by a
	coercivity type estimate of the Lyapunov functional, 
	see Proposition \ref{coe} below.

	The subtleness here 
	is that,  
	unlike in the deterministic case \cite{CMM11}, 
	the presence of noise,
	especially with the polynomial decay rate in {\rm Case (II)},
	affects 
    the decay rates in the bootstrap estimates.  
    
	As a matter of fact,  
	in the deterministic case, 
	because the ground state and eigenfunctions 
	decay exponentially fast at infinity, 
	the exponential decay rate is sufficient 
	to close bootstrap estimates.   
	However, in the stochastic {\rm Case (II)}, 
    the above key quantities, 
    that is, 
	the modulation parameters, the remainder
	and the unstable directions, 
     have merely polynomial decay rates,
	rather than the exponential decay rate in the deterministic case. 
	As a result, 
	the derivation of 
    appropriate decay rates 
	to close bootstrap estimates 
	is much more delicate. 
    Under the a-priori control 
of the unstable direction 
$
    |{a}_k^-(t)|\le \phi^{\frac{1}{2}+\frac{1}{4d}}(\widetilde{\delta}t),
$
we derive the bootstrap estimates 
\begin{align*} 	
        |\alpha_{k}(t)-\alpha_k^0|+|\theta_{k}(t)-\theta_k^0| 
            \le t\phi^{\frac{1}{2}}(\widetilde{\delta}t),
            \quad 
            \Vert \varepsilon(t)\Vert_{H^1}\le \phi^{\frac{1}{2}}(\widetilde{\delta}t),
            \quad 
            |{a}_k^+(t)|\le \phi^{\frac{1}{2}}(\widetilde{\delta}t). 
		\end{align*} 
        
    We note that the above decay rates of bootstrap estimates 
	are dictated by the spatial decay rate of the noise.   
   More technically,  
	the $t^{-1}$ decay rate of 
	the tail of the noise $B_*(t)$ in \eqref{B*-bdd} below is essential to close the bootstrap estimate of $\|\varepsilon\|_{H^1}$ in {\rm Case (II)}, 
	see, e.g., estimate \eqref{5.24} below.  
These facts reflect 
    the effect of noise on the  soliton dynamics.

	It is also worth noting that,
	because of the possible singularity at the origin of
	the second derivative of the supercritical nonlinearity  
	in high dimensions,
	the extra unstable directions $\{a_k^\pm\}$
	are controlled by $\|\ve\|_{H^1}^{p\wedge 2}$,
	together with the noise decay rate and the negligible exponential decay rate.
	See Proposition \ref{prop3.6} below.

	\medskip
	\paragraph{\bf $\bullet$ Topological arguments:} 

    In general, 
    with the help of bootstrap estimates,
	one can obtain  uniform estimates 
	of approximating solutions 
	by using standard continuity argument. 
    See, e.g., \cite{MM06,RSZ23} 
    in the mass-(sub)critical case.

    In contrast to that, 
    due to the unstable direction 
	$a_k^-$ in the supercritical case, 
	the above bootstrap estimates 
	only allow to refine the estimates of the modulation parameters, 
    the remainder 
	and   the unstable direction $a_k^+$, 
    but require a-priori control of 
	the other unstable direction 
	$a_k^-$.

	In order to achieve the required a-priori control of $a_k^-$,  
	we use a topological argument 
	based on the Brouwer fixed point theorem 
	inspired by \cite{CMM11}. 
	Again, one needs to take into account the influence of noise, 
    and the radius of the ball where the topological argument is performed is dictated 
    by the decay rate of the noise. 
	For instance, 
	in {\rm Case (II)},
	the topological arguments are performed on 
    the ball 
$B_{\mathbb{R}^K} (\phi^{\frac{1}{2}+\frac{1}{4d}}(\widetilde{\delta}{n}))$, 
where $\phi$ has polynomial decay 
rate in \eqref{1.17}, 
that is different from the exponential rate in the deterministic case. 
See Proposition \ref{prop5.4} below.

	\medskip
	
	\begin{notation}
		For any $x=(x_1,x_2,\cdots,x_d)\in\mathbb{R}^d$ and any multi-index $\nu=(\nu_1,\nu_2,\cdots,\nu_d)$, let $|x|:=\big(\sum_{l=1}^dx_l^2\big)^{\frac{1}{2}}$, $\partial^{\nu}_x:=\partial^{\nu_1}_{x_1}\cdots\partial^{\nu_d}_{x_d}$ and $|\nu|:=\nu_1+\cdots+\nu_d$.
		For $s\in\mathbb{R}$, $1\le p\le+\infty$, let $W^{s,p}(\mathbb{R}^d)$ denote the standard Sobolev spaces, and $H^s:=W^{s,2}(\mathbb{R}^d)$. In particular, $L^p:=W^{0,p}(\mathbb{R}^d)$ is the space of $p$-integrable (complex-valued) functions endowed with the norm $\|\cdot\|_{L^p}$.
		When $p=2$, $L^2$ is the Hilbert space endowed with the inner product $\langle v,w\rangle=\int_{\mathbb{R}^d}v(x)\bar{w}(x)dx$.

		Given any two Banach spaces $\mathcal{X}$ and $\mathcal{Y}$ and any
		Fr\'echet differentiable map $F:\mathcal{X}\to\mathcal{Y}$,
		let $dF(x)\in L(\mathcal{X},\mathcal{Y})$
		($dF$ for short)
		denote the Fr\'echet derivative of $F$ at $x$.
		For any $h\in\mathcal{X}$,
		$dF.h$ denotes the directional derivative of $dF$ along the direction $h$.
		Moreover, let $B_{\mathcal{X}}(x,r)$ (resp. $\mathring B_{\mathcal{X}}(x,r)$) denote the closed (resp. open) ball of the Banach space $\mathcal{X}$, centered at $x$ with radius $r>0$,
		and $S_{\mathcal{X}}(x,r)$ the corresponding sphere. If $x=0$, we simply write $B_{\mathcal{X}}(r)$, $\mathring B_{\mathcal{X}}(r)$ and  $S_{\mathcal{X}}(r)$.
		
		Throughout this paper, positive constants $C$, $\eta$ and $\delta$ may change from line to line.
		Finally,  $f=\mathcal{O}(g)$ means that $|f/g|$ stays bounded, and $f=o(g)$ means that $|f/g|$ converges to zero.
	\end{notation}

	\section{Geometrical decomposition} \label{sec2}
	
	This section is devoted to the geometrical decomposition of \eqref{RNLS} with modulated initial data.
	The crucial role here is played by
	the spectrum of linearized Schr\"odinger
	operators around the ground state.

	\subsection{Linearized Schr\"odinger operators}
	Let $\mathscr{L}=(\mathscr{L}_+,\mathscr{L}_-)$ be the linearized operator around the ground state,
	given by
	\begin{equation}
		\mathscr{L}_+:=-\Delta+I-pQ^{p-1}, \quad \mathscr{L}_-:=-\Delta+I-Q^{p-1}
		\label{1.8}
	\end{equation}
	with $p\in(1+\frac{4}{d},1+\frac{4}{(d-2)_+})$. For any complex valued function $f=f_1+if_2\in H^1$, let
	\begin{equation}
		\mathscr{L}f:=-\mathscr{L}_-f_2+i\mathscr{L}_+f_1,
		\label{1.81}
	\end{equation}
	and
	\begin{equation}
		\left(\mathscr{L}f,f\right):=\int f_1\mathscr{L}_+f_1dx+\int f_2\mathscr{L}_-f_2dx.
		\label{1.9}
	\end{equation}
	
	It is known (see, e.g., \cite{G90,S06,W85}) that  the linearized Schr\"odinger operator $\mathscr{L}$
	has exactly one pair of real nonzero eigenvalues $e_0(>0)$ and $-e_0$, with the corresponding normalized eigenfunctions $Y^{\pm}$
	satisfying
	\begin{align} \label{Y-S-decay}
		Y^\pm \in \mathcal{S}(\mathbb{R}^d),
	\end{align}
	$\Vert Y^{\pm}\Vert_{L^2}=1$,
	$\bar{Y}^+=Y^-$ and
	\begin{equation}
		\mathscr{L}Y^\pm=\pm e_0Y^\pm. \label{1.10}
	\end{equation}
	Moreover,
	it  has the following coercivity property,
	which is important to derive the geometrical decomposition
	and uniform estimates of solutions:
	\begin{equation}
		\left(\mathscr{L}f,f\right)\ge C\Vert f\Vert_{H^1}^2-\frac{1}{C}\left(\!\left(\int\nabla Qf_1dx\right)^2\!+\!\left(\int Qf_2dx\right)^2\!+\!\left(\text{Im}\int Y^+\bar{f}dx\right)^2\!+\!\left(\text{Im}\int Y^-\bar{f}dx\!\right)^2\right) 
		\label{6.1}
	\end{equation}
	for any $f=f_1+if_2\in H^1$,
	where $C>0$ is a universal positive constant.
	See, e.g.,  \cite{CMM11, DR10}.
	
	\begin{remark}
		The above coercivity estimate reveals that 
		the linearized operator $\mathscr{L}$ has four unstable directions.
		The first two can be controlled by the orthogonality conditions in the geometrical decomposition
		as in the mass-(sub)critical case \cite{RSZ23},
		see Proposition \ref{prop3.2} below.
		In contrast,
		the latter two unstable directions
		gives rise to the main difficulty in the 
        soliton analysis in the supercritical case.
		In order to control these new unstable directions,  
		we will further modulate the initial data and apply topological arguments.
	\end{remark}

	The following result shows that the  eigenfunctions $Y^\pm$ decay exponentially fast at infinity. 
	
	\begin{lemma} [Exponential decay of eigenfunctions]  \label{Y}
		There exist $C, \delta>0$ such that
		\begin{equation}
			|Y^+(x)|+|Y^-(x)|\le Ce^{-\delta|x|},\ \ x\in\mathbb{R}^d.
			\label{p0}
		\end{equation}
	\end{lemma}

	\begin{proof}
		The proof proceeds in two steps:
		we first prove that
		$Y^\pm$ are exponentially integrable,
		and then we show
		the pointwise exponential decay
		of $Y^\pm$.
		
		$(i)$
		Exponential integrability:
		Let $Y_1:=\text{Re}Y^+$,  $Y_2:=\text{Im}Y^+$. Let us first show that 
		\begin{align}\label{est-int}
			\int(Y_1^2+Y_2^2)e^{|x|}dx<\infty.
		\end{align}

		For this purpose,
		we let $f_{\eta}(x):=e^{\frac{|x|}{1+\eta|x|}}$, $\eta>0$, $x\in\mathbb{R}^d$. It is easy to check that $f_{\eta}$ is bounded, Lipschitz continuous  and satisfies  that 
		\begin{align}  \label{fve}
			|\nabla f_{\eta}(x)|\le f_{\eta}(x),\quad x\neq{\bf 0}.
		\end{align}
		
		By \eqref{1.8}, \eqref{1.81},  \eqref{1.10} and $\bar{Y}^+=Y^-$ one has
		\begin{align}
			\label{p1}&-\Delta Y_2+Y_2-Q^{p-1}Y_2=-e_0Y_1,\\
			\label{p2}&-\Delta Y_1+Y_1-pQ^{p-1}Y_1=e_0Y_2.
		\end{align}
		Taking the inner product of \eqref{p1} and \eqref{p2} with $f_{\eta}Y_2$ and $f_{\eta}Y_1$, respectively, and then summing up the results we obtain
		\begin{equation}
			\int\nabla Y_2\cdot\nabla(f_{\eta}Y_2)dx+\int\nabla Y_1\cdot\nabla(f_{\eta}Y_1)dx+\int f_{\eta}Y_2^2dx+\int f_{\eta}Y_1^2dx=\int Q^{p-1}\big(f_{\eta}Y_2^2+pf_{\eta}Y_1^2\big)dx.\label{p3}
		\end{equation}

		Next let us treat the left-hand side of \eqref{p3}. Using \eqref{fve} and H\"older's inequality one has
		\begin{align}
			\int\nabla Y_2\cdot\nabla(f_{\eta}Y_2)dx=&\int f_{\eta}|\nabla Y_2|^2dx+\int(\nabla f_{\eta}\cdot\nabla Y_2)Y_2dx\nonumber\\
			\ge&\frac{1}{2}\int f_{\eta}|\nabla Y_2|^2dx-\frac{1}{2}\int f_{\eta}Y_2^2dx.   \label{p4}
		\end{align}
		Similarly,  one has
		\begin{equation}
			\int\nabla Y_1\cdot\nabla(f_{\eta}Y_1)dx\ge\frac{1}{2}\int f_{\eta}|\nabla Y_1|^2dx-\frac{1}{2}\int f_{\eta}Y_1^2dx.\label{p5}
		\end{equation}
		Plugging \eqref{p4} and \eqref{p5} into \eqref{p3} we then derive that
		\begin{equation}
			\text{L.H.S. of \eqref{p3}}\ge\frac{1}{2}\int f_{\eta}(Y_1^2+Y_2^2)dx.
			\label{p6}
		\end{equation}
		
		Regarding the right-hand side of \eqref{p3}, by the exponential decay of the ground state \eqref{1.7},
		we see that   $pQ^{p-1}<1/4$
		for any $|x|>C_0$ with $C_0$ large enough.
		This yields that for some $C>0$ independent of $\eta$,
		\begin{align}
			\text{R.H.S. of \eqref{p3}}\le&\frac{1}{4}\int_{|x|>C_0}f_{\eta}(Y_1^2+Y_2^2)dx+\int_{|x|\le C_0}Q^{p-1}(f_{\eta}Y_2^2+pf_{\eta}Y_1^2)dx\nonumber\\
			\le&\frac{1}{4}\int f_{\eta}(Y_1^2+Y_2^2)dx+C.
			\label{p7}
		\end{align}
		
		Thus, combining \eqref{p6} and \eqref{p7} together we get 
		$$\int(Y_1^2+Y_2^2)f_{\eta}dx\le C. $$
		Letting $\eta\to0$ and using Fatou's lemma we obtain \eqref{est-int}.
		
		\medskip
		$(ii)$ Next, we shall prove the 
		pointwise exponential decay
		\begin{equation}\label{decay}
			\sup\limits_{x\in \mathbb{R}^d}(Y_1^{d+2}+Y_2^{d+2})e^{|x|}<\infty.
		\end{equation}
		It follows from the following more general claim.

		\medskip
		\paragraph{\bf Claim:}
		Let $d\ge1$.
		If $g\in\mathcal{S}(\mathbb{R}^d)$ is nonnegative 
		and  there exists $m_0\ge2$ 
        such that  
		\begin{equation}\label{int-beta}
			\int g^{m_0}(x)e^{|x|}dx <\infty,
		\end{equation}
		then 
		\begin{equation}\label{est-beta0+k}
			\sup\limits_{x\in \mathbb{R}^d} 
            g^{m_0+d}(x)e^{|x|} 
            <\infty. 
		\end{equation}
		Applying this claim to the case where
		$m_0=2$ and $g=Y_1$, $Y_2$ in \eqref{est-beta0+k} we thus obtain \eqref{decay}. 
		
		\medskip
		 It remains to prove the claim.
		For this purpose,
		we shall use the induction argument on dimensions.
		First when $d=1$,
		by the continuity of $g(x)$ and $e^{|x|}$,
		\begin{equation}\label{|x|<1}
			\sup\limits_{|x|\le1} 
			g^{m_0+1}(x)e^{|x|} 
			<\infty.
		\end{equation}
		Then, using \eqref{int-beta} and $g\in\mathcal{S}(\mathbb{R})$ we have 
		\begin{equation*}
			\sup\limits_{x>1}\int_1^{x}|(g^{m_0+1}(s)e^s)'|ds<\infty,\ \ {\rm and} \ \ \sup\limits_{y<-1}\int_{y}^{-1}|(g^{m_0+1}(s)e^{-s})'|ds<\infty,
		\end{equation*}
		which, via the fundamental theorem of calculus, yields that 
		\begin{equation}\label{x>1<-1}
			\sup\limits_{x>1}\Big|g^{m_0+1}(x)e^{|x|}-g^{m_0+1}(1)e\Big|<\infty\ \ {\rm and} \ \ \sup\limits_{y<-1}\Big|g^{m_0+1}(y)e^{|y|}-g^{m_0+1}(-1)e\Big|<\infty.
		\end{equation}
		It follows from \eqref{|x|<1} and \eqref{x>1<-1} that \eqref{est-beta0+k} is valid when $d=1$.

		\medskip
		Now, we assume that \eqref{int-beta} implies \eqref{est-beta0+k} for all $d\le k$
		and consider the case where $d=k+1$.
		
        Let $x=(x_1,x_2,\cdots,x_{k+1})\in\mathbb{R}^{k+1}$. Denote the cube $O^c_{k+1}:=[-1, 1]^{k+1}$ and  its complement $O_{k+1}:=\mathbb{R}^{k+1}\backslash O^c_{k+1}$.
		Using the continuity of $g(x)$ and $e^{|x|}$ again we have
		\begin{equation*}
			\sup\limits_{x\in O_{k+1}^c} g^{m_0+k+1}(x)e^{|x|}<\infty.
		\end{equation*}
		Hence, it suffices to prove
		\begin{equation}\label{est-omega-k+1}
			\sup\limits_{x\in O_{k+1}} 
			g^{m_0+k+1}(x)e^{|x|}<\infty.
		\end{equation}
		
		Let \begin{equation*}
			D_j:=\{(l_1,l_2,\cdots,l_j)\in \mathbb{Z}^j:1\le l_1<l_2<\cdots<l_j\le k+1\} \quad \text{for}\quad 1\le j\le k+1.
		\end{equation*}	
        It is obvious that
        $O_{k+1}$ is a union of the sets $O_{k+1}^0$ and $O_{k+1}^{(l_1,\cdots,l_j)}$, $\forall~ (l_1,\cdots,l_j)\in D_j$,
		where 
        \begin{equation*}
            O_{k+1}^0:=\{x\in O_{k+1}:x_m\ge0,1\le m\le k+1\},
        \end{equation*}
        \begin{equation*}
            O_{k+1}^{(l_1,\cdots,l_j)}:=\{x\in O_{k+1}:x_m\ge0,\ m\neq l_1,l_2,\cdots,l_j,\ {\rm and}\ x_{l_1},x_{l_2},\cdots,x_{l_j}\le0\}.
        \end{equation*}

        Thus, \eqref{est-omega-k+1} follows immediately from 
        \begin{equation}\label{est-omega-0}
            \sup\limits_{x\in O_{k+1}^0} 
			g^{m_0+k+1}(x)e^{|x|}<\infty,
        \end{equation}
        and
        \begin{equation}\label{est-omega-1j}
            \sup\limits_{x\in O_{k+1}^{(l_1,\cdots,l_j)}} 
			g^{m_0+k+1}(x)e^{|x|}<\infty,\ \ \forall\ (l_1,\cdots,l_j)\in D_j,\ 1\le j\le k+1.
        \end{equation}
        
        Below, we prove \eqref{est-omega-0}, and the proof of \eqref{est-omega-1j} is similar.
		For any $x\in\mathbb{R}^{k+1}$ and $(l_1,\cdots,l_j)\in D_j$,
		we set
		\begin{align*} 
			& x^{(l_1,\cdots,l_j)}:=(x_1,\cdots,x_{l_1-1},1,x_{l_1+1},\cdots,x_{l_2-1},1,x_{l_2+1},\cdots,x_{l_j-1},1,x_{l_j+1},\cdots,x_{k+1}), \notag \\ 
			& \bar{x} ^{(l_1,\cdots,l_j)}:=(x_1,\cdots,x_{l_1-1},0,x_{l_1+1},\cdots,x_{l_2-1},0,x_{l_2+1},\cdots,x_{l_j-1},0,x_{l_j+1},\cdots,x_{k+1}), \notag \\
			& dx^{(l_1,\cdots,l_j)}:=dx_1\cdots dx_{l_1-1}dx_{l_1+1}\cdots dx_{l_2-1}dx_{l_2+1}\cdots dx_{l_j-1}dx_{l_j+1}\cdots dx_{k+1}.
		\end{align*} 
		Note that 
		\begin{equation}\label{est-l1}
			|x^{(l_1,\cdots,l_j)}|\le|\bar{x}^{(l_1,\cdots,l_j)}|+j,\ \ |\bar{x}^{(l_1,\cdots,l_j)}|\le|x|,\ \ 1\le j\le k+1,
		\end{equation}
		and 
		\begin{equation}\label{est-lj}
		    |\bar{x}^{(l_1,\cdots,l_{j+1})}|\le|\bar{x}^{(l_1,\cdots,l_{j})}|,\ \ 1\le j\le k.
		\end{equation}

		Since $O_{k+1}^0$ does not contain the singular point ${\bf 0}\in\mathbb{R}^{k+1}$, by \eqref{int-beta} and 
        $g\in\mathcal{S}(\mathbb{R}^{k+1})$, there exists $C_{k+1}<\infty$ such that
		\begin{equation*}
				\int_{O_{k+1}^0}\Big|\partial_x^{\nu}\big(g^{m_0+k+1}(x)e^{|x|}\big)\Big|dx\le C_{k+1} 
		\end{equation*} 
		with $\nu=(1,1,\cdots,1)$ being a $(k+1)$-dimensional index. This along with \eqref{est-l1} and the fundamental theorem of calculus yields that for any $x\in O_{k+1}^0$,
		\begin{align}\label{est-I}
			&e^{|x|}g^{m_0+k+1}(x)\nonumber\\
            \le&\sum\limits_{l_1\in D_1}g^{m_0+k+1}(x^{(l_1)})e^{|x^{(l_1)}|}+\cdots+(-1)^{k+1}\sum\limits_{(l_1,\cdots,l_{k+1})\in D_{k+1}}g^{m_0+k+1}(x^{(l_1,\cdots,l_{k+1})})e^{|x^{(l_1,\cdots,l_{k+1})}|}+C_{k+1}\nonumber\\
			\le&\sum\limits_{l_1\in D_1}g^{m_0+k+1}(x^{(l_1)})e^{|\bar{x}^{(l_1)}|+1}+\cdots+\sum\limits_{(l_1,\cdots,l_{k+1})\in D_{k+1}}g^{m_0+k+1}(x^{(l_1,\cdots,l_{k+1})})e^{|\bar{x}^{(l_1,\cdots,l_{k+1})}|+k+1}+C_{k+1}\nonumber\\
			=&:\sum\limits_{j=1}^{k+1} I_j+C_{k+1}.
		\end{align}
		
		In order to obtain \eqref{est-omega-0},
		we need to show that $\sum_{j=1}^{k+1} I_j$ has a universal upper bound for all $x\in O_{k+1}^0$.
		In view of the induction when $d\le k$, we just need to prove that for any $(l_1,\cdots, l_j) \in D_j$   
		with $1\leq j\leq k$, 
		\begin{align}\label{sum-int}
			\int g^{m_0+j}(x^{(l_1,l_2,\cdots,l_j)})e^{|\bar{x}^{(l_1,l_2,\cdots,l_j)}|}dx^{(l_1,\cdots,l_j)}<\infty.
		\end{align}
		
		To this end,
		we proceed by a further induction argument on the index $j$ to obtain \eqref{sum-int}.
		
		First, in the case where $j=1$, by \eqref{int-beta} and \eqref{est-l1}, we have
		\begin{equation}
			\int g^{m_0+1}(x)e^{|\bar{x}^{(l_1)}|}dx
			\leq \int g^{m_0+1}(x)e^{|x|}dx 
			<\infty. 
		\end{equation} 
		In particular, there exists $x_{l_1}^*>1$ such that
		\begin{align} \label{j=0-j=1}
			\int g^{m_0+1}(x^{*,(l_1)})e^{|\bar{x}^{(l_1)}|}dx^{(l_1)}  <\infty, 
		\end{align}
		where $x^{*,(l_1)}=(x_1,\cdots,x_{l_1-1},x_{l_1}^*,x_{l_1+1},\cdots,x_{k+1})$.
		Moreover, using $g\in\mathcal{S}(\mathbb{R}^{k+1})$, \eqref{int-beta} and \eqref{est-l1} once more we obtain 		
			\begin{equation*}
				\Big|\int\partial_{x_{l_1}}\big(g^{m_0+1}(x)\big)e^{|\bar{x}^{(l_1)}|}dx\Big|\leq \int\Big|\partial_{x_{l_1}}\big(g^{m_0+1}(x)\big)\Big|e^{|x|}dx<\infty.
		\end{equation*}
		Then, integrating
		over $x^{(l_1)}$ in $\mathbb{R}^k$ and over the $x_{l_1}$-coordinate from $1$ to $x_{l_1}^*$, 
		we get that there exists $C^{(l_1)}<\infty$ such that
		\begin{equation*}
			\Big|\int g^{m_0+1}
			(x^{*,(l_1)})e^{|\bar{x}^{(l_1)}|}dx^{(l_1)}
			-\int g^{m_0+1}(x^{(l_1)})e^{|\bar{x}^{(l_1)}|}dx^{(l_1)}\Big| 
			\le C^{(l_1)}, 
		\end{equation*}
		which along with \eqref{j=0-j=1} yields that 
        \begin{equation*}
            \int g^{m_0+1}(x^{(l_1)})e^{|\bar{x}^{(l_1)}|}dx^{(l_1)}\le\int g^{m_0+1}(x^{*,(l_1)})e^{|\bar{x}^{(l_1)}|}dx^{(l_1)}+C^{(l_1)}<\infty.
        \end{equation*}
        Thus, \eqref{sum-int} holds when $j=1$.
		
		\medskip 
		Next, we assume that \eqref{sum-int} holds for $1\leq j\leq n\leq k-1$ 
		and consider the case where $j=n+1$.
		According to $g\in\mathcal{S}(\mathbb{R}^{k+1})$, \eqref{est-lj} and the induction hypothesis \eqref{sum-int} 
		when $j=n$, we have
		\begin{align}\label{j=n-j=n+1}
			&\int g^{m_0+n+1}(x^{(l_1,l_2,\cdots,l_n)})e^{|\bar{x}^{(l_1,l_2,\cdots,l_{n+1})}|}dx^{(l_1,l_2,\cdots,l_n)}\nonumber\\
			\leq &C 
			\int g^{m_0+n}(x^{(l_1,l_2,\cdots,l_n)})e^{|\bar{x}^{(l_1,l_2,\cdots,l_n)}|}dx^{(l_1,l_2,\cdots,l_n)}
			<\infty,
		\end{align}
		and 
		\begin{equation} \label{dg-j=n-j=n+1}
			\Big|\int\partial_{x_{l_{n+1}}}\big(g^{m_0+n+1}(x^{(l_1,l_2,\cdots,l_n)})\big)e^{|\bar{x}^{(l_1,l_2,\cdots,l_{n+1})}|}dx^{(l_1,l_2,\cdots,l_n)}\Big| <\infty. 
		\end{equation} 
		In particular, 
        the finiteness in \eqref{j=n-j=n+1} implies that  there exists $x_{l_{n+1}}^*>1$ such that 
		\begin{align} \label{j=n-j=n+1*}
			\int g^{m_0+n+1}(x^{*,(l_1,l_2,\cdots,l_{n+1})})e^{|\bar{x}^{(l_1,l_2,\cdots,l_{n+1})}|}dx^{(l_1,l_2,\cdots,l_{n+1})} <\infty,
		\end{align}
		where $x^{*,(l_1,l_2,\cdots,l_{n+1})}=(x_1,\cdots,x_{l_1-1},1,x_{l_1+1},\cdots,x_{l_{n+1}-1},x_{l_{n+1}}^*,x_{l_{n+1}+1},\cdots,x_{k+1})$.
		Then,  integrating the integrand 
		of \eqref{dg-j=n-j=n+1} 
		for $x^{(l_1,l_2,\cdots,l_{n+1})}$ 
        in $\mathbb{R}^{k-n}$ and for  the $x_{l_{n+1}}$-coordinate 
		from $1$ to $x_{l_{n+1}}^*$ 
		we get that there exists $C^{(l_1,l_2,\cdots,l_{n+1})}<\infty$ such that
		\begin{align}\label{j=n+1-int}
			\Big|&\int g^{m_0+n+1}(x^{*,(l_1,l_2,\cdots,l_{n+1})})e^{|\bar{x}^{(l_1,l_2,\cdots,l_{n+1})}|}dx^{(l_1,l_2,\cdots,l_{n+1})}\nonumber\\
			&-\int g^{m_0+n+1} ( x^{(l_1,l_2,\cdots,l_{n+1})}) e^{|\bar{x}^{(l_1,l_2,\cdots,l_{n+1})}|}dx^{(l_1,l_2,\cdots,l_{n+1})}\Big|
			\le C^{(l_1,l_2,\cdots,l_{n+1})}. 
		\end{align} 
		Hence, it follows from \eqref{j=n-j=n+1*} and \eqref{j=n+1-int} that 
        \begin{align*}
            &\int g^{m_0+n+1} (x^{(l_1,l_2,\cdots,l_{n+1})}) e^{|\bar{x}^{(l_1,l_2,\cdots,l_{n+1})}|}dx^{(l_1,l_2,\cdots,l_{n+1})}\\
            \le&\int g^{m_0+n+1}(x^{*,(l_1,l_2,\cdots,l_{n+1})})e^{|\bar{x}^{(l_1,l_2,\cdots,l_{n+1})}|}dx^{(l_1,l_2,\cdots,l_{n+1})}+C^{(l_1,l_2,\cdots,l_{n+1})}<\infty,
        \end{align*}
        which yields \eqref{sum-int} in the case where $j=n+1$.

		\medskip 
		Thus, by induction over 
		$j$, 
		we obtain  \eqref{sum-int} for any 
		$j \in \{1,\cdots, k\}$. 
		This in turn implies that $\sum_{j=1}^k I_j$ in \eqref{est-I} is uniformly bounded,
		and so, \eqref{est-omega-0} follows.
        Analogous arguments also lead to 
        \eqref{est-omega-1j} for the remaining regimes $O_{k+1}^{(l_1,\cdots,l_j)}$, $(l_1,\cdots,l_j)\in D_j$, $1\le j\le k+1$, 
        and thus yield \eqref{est-beta0+k} for $d=k+1$. 
        Therefore, 
        using the induction argument again,
        we prove the claim for any $d\ge1$ 
        and finish the proof.
	\end{proof}

	\subsection{Geometrical decomposition}\label{geo}

	In this subsection we derive the geometrical decomposition of solutions to the rescaled random NLS with modulated final data
	\begin{equation}
		\begin{cases}\text{$i$}\partial_tu+(\Delta+b_*\cdot\nabla+c_*)u+|u|^{p-1}u=0,
			\\[5pt]u(T)=R(T)+\text{$i$}\sum\limits_{k,\pm}b_k^{\pm}Y_k^{\pm}(T).
		\end{cases}
		\label{3.1}
	\end{equation}
	Here, $p\in (1+4/d, 1+4/(d-2)_+)$,  $T>0$ is sufficiently large,
	$b_*$ and $c_*$ are the random coefficients
	given by \eqref{1.30} and \eqref{1.31}, respectively.
	Moreover,
	\begin{equation}\label{sumRk}
		R=\sum_{k=1}^{K}R_k
	\end{equation}
	with the soliton profiles $R_k$ given by \eqref{1.11},
	$Y^\pm_k$ are the rescaled eigenfunctions
	of the form 
	\begin{equation}
		Y_{k}^\pm(t,x)=Y_{w_k}^\pm(y_k)e^{i(\frac{1}{2}v_k\cdot x-\frac{1}{4}|v_k|^2t+(w_k)^{-2}t+\theta_k^0)},
		\label{3.2}
	\end{equation}
	with
	\begin{equation}\label{def-Ywk}
		Y_{w_k}^\pm(y_k)=(w_k)^{-\frac{2}{p-1}}Y^\pm\Big(\frac{y_k}{{w_k}}\Big),\quad{\rm and}\quad y_k = x-v_kt-\alpha_k^0,
	\end{equation}
	and $Y^{\pm}$ defined in \eqref{1.10}. 
    We also denote the vector of the perturbation parameters in the initial data by
	\begin{align}\label{def-b}
	    \boldsymbol{b}:=(b_1^+,\cdots,b_K^+,b_1^-, \cdots,b_K^-)\in\mathbb{R}^{2K}.
	\end{align}

	\begin{remark}\label{rmk31}
		We note that, 
		unlike equation (3.1) of \cite{RSZ23}, 
		the rescaled random NLS \eqref{3.1} 
		has the additional modulated term $i\sum_{k,\pm}b_k^{\pm}Y_k^{\pm}(T)$ in the final condition. 
		It is introduced mainly to control the extra unstable directions of the linearized 
        Schr\"odinger operator $\mathscr{L}$ in the mass-supercritical case.
	\end{remark}

	For convenience, we set $\mathcal{P}_k:=(\alpha_k,\theta_k)\in\mathbb{Y}:=\mathbb{R}^d\times\mathbb{R}$, $1\le k\le K$, and $\mathcal{P}:=(\mathcal{P}_1,\mathcal{P}_2,\cdots,\mathcal{P}_K)\in\mathbb{Y}^K$.

	\begin{proposition}[Geometrical decomposition]\label{prop3.2}
		Let $u$ be a local solution to  equation \eqref{3.1}.
		Then, there exist deterministic constants $M$, $\eta>0$, such that the following holds: 
        
		For any $T\ge M$ and any $\boldsymbol{b}\in B_{\mathbb{R}^{2K}}(\eta)$, there exist $T^*\in (0,T)$ and unique modulation parameters $\mathcal{P}\in{C}^1\left([T^*,T];\mathbb{Y}^K\right)$, such that $u$ admits the unique geometrical decomposition
		\begin{equation}
			u(t,x)=\sum\limits_{k=1}^{K}\widetilde{R}_k(t,x)+\varepsilon(t,x)\ (=:\widetilde{R}(t,x)+\varepsilon(t,x)),
			\label{3.3}
		\end{equation}
		with the modulated soliton profile given by 
		\begin{equation}
			\widetilde{R}_k(t,x):=Q_{w_k}\left(x-v_kt-\alpha_k(t)\right)e^{\text{$i$}\left(\frac{1}{2}v_k\cdot{x}-\frac{1}{4}|v_k|^2t+(w_k)^{-2}t+\theta_k(t)\right)},
			\label{3.4}
		\end{equation}
		and the following orthogonality conditions hold on $[T^*,T]:$
		\begin{equation}
			{\rm Re} \int\nabla\widetilde{R}_k(t)\bar{\varepsilon}(t)dx=0,
			\quad
			{\rm Im} \int\widetilde{R}_k(t)\bar{\varepsilon}(t)dx=0,
			\quad \forall1\le{k}\le{K}.
			\label{3.5}
		\end{equation}
		Moreover, the value of modulation parameters $(\alpha_k,\theta_k)$ and remainder $\varepsilon$ at time $T$ satisfy
		\begin{equation}
			\Vert\varepsilon(T)\Vert_{H^1}+\sum\limits_{k=1}^{K}\left(|\alpha_k(T)-\alpha_k^0|+|\theta_k(T)-\theta_k^0|\right)\le C|\boldsymbol{b}|,
			\label{time T}
		\end{equation}
		for some deterministic positive constant $C$ independent of $T$.
	\end{proposition}

	\begin{remark}\label{rmk3.3}
		(i)	The orthogonality conditions in (\ref{3.5}) allow to control the first two unstable directions 
        arising from the coercive property \eqref{6.1} of the linearized Schr\"odinger operator.
		The remaining two unstable directions will be controlled in Proposition \ref{prop3.5} and Proposition \ref{prop3.6} below by selecting an appropriate vector $\boldsymbol{b}$ in the final condition of equation \eqref{3.1}.
		
		(ii) It is worth noting that the control of $\varepsilon(T), \alpha_k(T)$ and $\theta_k(T)$ in \eqref{time T} is used to obtain the coercivity estimate of $\|\varepsilon(t)\|_{H^1}$ in Proposition \ref{coe},
		as well as the bootstrap estimates of $\sum_{k=1}^K(|\alpha_k(t)-\alpha_k^0|+|\theta_k(t)-\theta_k^0|)$ in Proposition \ref{prop5.2}.
		Moreover,
		estimate \eqref{time T} is also used to
		derive the a-priori estimates \eqref{3.8} and \eqref{3.9} on $[T^*,T]$,
		which guarantee that all constants  
        appearing in the estimates of modulation equations and functionals are deterministic and uniformly bounded.
	\end{remark}

	In order to prove Proposition \ref{prop3.2}, 
	let us first present  the following result. 
    It can be proved in an  analogous manner as in the proof of \cite[Lemma 6.4]{RSZ23}.
	Thus, the proof is omitted here.
	
	\begin{lemma}\label{lem6.3}
		Given any $L,w_k\in\mathbb{R}^+$, $\alpha_k^0$, $v_k\in\mathbb{R}^d$, $\theta_k^0\in\mathbb{R}$, $1\le k\le K$. Set
		\begin{equation}
			R_L(x):=\sum\limits_{k=1}^{K}R_{k,L}(x)=\sum\limits_{k=1}^{K}Q_{w_k}(x-v_kL-\alpha_k^0)e^{i(\frac{1}{2}v_k\cdot x-\frac{1}{4}|v_k|^2L+\left(w_k\right)^{-2}L+\theta_k^0)}.
			\label{6.45}
		\end{equation}
		Then, there exists a deterministic small constant $\delta_*>0$ such that the following holds:
		
		For any $0<r,L^{-1}<\delta_*$ and for any $u\in H^1(\mathbb{R}^d)$ satisfying $\Vert u-R_L\Vert_{H^1}\le r$, there exist a unique $C^1$ function $\mathcal{P}(u)=(\alpha,\theta)$: $H^1\to \mathbb{Y}^K$ with $\alpha:=(\alpha_1,\alpha_2,\cdots,\alpha_K)$ 
		and $\theta:=(\theta_1,\theta_2,\cdots,\theta_K)$,
		such that u admits the decomposition
		\begin{align}
			u=&\sum\limits_{k=1}^{K}Q_{w_k}(x-v_kL-\alpha_k)e^{i(\frac{1}{2}v_k\cdot x-\frac{1}{4}|v_k|^2L+\left(w_k\right)^{-2}L+\theta_k)}+\varepsilon_L\nonumber\\
			=&:\sum\limits_{k=1}^{K}\widetilde{R}_{k,L}+\varepsilon_L,
			\label{6.46}
		\end{align}
		and $\widetilde{R}_{k,L}$, 
        $\varepsilon_L$ satisfy the  orthogonality conditions:
		\begin{equation}
			\text{Re} \int\nabla\widetilde{R}_{k,L} \bar\varepsilon_Ldx={\bf 0},\quad \text{Im}\int\widetilde{R}_{k,L}\bar\varepsilon_Ldx=0, \ \ 1\le k\le K.
			\label{6.47}
		\end{equation}
		Moreover, there exists a deterministic constant $C>0$ such that
		\begin{equation}
			\Vert\varepsilon_L\Vert_{H^1}+\sum\limits_{k=1}^{K}\left(|\alpha_k-\alpha_k^0|+|\theta_k-\theta_k^0|\right)\le C\Vert u-R_L\Vert_{H^1}.
			\label{6.48}
		\end{equation}
	\end{lemma}

	Now, Proposition \ref{prop3.2} 
	follows easily from Lemma \ref{lem6.3}. 
	
	\begin{proof}[Proof of Proposition \ref{prop3.2}]
	   Let $\delta_*$ be as in Lemma \ref{lem6.3} and $M=2 \delta_*^{-1}$.
		For any $T\ge M$,
		using \eqref{Y-S-decay} and the explicit expression \eqref{3.2} 
        we estimate that for every $1\le k\le K$,
		\begin{equation}\label{bound-Yk}
			\|Y_k^\pm(T)\|_{H^1}\le C(\|Y^\pm\|_{L^2}+\|\nabla Y^\pm\|_{L^2})\le C,
		\end{equation}
		which yields that for any $\boldsymbol{b}\in B_{\mathbb{R}^{2K}}(\eta)$ 
        with $\eta$ sufficiently small, 
		$\|u(T)-R(T)\|_{H^1}\le \delta_*/2$.
		Then, by the local well posedness theory there exists $T^*(\geq \delta_*^{-1})$
		close to $T$,
		such that $u(t)\in  C([T^*,T];H^1)$ and $\|u(t)-R(t)\|_{H^1}\le\delta_*$ for any $t\in[T^*,T]$.
		
		Thus, by virtue of Lemma \ref{lem6.3},
		there exist $C^1$ functions $(\alpha_k(t),\theta_k(t))\in C^1([T^*,T];\mathbb{Y})$, $1\le k\le K$,
		such that for any $t\in [T^*,T]$, $u(t)$ admits the decomposition (\ref{6.46}), 
        and the orthogonality conditions in (\ref{6.47}) hold with $t$ replacing $L$, which verify \eqref{3.4} and \eqref{3.5}.
		At last, \eqref{time T} follows from \eqref{6.48} and \eqref{bound-Yk}.
	\end{proof}

	\subsection{Modulated final data}\label{mfd}
	Let $(\alpha_k,\theta_k)\in C^1([T^*,T];\mathbb{Y})$ and $\varepsilon$ be remainder from Proposition \ref{prop3.2}. Set
	\begin{equation}
		\widetilde{Y}_k^\pm(t,x):=Y_{w_k}^\pm(y_k(t,x))e^{i\Phi_k(t,x)},  \label{3.49}
	\end{equation}
	with
	\begin{align}\label{y2.50}
		y_k(t,x) := x-v_kt-\alpha_k(t)
	\end{align}
	and
	the phase function
	\begin{equation}
		\Phi_k(t,x):=\frac{1}{2}v_k\cdot{x}-\frac{1}{4}|v_k|^2t+(w_k)^{-2}t+\theta_k(t).
		\label{3.12}
	\end{equation}
	Let
	\begin{equation}  \mathbf{a}^\pm(t):=\left(a_k^\pm(t)\right)_{1\le k\le K}
		\quad \text{with}\ \
		a_k^\pm(t):={\rm Im}\int\widetilde{Y}_k^{\pm}(t,x)\bar{\varepsilon}(t,x)dx.
		\label{3.50}
	\end{equation}

 The following result permits to steer the unstable 
 directions to prescribed vectors 
 at the final time. 
	
	\begin{proposition}[Modulated final data]\label{prop3.5} 
	There exist deterministic positive constants $M$, 
    $r>0$,   
	such that for any $T\ge M$ 
    and any $\boldsymbol{a}^-\in{B}_{\mathbb{R}^K} (r)$, 
        there exists a unique $\boldsymbol{b}\in B_{\mathbb{R}^{2K}}(\eta)$,  
        where $\eta$ is a deterministic positive constant depending on $r$, 
        such that 
        $\boldsymbol{b}$
        depends 
        continuously on $\boldsymbol{a}^-$,  
		$|\boldsymbol{b}|\le{C}|\boldsymbol{a}^-|$  
		and 
		\begin{equation}
			\mathbf{a}^+(T)={\bf 0} \quad
			{\rm and}\ \ \mathbf{a}^-(T)=\boldsymbol{a}^-, 
			\label{3.52}
		\end{equation} 
    where $C$ is a deterministic positive constant.
	\end{proposition}

	\medskip
	Before proving Proposition \ref{prop3.5}, let us first introduce the following decoupling lemma 
    that helps to decouple different soliton profiles and eigenfunctions $Y^\pm$,
	thanks to the exponential decay results
	\eqref{1.7} and Lemma \ref{Y}.
	Its proof follows in an analogous manner as in the proof of Lemma 6.3 in \cite{RSZ23}.

	\begin{lemma}[Decoupling lemma]\label{lem6.2}
		Let $\delta_0$ be the minimum of the positive constants $\delta$ in
		\eqref{1.7} and \eqref{p0}.
		Assume that $g_i\in C^2_b$, 
		$i=1,2$, satisfy 
		\begin{equation}
			|g_i(x)|\le C_1e^{-\delta_0|x|},~~x\in \mathbb{R}^d,~~i=1,2
			\label{6.38}
		\end{equation}
		for some positive constant $C_1$.
		For every $1\le k\le K$, let
		\begin{equation}
			G_{i,k}(t,x):=\left(w_k\right)^{-\frac{2}{p-1}}g_i\left(\frac{x-v_kt-\alpha_k}{w_k}\right),~~i=1,2,
			\label{6.36}
		\end{equation}
		where $p\in{\big(1+\frac{4}{d},1+\frac{4}{(d-2)_+}\big)}$, $v_j \neq v_k$ 
        for any $j\neq k$, 
        and the parameters $w_k\in\mathbb{R}^+,v_k,\alpha_k\in \mathbb{R}^d$ satisfy
		\begin{equation}
			\left(w_k\right)^{-1}+w_k+|v_k|+|\alpha_k|\le C_2,
			\label{6.37}
		\end{equation}
		for some positive constant $C_2$.
		
		Then, we have that for any $j\neq k$ and $p_1$, $p_2>0$,
		\begin{equation}
			\int|G_{1,j}(t,x)|^{p_1}|G_{2,k}(t,x)|^{p_2}dx\le Ce^{-\delta_2t},
			\label{6.39}
		\end{equation}
		where $C$ and $\delta_2(>0)$ depend on $\delta_0$, $C_1$, $C_2$, $p_1$ and $p_2$.
	\end{lemma}	
	
	\medskip
	Now, we come to the proof of Proposition \ref{prop3.5}.
	\begin{proof}[Proof of Proposition \ref{prop3.5}]
		Define the map
		\begin{equation*}
			G_1:\mathbb{R}^{2K}\mapsto H^1  \  \ {\rm by} \ \ G_1(\boldsymbol{b})=i\sum_{k,\pm}b_k^{\pm}Y_k^{\pm}, 
		\end{equation*}
		where $\boldsymbol{b}$ is given by \eqref{def-b} 
		and $Y_k^\pm$ are the rescaled eigenfunctions given by \eqref{3.2}.
		
		Moreover,
		in view of Lemma \ref{lem6.3},
		for any $v\in B_{H^1}(\delta_*)$,
		there exist unique 
		$\alpha(v)=(\alpha_1(v),\alpha_2(v),\cdots,\alpha_K(v))$ $ \in (\mathbb{R}^d)^K$ 
		and 
		$\theta(v)=(\theta_1(v),\theta_2(v),\cdots,\theta_K(v)) \in \mathbb{R}^K$, 
		such that 
		the geometrical decomposition \eqref{6.46}
		and the  orthogonal conditions \eqref{6.47} hold, 
		where 
		$u$ is replaced by $v+R(T)$, 
		and 
		$R(T)$ is as in \eqref{6.45} with $L$ replaced by $T$. 
		Then, we define the second map
		\begin{align*}
			G_2: 
			B_{H^1}(\delta_*) \mapsto H^1\times (\mathbb{R}^d)^K\times\mathbb{R}^K,
		\end{align*}
		by
		\begin{align} \label{G2-def}
			G_2(v)=(\varepsilon,\alpha,\theta)
			:=\Bigg(v+R(T)-\sum\limits_{k=1}^KQ_{w_k}
			(y_k(\alpha(v))) 
			e^{i\Phi_k(\theta(v))},\alpha(v), \theta(v)\Bigg),
		\end{align}
		where
		\begin{align} \label{phijv-def}
			y_k(\alpha(v))=x-v_kT-\alpha_k(v),\quad
            \Phi_k(\theta(v)) = \frac{1}{2}v_k\cdot{x}-\frac{1}{4}|v_k|^2T+(w_k)^{-2}T+\theta_k(v). 
		\end{align} 
		
		Let us further define the third map
		\begin{align*}
			G_3: H^1\times (\mathbb{R}^d)^K\times\mathbb{R}^K
			\mapsto \mathbb{R}^{2K},
		\end{align*}
		by
		\begin{align}  \label{G3-def}
			G_3(\varepsilon,\alpha,\theta):
			=\Big(\text{Im}\int\widetilde{Y}_k^{\pm}\bar{\varepsilon}dx\Big)_{1\le k\le K,\pm},
		\end{align}
        where $\widetilde{Y}_k^\pm$ is as in \eqref{3.49} with $y_k(t,x)$ and $\Phi_k(t,x)$ replaced by $y_k(\alpha)$ and $\Phi_k(\theta)$, respectively.
		For simplicity, we drop the parameter $T$ in the following arguments.
		
		\medskip
		{\bf Claim:} $G=G_3\circ G_2\circ G_1$ is a diffeomorphism in a
		small neighborhood of ${\bf 0}\in\mathbb{R}^{2K}$. 
		
        \medskip
        
In order to prove this claim, 
by the chain rule, 
we shall compute the derivatives $dG_i$, 
$i=1,2,3$, 
separately in the following. 
Before that, let us note that 
since by \eqref{bound-Yk}, 
        $\|Y_k^\pm\|_{H^1}$ is bounded,
		there exists a deterministic constant $\eta'>0$ 
		such that for any $\boldsymbol{b}\in B_{\mathbb{R}^{2K}}(\eta')$, $\|G_1(\boldsymbol{b})\|_{H^1}\le\delta_*$. 
        Hence, 
        the map $G_2\circ G_1$ is well-defined on $B_{\mathbb{R}^{2K}}(\eta')$.

        \medskip 
        $(i)$ $dG_1$:
		We compute that 
		\begin{equation*}
			dG_1=(\text{$i$}Y_1^+,\text{$i$}Y_2^+,\cdots,\text{$i$}Y_K^+,\text{$i$}Y_1^-,\text{$i$}Y_2^-,\cdots,\text{$i$}Y_K^-),
		\end{equation*}
		which yields that for any $\boldsymbol{b}\in\mathbb{R}^{2K}$,
		\begin{equation}
			dG_1.\boldsymbol{b}=\text{$i$}\sum\limits_{k,\pm}b_k^{\pm}Y_k^{\pm}.
			\label{3.54}
		\end{equation}

		$(ii)$ $dG_2$: 
		Next, let us consider the second map
		$G_2$ given by \eqref{G2-def}.
		Set $F_1^j:=(f_{1}^j,f_{2}^j,\cdots,f_{d}^j)^\top$
		with 
		$f_{l}^j (v, \alpha, \theta)
		:=\text{Re}\int\partial_l\widetilde{R}_j\bar{\varepsilon}dx$, 
		and 
        $F_2^j(v, \alpha, \theta):=\text{Im}\int\widetilde{R}_j\bar{\varepsilon}dx$, $1\le l\le d$, $1\le j\le K$, 
		where $\widetilde{R}_j$ 
		is as in \eqref{3.4} 
		with $t, \alpha_j(t), \theta_j(t)$ 
		replaced by $T$, $\alpha_j$, $\theta_j$, 
		respectively,   
		$1\leq j\leq K$.

		By straightforward computations,
		we have that for any $h\in H^1$,
		\begin{align}
			dG_2.h
			=\Bigg(h+\sum\limits_{k=1}^{K}\sum\limits_{l=1}^{d}(\partial_lQ_{w_k})e^{i\Phi_k(\theta)}\Big(\frac{d\alpha_{k,l}}{dv}.h\Big)-i\sum\limits_{k=1}^{K}\widetilde{R}_k\Big(\frac{d\theta_k}{dv}.h\Big),\Big(\frac{d\alpha_k}{dv}.h\Big)_{1\le k\le K},\Big(\frac{d\theta_k}{dv}.h\Big)_{1\le k\le K}\Bigg), 
			\label{dg}
		\end{align}
		where 
		$\frac{d\alpha_{k}}{dv}.h:=(\frac{d\alpha_{k,1}}{dv}.h,\frac{d\alpha_{k,2}}{dv}.h,\cdots,\frac{d\alpha_{k,d}}{dv}.h)^\top$. 
		
		In order to compute the directional derivatives  $\frac{d\alpha_{k,l}}{dv}.h$ and $\frac{d\theta_k}{dv}.h$, 
		we let $\frac{\partial F_1^j}{\partial\alpha_k}$, $\frac{\partial F_1^j}{\partial\theta_k}$ and $\frac{\partial F_2^j}{\partial\alpha_k}$ denote the following three Jacobian matrices, respectively, 
		\begin{equation}
			\label{3.58}\frac{\partial F_1^j}{\partial\alpha_k}:=
			\begin{bmatrix}
				\frac{\partial f_{1}^j}{\partial\alpha_{k,1}} & \frac{\partial f_{1}^j}{\partial\alpha_{k,2}} & \cdots & \frac{\partial f_{1}^j}{\partial\alpha_{k,d}}\\[10pt]
				\frac{\partial f_{2}^j}{\partial\alpha_{k,1}} & \frac{\partial f_{2}^j}{\partial\alpha_{k,2}} & \cdots & \frac{\partial f_{2}^j}{\partial\alpha_{k,d}}\\
				\vdots & \vdots & \ddots & \vdots\\
				\frac{\partial f_{d}^j}{\partial\alpha_{k,1}} & \frac{\partial f_{d}^j}{\partial\alpha_{k,2}} & \cdots & \frac{\partial f_{d}^j}{\partial\alpha_{k,d}}
			\end{bmatrix},\quad
			\frac{\partial F_1^j}{\partial\theta_k}:=
			\begin{bmatrix}
				\frac{\partial f_{1}^j}{\partial\theta_{k}}\\[10pt]
				\frac{\partial f_{2}^j}{\partial\theta_{k}}\\
				\vdots\\
				\frac{\partial f_{d}^j}{\partial\theta_{k}}
			\end{bmatrix},\quad
			\frac{\partial F_2^j}{\partial\alpha_k}:=
			\begin{bmatrix}
				\frac{\partial F_2^j}{\partial\alpha_{k,1}} & \frac{\partial F_2^j}{\partial\alpha_{k,2}} & \cdots & \frac{\partial F_2^j}{\partial\alpha_{k,d}}
			\end{bmatrix}.
		\end{equation}
		By straightforward computations and
		the decoupling Lemma \ref{lem6.2},
		\begin{equation}
			\label{6.57}
			\begin{aligned}
				&\frac{\partial f_{l}^j}{\partial\alpha_{j,l}}=\left(w_j\right)^{-2}\Vert\partial_lQ_{w_j}\Vert_{L^2}^2+\mathcal{O}\left(\Vert\varepsilon\Vert_{L^{2}}\right),~~\frac{\partial f_{l}^j}{\partial\theta_j}=-\frac{v_{j,l}}{2}\Vert Q_{w_j}\Vert_{L^2}^2+\mathcal{O}\left(\Vert\varepsilon\Vert_{L^{2}}\right), \\
				&\frac{\partial F_2^j}{\partial\theta_j}=\Vert Q_{w_j}\Vert_{L^2}^2+\mathcal{O}\left(\Vert\varepsilon\Vert_{L^{2}}\right),\quad1\le l\le d,\quad1\le j\le K. 
			\end{aligned}
		\end{equation}
		The other terms in matrices \eqref{3.58} are of small order $\mathcal{O}(\|\varepsilon\|_{L^2}+e^{-\delta_2T})$.
		
		In view of the orthogonality conditions in \eqref{6.47},
		we have
		\begin{equation*} F_1^j(v,\alpha(v),\theta(v))={\bf 0},~~F_2^j(v,\alpha(v),\theta(v))=0.
		\end{equation*}
		Then, differentiating the above identities with respect to $v$ we get
		\begin{equation}
			\begin{aligned}
				&  \partial_vF_1^j+
				\sum\limits_{k=1}^K
				\frac{\partial F_1^j}{\partial\alpha_k}
				\left(\frac{d\alpha_k}{dv}\right)
				+\sum\limits_{k=1}^K
				\frac{\partial F_1^j}{\partial\theta_k}\frac{d\theta_k}{dv}
				={\bf 0}, \\
				&  \partial_vF_2^j
				+ \sum\limits_{k=1}^K
				\frac{\partial F_2^j}{\partial\alpha_k}\cdot\frac{d\alpha_k}{dv}
				+ \sum\limits_{k=1}^K
				\frac{\partial F_2^j}{\partial\theta_k}\frac{d\theta_k}{dv}=0,
			\end{aligned}
			\label{F1F2}
		\end{equation}
		where
		$ \frac{\partial F_1^j}{\partial\alpha_k}
		(\frac{d\alpha_k}{dv} ) $
		means that the Jacobian matrix
		$\frac{\partial F_1^j}{\partial\alpha_k}$ acts on the vector
		$\frac{d\alpha_k}{dv}$,
		and $ \frac{\partial F_2^j}{\partial\alpha_k}\cdot\frac{d\alpha_k}{dv}$
		denotes the inner product between two vectors
		$ \frac{\partial F_2^j}{\partial\alpha_k}$ and $\frac{d\alpha_k}{dv}$.
		
		Plugging \eqref{3.58} and \eqref{6.57} into \eqref{F1F2} and using \eqref{6.48} we compute that for any $h\in H^1$,
		\begin{equation}
			\begin{aligned}
				& \frac{d\alpha_{k,l}}{dv}.h =-(w_k)^2\Vert\partial_lQ_{w_k}\Vert_{L^2}^{-2}\Big(\frac{1}{2}v_{k,l}\text{Im}\int\widetilde{R}_k\bar{h}dx+\text{Re}\int\partial_l\widetilde{R}_k\bar{h}dx\Big)+\mathcal{O}\big(\|h\|_{L^2}(\|v\|_{H^1}+e^{-\delta_2T})\big),\\
				&  \frac{d\theta_k}{dv}.h =-\Vert Q_{w_k}\Vert_{L^2}^{-2}\text{Im}\int\widetilde{R}_k\bar{h}dx+\mathcal{O}\big(\|h\|_{L^2}(\|v\|_{H^1}+e^{-\delta_2T})\big),
				\quad 1\le l\le d,\ \ 1\le k\le K.
			\end{aligned}\label{.h}
		\end{equation}
		
		\medskip
        $(iii)$ $dG_3$: 
		Regarding the third map $G_3$
		given by \eqref{G3-def},
		we compute that for any $h\in H^1$, $\beta=(\beta_{j})_{1\le j\le K}\in\mathbb{R}^{d\times K}$ with $\beta_j\in\mathbb{R}^d$, $\gamma\in\mathbb{R}^K$, 
		\begin{equation*}
			\big(dG_3(\varepsilon,\alpha,\theta).(h,\beta,\gamma)\big)_{k,\pm}=\text{Im}\int\widetilde{Y}_k^{\pm}\bar hdx+\sum\limits_{j=1}^{K}\sum\limits_{l=1}^{d}\beta_{j,l}\partial_{\alpha_{j,l}}\text{Im}\int\widetilde{Y}_k^{\pm}\bar{\varepsilon}dx+\sum\limits_{j=1}^{K}\gamma_j\partial_{\theta_j}\text{Im}\int\widetilde{Y}_k^{\pm}\bar{\varepsilon}dx.
		\end{equation*}
		Note that
		\begin{align*}
			&\partial_{\alpha_{j,l}}\text{Im}\int\widetilde{Y}_k^{\pm}\bar{\varepsilon}dx=-\delta_{jk}~\text{Im}\int(w_k)^{-\frac{p+1}{p-1}}(\partial_l Y^\pm)\Big(\frac{y_k(\alpha)}{w_k}\Big)e^{i\Phi_k(\theta)}\bar{\varepsilon}dx,\\
			&\partial_{\theta_j}\text{Im}\int\widetilde{Y}_k^{\pm}\bar{\varepsilon}dx=\delta_{jk}~\text{Re}\int\widetilde{Y}_k^{\pm}\bar{\varepsilon}dx,~~1\le j\le K,~l\le l\le d.
		\end{align*}
		Thus,
		\begin{align}
			\big(dG_3(\varepsilon,\alpha,\theta).(h,\beta,\gamma)\big)_{k,\pm}=\text{Im}\int\widetilde{Y}_k^{\pm}\bar{h}dx-\beta_k\cdot\text{Im}\int(w_k)^{-\frac{p+1}{p-1}}(\nabla Y^{\pm})\Big(\frac{y_k(\alpha)}{w_k}\Big)e^{i\Phi_k(\theta)}\bar{\varepsilon}dx+\gamma_k\text{Re}\int\widetilde{Y}_k^{\pm}\bar{\varepsilon}dx.
			\label{3.57}
		\end{align}

        $(iv)$ Non-degeneracy of $dG$: 
		Now, for any $\boldsymbol{c}=\big(c_1^+,\cdots,c_K^+,c_1^-,\cdots,c_K^-\big)\in\mathbb{R}^{2K}$, let $v=G_1(\boldsymbol{b})\in B_{H^1}(\delta_*) \subseteq H^1$, $h=G_1(\boldsymbol{c})\in H^1$.
		Applying the chain rule we have
		\begin{equation} \label{dG}
			dG(\boldsymbol{b}).\boldsymbol{c}
			=dG_3\big(G_2(G_1(\boldsymbol{b}))\big).\big(dG_2(G_1(\boldsymbol{b})).(dG_1.\boldsymbol{c})\big).
		\end{equation}
		By \eqref{dg}, we have
		\begin{align}
			dG_2(G_1(\boldsymbol{b})).G_1(\boldsymbol{c})=
			\Bigg(&G_1(\boldsymbol{c})+\sum\limits_{k=1}^{K}\sum\limits_{l=1}^{d}\left(\partial_lQ_{w_k}\right)e^{i\Phi_k(\theta(G_1(\boldsymbol{b})))}\Big(\frac{d\alpha_{k,l}}{dv}.G_1(\boldsymbol{c})\Big)-i\sum\limits_{k=1}^{K}\widetilde{R}_k\Big(\frac{d\theta_k}{dv}.G_1(\boldsymbol{c})\Big),\nonumber\\[5pt]
			&~~\Big(\frac{d\alpha_k}{dv}.G_1(\boldsymbol{c})\Big)_{1\le k\le K},\Big(\frac{d\theta_k}{dv}.G_1(\boldsymbol{c})\Big)_{1\le k\le K}\Bigg).\label{dG2G1}
		\end{align}
		Set $Y_1=\text{Re}Y^+$, $Y_2=\text{Im}Y^+$. 
		Note that, by \eqref{6.48}, Lemma \ref{lem6.2} and the algebraic identity
		\begin{equation}
			\text{Re}\int QY^{\pm}dx=\int QY_1dx=-e_0^{-1}\int Q(L_-Y_2)dx=-e_0^{-1}\int(L_-Q)Y_2dx=0,\label{RE}
		\end{equation}
		we have
		\begin{align}
			\text{Im}\int\widetilde{R}_k 
			\overline{G_1(\boldsymbol{c})}dx
			=&\text{Re}\int\widetilde{R}_kc_k^\pm{\bar{\widetilde{Y}}}_k^{\scalebox{0.8}{\raisebox{-1.3ex}{$\pm$}}}dx+\text{Re}\int\widetilde{R}_kc_k^\pm\big(\bar{Y}_k^\pm-{\bar{\widetilde{Y}}}_k^{\scalebox{0.8}{\raisebox{-1.3ex}{$\pm$}}} \big)dx+\sum\limits_{j\neq k,\pm}\text{Re}\int\widetilde{R}_kc_j^\pm\bar{Y}_j^\pm dx\nonumber\\[5pt]
			\le& C\text{Re}\int QY^\pm dx+C|c_k^\pm|(|\alpha_k(G_1(\boldsymbol{b}))-\alpha_k^0|+|\theta_k(G_1(\boldsymbol{b}))-\theta_k^0|)+C |\mathbf{c}| e^{-\delta_2T} \nonumber\\
			=&\mathcal{O}\big(|\boldsymbol{c}|(e^{-\delta_2T}+|\boldsymbol{b}|)\big),
			\label{imi}
		\end{align}
		where $\widetilde{Y}_k$ is as in \eqref{3.49} with $t, \alpha_k(t), \theta_k(t)$ 
		replaced by $T$, $\alpha_k(G_1(\boldsymbol{b}))$, $\theta_k(G_1(\boldsymbol{b}))$, 
		respectively,   
		$1\leq k\leq K$.
		Similarly,
		using the algebraic identity
		\begin{equation}
			\text{Im}\int\nabla QY^\pm dx=\pm\int\nabla QY_2dx=\pm e_0^{-1}\int\nabla Q(L_+Y_1)dx=\pm e_0^{-1}\int(L_+\nabla Q)Y_1dx={\bf 0},\label{IM}
		\end{equation}
		and \eqref{6.48}, Lemma \ref{lem6.2} again we obtain		
		\begin{align}
			\text{Re}\int\partial_l\widetilde{R}_k\overline{G_1(\boldsymbol{c})} dx=\mathcal{O}\big(|\boldsymbol{c}|(e^{-\delta_2T}+|\boldsymbol{b}|)\big),\quad l=1,2,\cdots,d.
			\label{rei}
		\end{align}
		Plugging the above estimates \eqref{imi} and \eqref{rei} into \eqref{.h} and \eqref{dG2G1}
		we have
		\begin{equation*}
			dG_2\left(G_1(\boldsymbol{b})\right).G_1(\boldsymbol{c})=(G_1(\boldsymbol{c}),{\bf 0},{\bf 0})+\mathcal{O}\big(|\boldsymbol{c}|(e^{-\delta_2T}+|\boldsymbol{b}|)\big).
		\end{equation*}
		Inserting this into \eqref{dG} and using \eqref{3.57} we thus come to 
		\begin{align*}
			dG(\boldsymbol{b}).\boldsymbol{c}=\left(A_1^+,\cdots,A_K^+,A_1^-,\cdots,A_K^-\right)^\top
            +\mathcal{O}\big(|\boldsymbol{c}|(e^{-\delta_2T}+|\boldsymbol{b}|)\big) 
		\end{align*}
		with
		\begin{equation*}
			A_k^+=-\sum\limits_{j,\pm}c_j^\pm\ \text{Re}\int Y_k^+ \bar{Y}_j^\pm dx,~~\text{and}~~A_k^-=-\sum\limits_{j,\pm}c_j^\pm\ \text{Re}\int Y_k^- \bar{Y}_j^\pm dx,\quad 1\le k\le K. 
		\end{equation*}
		This along with Lemma \ref{lem6.2} yields that
		\begin{equation}
			dG(\boldsymbol{b})=
			\begin{bmatrix}
				A & y_*A \\[5pt]
				y_*A & A
			\end{bmatrix}+\mathcal{O}\big(e^{-\delta_2T}+|\boldsymbol{b}|\big),
			\label{3.60}
		\end{equation}
		where 
        $A:=\text{diag}\left(-\left(w_k\right)^{d-\frac{4}{p-1}}\right)_{1\le k\le K}$ 
        and 
        $y_*:=\int Y_1^2-Y_2^2dx$. 
        Note that 
		\begin{align*}
			y_*^2=\left(\int Y_1^2+Y_2^2dx\right)^2-4\int Y_1^2dx\int Y_2^2dx=1-4\int Y_1^2dx\int Y_2^2dx<1. 
		\end{align*} 
        
        We thus infer that 
        there exist deterministic constants $M>0$ 
        and $\eta \in (0,\eta')$, 
		such that for any $T\ge M$ and 
        any $\boldsymbol{b}\in B_{\mathbb{R}^{2K}}(\eta)$,
		the determinant of  $dG(\boldsymbol{b})$ is equal to 
		$$\big({\rm det}(A)\big)^2\big(1-y_*^2\big)^K+\mathcal{O}\big(e^{-\delta_2M}+\eta\big)> 0$$ 
		and is uniformly bounded from below, 
		independent of $\omega$. 
		Taking into account $G({\bf 0})={\bf 0}$ 
		we obtain that  
		for any $T\ge M$, 
		$G$ is a diffeomorphism from $B_{\mathbb{R}^{2K}}(\eta)$ to some neighborhood $U$ 
        of 
        ${\bf 0}\in \mathbb{R}^{2K}$, as claimed. 

\medskip  
 $(v)$ Uniform deterministic ball inside the random image of $G$:
We claim that, 
though the map $G$ depends on $T$ 
and the underlying probabilistic argument $\omega$, 
the image $U$ contains a deterministic small ball 
$B_{\mathbb{R}^{2K}}(r)$, 
where $r>0$ is a sufficiently small deterministic constant.

Actually, 
by the differential mean value theorem, 
there exists $\xi=\lambda\boldsymbol{b}$ with $0\le\lambda\le1$ 
such that 
\begin{equation*}
    G(\boldsymbol{b})=dG(\xi)\cdot\boldsymbol{b}.
\end{equation*}
Using the matrix in \eqref{3.60} we estimate 
\begin{align*}
    \big|G(\boldsymbol{b})\big|^2\ge&\sum\limits_{k=1}^K\Big(\big(y_*^2+1\big)(w_k)^{2d-\frac{8}{p-1}}\big((b_k^+)^2+(b_k^-)^2\big)+4y_*(w_k)^{2d-\frac{8}{p-1}}b_k^+b_k^-\Big)-C\big(e^{-\delta_2T}|\boldsymbol{b}|^2+|\boldsymbol{b}|^3\big)\\
    \ge&\Big(\big(y_*-1\big)^2(w_*)^{2d-\frac{8}{p-1}}-Ce^{-\delta_2T}\Big)|\boldsymbol{b}|^2-C|\boldsymbol{b}|^3,
\end{align*}
where $w_*:=\min_{1\le k\le K}\{w_k\}>0$, and $C$ is a positive constant independent of $T$. 
Then, 
for $M$ possibly larger  
and $\eta$ possibly even smaller,  
we get that for any $T\ge M$ and 
$\boldsymbol{b}\in B_{\mathbb{R}^K}(\eta)$,
\begin{equation*}
    \big|G(\boldsymbol{b})\big|^2\ge\frac{1}{2}\big(y_*-1\big)^2(w_*)^{2d-\frac{8}{p-1}}|\boldsymbol{b}|^2-C|\boldsymbol{b}|^3
    \ge\frac{1}{4}\big(y_*-1\big)^2(w_*)^{2d-\frac{8}{p-1}}|\boldsymbol{b}|^2, 
\end{equation*} 
Thus, letting $r := \frac{1}{2}
(1-y_*)(w_*)^{d-\frac{4}{p-1}} \eta$ 
and taking into account 
$G (\partial B_{\mathbb{R}^K}(\eta)) = \partial U$ we obtain that $U$ contains the 
deterministic small ball 
$B_{\mathbb{R}^{2K}}(r)$, as claimed.   

\medskip 
		Now,
        for any
		$\boldsymbol{a}^-
		\in B_{\mathbb{R}^K}(r)$, $\left({\bf 0},\boldsymbol{a}^-\right) \in B_{\mathbb{R}^{2K}}(r)$, 
        by the inverse of the diffeomorphism $G$, 
		there exists a unique $\boldsymbol{b}=\boldsymbol{b}(\boldsymbol{a}^-) \in B_{\mathbb{R}^{2K}}(\eta)$ 
		such that 
		$G(\boldsymbol{b}(\boldsymbol{a}^-))=({\bf 0},\boldsymbol{a}^-)$ 
		and 
		$|\boldsymbol{b}(\boldsymbol{a}^-)|\le C|\boldsymbol{a}^-|$,  
		where $C$ is a deterministic positive constant as the determinant of 
		$dG(\boldsymbol{b})$ has a 
		deterministic uniform lower bound. 
		Therefore, 
		we obtain \eqref{3.52} 
		and finish the proof of Proposition \ref{prop3.5}.
	\end{proof}

	\section{Modulation equations
		and remainder}  \label{subsec cm}
	
	In this section,
	we aim to control the modulation parameters
	and remainder in the geometrical decomposition
	\eqref{3.3}.

    To begin with,
	let us first control the noise appearing in the rescaling transform \eqref{1.28}. 

   \subsection{Control of noise} 
   \label{Subsec-Noise}
	Set $B_{*,k}(t):=\int_{t}^{\infty}g_k(s)dB_k(s)$,  
    $1\leq k\leq N$, 
    and 
    \begin{align}  \label{B*-def} 
    B_{*}(t):=\sup_{t\le{s}<\infty}\sum_{k=1}^{N}|B_{*,k}(s)|, \quad t>0.
    \end{align} 
    Since $g_k\in L^2(\mathbb{R}^+)$, one has
	\begin{equation}\label{limB*}
		\lim_{t\to+\infty}B_*(t)=0,~~\mathbb{P}\text{-a.s.}. 
	\end{equation} 
	This yields that there exists a large random time $\sigma_1$, $\sigma_1\in (0,\infty)$, $\mathbb{P}$-a.s.,  
    such that 
	\begin{equation}
		\sup \limits_{t\ge{\sigma_1}}B_*(t)\le1, 
        \quad \mathbb{P}-a.s.
		\label{3.7}
	\end{equation}

In view of Lemma \ref{Lem-B*} 
and \eqref{6.48}, 
for $|\boldsymbol{b}|$ small enough, 
	we can choose 
    $T^*\ge \sigma_1$  such that for any $t\in[T^*,T]$,
	\begin{equation}
		\sup\limits_{{T^*}\le{t}\le{T}}\Vert\varepsilon(t)\Vert_{H^1}<1,
		\label{3.8}
	\end{equation}
	and for every $1\le{k}\le{K}$, 
	\begin{equation}
		B_*(t)+|\alpha_k(t)-\alpha^0_k|\le\frac{1}{10}{\rm min}\{1,w_k,\alpha_k^0\}. 
		\label{3.9}
	\end{equation}
	In particular, $B_*(t)$,  $|\alpha_k(t)|$
	and $\|\ve(t)\|_{H^1}$
	are bounded by a universal deterministic constant on $[T^*,T]$.

\medskip 
We remark that in Section \ref{sec4} below, 
thanks to the bootstrap estimates 
\eqref{5.5}-\eqref{5.7}, 
both estimates 
\eqref{3.8} and \eqref{3.9} 
are valid after a large random time.

\medskip    
In Case (II),  
we have the following refined decay estimate of 
$B_{*}(t)$. 
It has been used in \cite{RSZ23}, 
for the convenience of reader, 
we include its proof here.

\begin{lemma}  \label{Lem-B*}
In Case (II),  
there exists a positive random time $\sigma_2$ such that $\mathbb{P}$-a.s. 
$\sigma_2\geq \sigma_*$, 
where $\sigma_*$ is the random time 
as in \eqref{1.15},  
and for any $t\ge \sigma_2$,
	\begin{align}  \label{5.25}
		|B_{*,k}(t)|
        \le
2\bigg(2\int_{t}^{\infty}g_k^2(s)ds\log\Big(\int_{t}^{\infty}g_k^2(s)ds\Big)^{-1}\bigg)^{\frac{1}{2}}
\le\frac{2\sqrt{c^*}}{t}, 
\quad 1\le k\le N. 
	\end{align}  
In particular, 
there exists a deterministic constant $C>0$ 
suc that  $\mathbb{P}$-a.s. for any $t\ge \sigma_2$, 
	\begin{align}  \label{B*-bdd} 
		B_{*}(t)  
\leq \frac{C}{t}. 
	\end{align}  
\end{lemma}

\begin{proof} 
In the following we fix $1\leq k\leq N$. 
Let $\sigma_{k,\infty}:= \int_0^\infty g_k^2(s)ds$. 
Then, $\sigma_{k,\infty} \in (0,\infty)$, $\mathbb{P}$-a.s., 
due to the $L^2$-integrability of $g_k$.  

In view of  the theorem on time-change for martingales,
there exists a Brownian motion $W_{k}$  
	such that 
    $B_{*,k}(t) = W_{k} (\sigma_{k,\infty}) - W_{k}(\int_0^t g_k^2(r)dr)$, $\mathbb{P}$-a.s..
Moreover, 
by the Levy H\"older continuity of Brownian motions 
and the invariance under time shift of the law of Brownian motions, 
i.e., for every $n\in \mathbb{N}$, 
$W_k(n+\cdot) - W_k(\cdot)$ 
has the same law as the standard Brownian motion, 
we note that $\mathbb{P}$-a.s.  
\begin{align} \label{Levy-Holder}
			\lim\limits_{h\to 0} 
			\sup_{\substack{|t-t'|\leq h\\[0pt] t,t'\in[n-1,n+1]}} 
			\frac{|W_k(t') - W_k(t)|}{\sqrt{2h\log(1/h)}}
			=1 
		\end{align}  
for every $n\in \mathbb{N}$. 
Then, set 
\begin{align*}
			h_n:= \inf
			\Bigg\{ 
			h\in [0,\frac{1}{10}], 
			\sup_{\substack{|t-t'|\leq h\\[0pt] t,t'\in[n-1,n+1]}} 
			\frac{|W_k(t') - W_k(t)|}{\sqrt{2h\log(1/h)}}
			>2
			\Bigg \} 
			\wedge \frac{1}{10}. 
		\end{align*}
We see that $h_n$ is $\mathscr{F}_\infty$-measurable, 
$0<h_n\leq 1/10$  $\mathbb{P}$-a.s. 
due to \eqref{Levy-Holder}, 
and 
\begin{align} \label{W-hn-upbdd}
				\sup_{\substack{|t-t'|\leq h\\[0pt] t,t'\in[n-1,n+1]}} 
				\frac{|W_k(t') - W_k(t)|}{\sqrt{2h\log(1/h)}}
			\leq 2, 
			\quad 
			\forall h\leq h_n. 
		\end{align} 

Let $[\sigma_{k,\infty}]$ be the largest integer less than $\sigma_{k,\infty}$, 
and 
define $h_{[\sigma_{k,\infty}]}$ by 
$h_{[\sigma_{k,\infty}]}(\omega) 
:= (h_{[\sigma_{k,\infty}(\omega)]})(\omega)$ 
for $\omega\in \Omega$. 
Then, $h_{[\sigma_{k,\infty}]}$ is $\mathscr{F}_\infty$-measurable.  
Moreover, 
by \eqref{W-hn-upbdd}, 
we infer that 
$\mathbb{P}$-a.s.   
\begin{align}  \label{Wtau-htau-upbdd}
    |W_k(\sigma_{k,\infty}) - W_k(\sigma_{k,\infty} -h)|
   \leq& 
   \sup_{{\substack{|t-t'|\leq h\\[0pt] 
    t,t'\in [[\sigma_{k,\infty}]-1,[\sigma_{k,\infty}]+1]}}} 
     |W_k(t')-W_k(t)|   \notag \\ 
   \leq& 2\sqrt{2h\log (1/h)}, 
   \quad 
   \forall h\leq h_{[\sigma_{k,\infty}]} \wedge \sigma_{k,\infty}. 
\end{align}

Now, 
let 
\begin{align*}
    \sigma_{k,0} 
    :=\inf\bigg\{ 
   t>0: 
    \int_t^\infty g_k^2(s) ds 
    \leq h_{[\sigma_{k,\infty}]}
    \bigg\}. 
\end{align*}  
Since by the $L^2$-integrability of $g_k$,  
$ \int_t^\infty g_k^2(s) ds \to 0$ 
as $t\to \infty$, $\mathbb{P}$-a.s., 
one has 
$0\leq \sigma_{k,0}  <\infty$, $\mathbb{P}$-a.s.. 
We also see from the definition of $\sigma_{k,0}$ 
that   $\mathbb{P}$-a.s.  
\begin{align*}
     \int_t^\infty g_k^2(s) ds 
     \leq \int_{\sigma_{k,0}}^\infty 
     g_k^2(s) ds 
     = h_{[\sigma_{k,\infty}]}, 
     \quad 
     \forall t\geq \sigma_{k,0}. 
\end{align*} 
Thus, taking into account \eqref {Wtau-htau-upbdd} 
we derive that  $\mathbb{P}$-a.s. 
for any $t\geq \sigma_{k,0}$, 
\begin{align}  
  |B_{*,k}(t)|
  =   |W_{k} (\sigma_{k,\infty}) 
  - W_{k}(\sigma_{k,\infty} 
  - \int_t^\infty g_k^2(s)ds)|   
  \leq 
2\bigg(2\int_{t}^{\infty}g_k^2(s)ds\log\Big(\int_{t}^{\infty}g_k^2(s)ds\Big)^{-1}\bigg)^{\frac{1}{2}}. 
\end{align} 
Therefore, 
setting 
\begin{equation}\label{tau-0-def}
    \sigma_2=\max\{\sigma_*, \sigma_{k,0}, \ 1\leq k\leq N\}
\end{equation} 
and using \eqref{1.15} in $(A1)$ 
we obtain \eqref{5.25} 
and finish the proof.
\end{proof}

	\subsection{Control of modulation equations}
	
	We have the following control of modulation equations.

	\begin{proposition}[Control of modulation equations]\label{prop3.4}
		Let $|\boldsymbol{b}|$ be sufficiently small, $T$ large enough, and $T^*$ close to $T$ such that Proposition \ref{prop3.2}, \eqref{3.8} and \eqref{3.9} hold.
		Then,  we have
		\begin{equation}
			\sum\limits_{k=1}^K(|\dot\alpha_k(t)|+|\dot\theta_k(t)|)\le{C}\left(\Vert\varepsilon(t)\Vert_{H^1}+B_*(t)\phi(\delta_1t)+e^{-\delta_2t}\right),
			\quad \forall{t}\in[T^*,T],
			\label{3.11}
		\end{equation}
		where 
		$\phi$ is the spatial decay function of noise in \eqref{1.17}, and 
		$C,\delta_1,\delta_2$ are deterministic constants,  depending on 
        $w_k$, $v_k$, $\alpha_k^0$ and $\delta_0$.
	\end{proposition}
    \begin{remark}
		By using random NLS equation \eqref{3.1} and the geometrical decomposition \eqref{3.3}, we obtain the equation of the remainder 
		$\varepsilon$ in  \eqref{3.63} below.
		Then, taking the inner product of equation \eqref{3.63} with two unstable directions, using the orthogonality conditions \eqref{3.5} and applying Lemma \ref{lem6.2}
		one can get the control of  modulation equations $|\dot\alpha_k(t)|$ and $|\dot\theta_k(t)|$.
		For more details,
		we refer to 
        analogous arguments in the proof of 
        Proposition 3.3 in \cite{RSZ23}. 
        
		It is worth noting that, 
        in the derivation of \eqref{3.11},
        the condition $\varepsilon(T)=0$ was used in the mass-(sub)critical case \cite{RSZ23} to obtain a-priori control of $\varepsilon$ on $[T^*,T]$.
		But in the mass-supercritical case here,
        we do not have this condition.
        Instead,
		the a-priori control of  $\varepsilon$ on $[T^*,T]$ is  derived from \eqref{time T} and the continuity of $\varepsilon$ by taking $|\boldsymbol{b}|$ small enough.
	\end{remark}

	For the extra unstable directions $\{a_k^\pm\}$,
	we have the following estimate.

	\begin{proposition}[Control of extra unstable directions] \label{prop3.6}
		Assume the conditions of Proposition \ref{prop3.4} to hold. Then,
		for every $1\le k \le K$,
		we have
		\begin{equation}
			\left|\dot a_k^\pm(t)\mp e_0(w_k)^{-2}a_k^\pm(t)\right|\le C\big(\Vert\varepsilon\Vert_{H^1}^{p\wedge2} + B_*(t)\phi(\delta_1t) + e^{-\delta_2t}\big),
			\quad \forall t\in [T^*, T],
			\label{3.61}
		\end{equation}
		where
		$C,\delta_1,\delta_2$ are universal deterministic constants,  depending on $w_k$, $v_k$, $\alpha_k^0$ and $\delta_0$.
	\end{proposition}

	\begin{remark}
		We note that the exponent $p\wedge 2$ 
		is due to the singularity of the second derivative of supercritical nonlinearity 
		at the origin when $p<2$. 
	\end{remark}

	\begin{proof}[Proof of Proposition \ref{prop3.6}]
		We first note from \eqref{3.50} that
		\begin{equation}
			\dot a_k^\pm(t)=\text{Im}\int\partial_t\bar{\varepsilon}~\widetilde{Y}_k^\pm dx+\text{Im}\int\bar{\varepsilon}~\partial_t\widetilde{Y}_k^\pm dx.
			\label{3.62}
		\end{equation}
		Using the explicit expression \eqref{3.4} and \eqref{3.49}
		we compute
		\begin{align}
			\label{R1}\partial_t\widetilde{R}_k(t,x)&=i\left(-\frac{|v_k|^2}{4}+\left(w_k\right)^{-2}+\dot{\theta}_k(t)\right)\widetilde{R}_k(t,x)-(v_k+\dot{\alpha}_k(t))\cdot(\nabla Q_{w_k})(y_k(t))e^{i\Phi_k(t,x)},\\
			\label{R2}\nabla\widetilde{R}_k(t,x)&=(\nabla Q_{w_k})(y_k(t) )e^{i\Phi_k(t,x)}+\frac{i}{2}v_k\widetilde{R}_k(t,x),\\
			\label{R3}
			\Delta\widetilde{R}_k(t,x)&=\left(\Delta Q_{w_k}+iv_k\cdot(\nabla Q_{w_k})-\frac{|v_k|^2}{4}Q_{w_k}\right)(y_k(t) )e^{i\Phi_k(t,x)},
		\end{align}
		and
		\begin{align}
			\partial_t\widetilde{Y}_k^\pm(t,x) &=i\left(-\frac{|v_k|^2}{4}+\left(w_k\right)^{-2}+\dot{\theta}_k(t)\right)\widetilde{Y}_k^\pm(t,x)-\left(v_k+\dot{\alpha}_k(t)\right)\cdot\left(\nabla Y_{w_k}^\pm\right)(y_k(t))e^{i\Phi_k(t,x)},     \label{Y1} \\
			\label{Y2}\nabla\widetilde{Y}_k^\pm(t,x) &=\left(\nabla Y_{w_k}^\pm\right)(y_k(t))e^{i\Phi_k(t,x)}+\frac{i}{2}v_k\widetilde{Y}_k^\pm(t,x),\\
			\label{Y3}\Delta\widetilde{Y}_k^\pm(t,x) &=\left(\Delta Y_{w_k}^\pm+iv_k\cdot(\nabla Y_{w_k}^\pm)-\frac{|v_k|^2}{4}Y_{w_k}^{\pm}\right)(y_k(t))
			e^{i\Phi_k(t,x)},
		\end{align}
	where $y_k$ is given by \eqref{y2.50}
		and the phase function $\Phi_k(t,x)$ is given by \eqref{3.12}.
		Taking into account  equation \eqref{1.4} of the rescaled ground state
		we infer that
		$\widetilde{R}_k(t, x)$ satisfies the equation
		\begin{align*}
			i\partial_t\widetilde{R}_k+\Delta\widetilde{R}_k+|\widetilde{R}_k|^{p-1}\widetilde{R}_k
			=-i\dot{\alpha}_k(t)\cdot(\nabla Q_{w_k})(y_k(t))e^{i\Phi_k}-\dot{\theta}_k(t)\widetilde{R}_k.
		\end{align*}
		Then, using the rescaled random NLS \eqref{3.1}
		and the geometrical decomposition \eqref{3.3}
		we obtain
		\begin{align}
			\label{3.63}
			\partial_t\varepsilon=&\text{$i$}\Delta\varepsilon+\sum\limits_{k=1}^{K}\dot{\alpha}_k\cdot\left(\nabla Q_{w_k}\right)(y_k(t))e^{i\Phi_k}-i\sum\limits_{k=1}^{K}\dot{\theta}_k\widetilde{R}_k+i|\widetilde{R}+\varepsilon|^{p-1}(\widetilde{R}+\varepsilon)\nonumber\\
			&-\text{$i$}\sum\limits_{k=1}^{K}|\widetilde{R}_k|^{p-1}\widetilde{R}_k+ib_*\cdot(\nabla\widetilde{R}+\nabla\varepsilon)+ic_*(\widetilde{R}+\varepsilon).
		\end{align}
		
		Plugging \eqref{Y1}-\eqref{3.63} into \eqref{3.62} we obtain
		\begin{align}
			\dot a_k^\pm(t)
			=&\text{Re}\int\left(\left(w_k\right)^{-2}Y_{w_k}^\pm-\Delta Y_{w_k}^\pm\right)(y_k(t))e^{i\Phi_k(t)}\bar{\varepsilon}dx\nonumber\\
			&+\left(\text{Re}\int\dot{\theta}_k\widetilde{Y}_k^\pm\bar{\varepsilon}dx-\text{Im}\int\dot{\alpha}_k\cdot\left(\nabla Y_{w_k}^\pm\right)(y_k(t))e^{i\Phi_k(t)}\bar{\varepsilon}dx\right)\nonumber\\
			&+\sum\limits_{j=1}^{K}\left(\text{Re}\int\dot{\theta}_j\bar{\widetilde{R}}_j\widetilde{Y}_k^\pm dx+\text{Im}\int\dot{\alpha}_j\cdot(\nabla Q_{w_j})(y_j(t))e^{-i\Phi_j(t)}\widetilde{Y}_k^{\pm}dx\right)\nonumber\\
			&-\text{Re}\int\bigg(|\widetilde{R}+\varepsilon|^{p-1}\big(\bar{\widetilde{R}}+\bar{\varepsilon}\big)-\sum\limits_{j=1}^{K}|\widetilde{R}_j|^{p-1}\bar{\widetilde{R}}_j\bigg)\widetilde{Y}_k^\pm dx\nonumber\\
			&-\text{Re}\int\left(\bar{b}_*\cdot\big(\nabla\bar{\widetilde{R}}+\nabla\bar{\varepsilon}\big)+\bar{c}_*\big(\bar{\widetilde{R}}+\bar{\varepsilon}\big)\right)\widetilde{Y}_k^\pm dx\nonumber\\
			=&:\sum\limits_{m=1}^{5}I_m.
			\label{3.67}
		\end{align}
		Note that, by identities \eqref{p1} and \eqref{p2}, 
		\begin{align}
			\label{3.68}Y^\pm
			-\Delta Y^\pm =\mp ie_0Y^\pm
			+Q^{p-1}Y^\pm
			+(p-1)Q^{p-1}\text{Re}Y^\pm.
		\end{align}
		It follows that
		\begin{align}
			I_1
			=&\pm e_0\left(w_k\right)^{-2}\text{Im}\int\widetilde{Y}_k^\pm \bar{\varepsilon}dx+\bigg(\frac{p+1}{2}\text{Re}\int |\widetilde{R}_k|^{p-1}\widetilde{Y}_k^\pm\bar{\varepsilon}dx+\frac{p-1}{2}\text{Re}\int|\widetilde{R}_k|^{p-3}\bar{\widetilde{R}}_k^{\scalebox{0.7}{\raisebox{-1.5ex}{2}}}\widetilde{Y}_k^{\pm}\varepsilon dx\bigg)\nonumber\\
			=&:\pm e_0\left(w_k\right)^{-2}\text{Im}\int\widetilde{Y}_k^\pm  \bar{\varepsilon}dx+I_6.
			\label{3.70}
		\end{align}
		Hence, plugging \eqref{3.70} into \eqref{3.67} we come to
		\begin{equation}
			\dot a_k^\pm(t)\mp e_0\left(w_k\right)^{-2}a_k^{\pm}(t)=\sum\limits_{m=2}^{6}I_m.
			\label{3.71}
		\end{equation}
		
		\medskip
		Next, we estimate each term $I_m$, $2\leq m\leq 6$,
		separately.
		
		$(i)$ Estimate of $I_2$.
		First, by H\"older's inequality,  the estimate of modulation equation \eqref{3.11}
		and the fact that $Y^\pm \in \mathcal{S}(\mathbb{R}^d)$,
		we have
		\begin{align}
			\left|I_2\right|\le&C (|\dot{\alpha}_k|+|\dot{\theta}_k|) \Vert\varepsilon\Vert_{L^2}\big(\Vert Y^\pm\Vert_{L^2}+\Vert\nabla Y^\pm\Vert_{L^2}\big)\nonumber\\
			\le& C(\Vert\varepsilon\Vert_{H^1}^2
			+B_*(t)\phi(\delta_1t)
			+e^{-\delta_2t}).
			\label{3.72}
		\end{align}
		
		$(ii)$ Estimate of $I_3$.
		Using the identities \eqref{RE}
		and \eqref{IM},
		the modulation estimate \eqref{3.11} and the decoupling Lemma \ref{lem6.2} we derive
		\begin{align}
			|I_3|
			\le&C|\dot{\theta}_k|\left|\text{Re}\int QY^{\pm}dx\right|+C|\dot{\alpha}_k|\left|\text{Im}\int\nabla QY^\pm dx\right|\nonumber\\
			&+\sum\limits_{j\neq k}\left(|\dot{\theta}_j|\int\left|\bar{\widetilde{R}}_j\widetilde{Y}_k^{\pm}\right|dx+|\dot{\alpha}_j|\int\left|(\nabla Q_{w_j})(y_j(t))e^{-i\Phi_j(t)}\widetilde{Y}_k^{\pm}\right|dx\right)\nonumber\\
			\le&C (\Vert\varepsilon\Vert^2_{H^1}
			+B_*(t)\phi(\delta_1t)+e^{-\delta_2t} ).\label{3.73}
		\end{align}
		
		$(iii)$ Estimate of $I_4$ and $I_6$.
		In order to estimate $I_4$ and $I_6$,
		when $p\ge2$, from direct computations and Lemma \ref{lem6.2},
		one has
		\begin{align}
			\text{Re}\int|\widetilde{R}+\varepsilon|^{p-1}(\bar{\widetilde{R}}+\bar{\varepsilon})\widetilde{Y}_k^{\pm}dx=&\text{Re}\int|\widetilde{R}_k+\varepsilon|^{p-1}(\bar{\widetilde{R}}_k+\bar{\varepsilon})\widetilde{Y}_k^\pm dx+\mathcal{O}\big(e^{-\delta_2t}\big)\nonumber\\
			=&\text{Re}\int|\widetilde{R}_k|^{p-1}\bar{\widetilde{R}}_k\widetilde{Y}_k^\pm dx+\frac{p+1}{2}\text{Re}\int|\widetilde{R}_k|^{p-1}\widetilde{Y}_k^\pm\bar{\varepsilon}dx\nonumber\\
			&+\frac{p-1}{2}\text{Re}\int|\widetilde{R}_k|^{p-3}\bar{\widetilde{R}}_k^{\scalebox{0.7}{\raisebox{-1.5ex}{2}}}\widetilde{Y}_k^{\pm}\varepsilon dx+\mathcal{O}\left(\Vert\varepsilon\Vert_{H^1}^2+e^{-\delta_2t}\right).
			\label{3.76}
		\end{align}
		Plugging \eqref{3.76} into $I_4$ we derive that
		for $p\geq 2$,
		\begin{equation}
			|I_4+I_6|\le C(\Vert\varepsilon\Vert_{H^1}^2
			+e^{-\delta_2t}).
			\label{3.80}
		\end{equation}
		
		When $p<2$,
		using Lemma \ref{lem6.2} one 
		can decouple different soliton profiles and eigenfucntions to get 
		\begin{align}
			I_4=&-\text{Re}\int\Big(|\widetilde{R}_k+\varepsilon|^{p-1}(\bar{\widetilde{R}}_k+\bar{\varepsilon})-|\widetilde{R}_k|^{p-1}\bar{\widetilde{R}}_k\Big)\widetilde{Y}_k^\pm dx+\mathcal{O}\big(e^{-\delta_2t}\big)\nonumber\\
			=&-\text{Re} \bigg(\int_{|\widetilde{R}_k|>2|\varepsilon|} + \int_{|\widetilde{R}_k|\le2|\varepsilon|}\bigg)
			|\widetilde{R}_k+\varepsilon|^{p-1}(\bar{\widetilde{R}}_k+\bar{\varepsilon})\widetilde{Y}_k^\pm-|\widetilde{R}_k|^{p-1}\bar{\widetilde{R}}_k\widetilde{Y}_k^\pm dx  +\mathcal{O}\big(e^{-\delta_2t}\big).
			\label{I4<}
		\end{align}
		Note that,
		since $p\in (1+\frac 4d, 1+\frac{4}{d-2})$,
		we may take $\rho(>1)$ close to 1 such that $2\le\rho p\le\frac{2d}{d-2}$ if $d\ge3$,
		$2\le\rho p<+\infty$ if $d=1$, $2$.
		Then, using the Sobolev embedding $H^1(\mathbb{R}^d)\hookrightarrow L^{\rho p}(\mathbb{R}^d)$ and 
        the H\"older inequality we estimate
		\begin{align}
			& \Big|\text{Re}\int_{|\widetilde{R}_k|\le2|\varepsilon|}|\widetilde{R}_k+\varepsilon|^{p-1}(\bar{\widetilde{R}}_k+\bar{\varepsilon})\widetilde{Y}_k^\pm-|\widetilde{R}_k|^{p-1}\bar{\widetilde{R}}_k\widetilde{Y}_k^\pm dx\Big| \notag \\ 
			\le&C\int|\varepsilon|^p|\widetilde{Y}_k^\pm|dx 
			\le C\|\varepsilon\|_{L^{\rho p}}^p\|\widetilde{Y}_k^\pm\|_{L^{\rho'}} 
			\le C\|\varepsilon\|_{H^1}^p.
			\label{^p}
		\end{align}
		
		Moreover, 
		one has the expansion 
		\begin{align}
			|\widetilde{R}_k+\varepsilon|^{p-1}(\bar{\widetilde{R}}_k+\bar{\varepsilon})
			=|\widetilde{R}_k|^{p-1}\bar{\widetilde{R}}_k+\frac{p-1}{2}|\widetilde{R}_k|^{p-3}
			\bar{\widetilde{R}}_k^{\scalebox{0.7}{\raisebox{-1.5ex}{2}}}
			\varepsilon+\frac{p+1}{2}|\widetilde{R}_k|^{p-1}\bar{\varepsilon}+Er
			\label{Wir}
		\end{align}
		with the error term
		\begin{align*}
			Er=\int_0^1(1-s)\Big[\frac{\partial^2f}{\partial{z^2}}(\widetilde{R}_k+s\varepsilon)
			\varepsilon^2+2\frac{\partial^2f}{\partial{z\bar{z}}}(\widetilde{R}_k+s\varepsilon)|\varepsilon|^2
			+\frac{\partial^2f}{\partial{\bar{z}^2}}(\widetilde{R}_k+s\varepsilon)\bar{\varepsilon}^2\Big]ds, 
		\end{align*}
		where $f$ is defined by 
		$f(z):=|z|^{p-1}\bar{z}$, 
		$z\in \mathbb{C}$. 
		When $p<2$ and $x\in\{x:|\widetilde{R}_k|>2|\varepsilon|\}$,
		we have
		\begin{align*}
			|Er|\le C\int_0^1|\widetilde{R}_k+s\varepsilon|^{p-2}|\varepsilon|^2ds\le C|\varepsilon|^p,
		\end{align*}
		and so, as in \eqref{^p}, 
		\begin{equation}
			\Big|\text{Re}\int_{|\widetilde{R}_k|>2|\varepsilon|}\widetilde{Y}_k^\pm Erdx\Big|\le C\int|\varepsilon|^p|\widetilde{Y}_k^\pm|dx\le C\|\varepsilon\|_{H^1}^p.
			\label{Er}
		\end{equation}
		
		Thus, plugging \eqref{^p}-\eqref{Er} into \eqref{I4<} we obtain
		\begin{align}
			I_4=&-\text{Re}\int_{|\widetilde{R}_k|>2|\varepsilon|}\frac{p+1}{2}|\widetilde{R}_k|^{p-1}\widetilde{Y}_k^\pm\bar{\varepsilon}
			+\frac{p-1}{2}|\widetilde{R}_k|^{p-3}\bar{\widetilde{R}}_k^{\scalebox{0.7}{\raisebox{-1.5ex}{2}}}\widetilde{Y}_k^\pm\varepsilon dx+\mathcal{O}(\|\varepsilon\|_{H^1}^p+e^{-\delta_2t}).
			\label{I4}
		\end{align}
		
		Similarly, as in the proof of \eqref{^p},
		one has
		\begin{align}
			I_6=&\text{Re}\int\frac{p+1}{2}|\widetilde{R}_k|^{p-1}\widetilde{Y}_k^\pm\bar{\varepsilon}dx+\text{Re}\int\frac{p-1}{2}|\widetilde{R}_k|^{p-3}\bar{\widetilde{R}}_k^{\scalebox{0.7}{\raisebox{-1.5ex}{2}}}\widetilde{Y}_k^{\pm}\varepsilon dx\nonumber\\
			=&\text{Re}\int_{|\widetilde{R}_k|>2|\varepsilon|}\frac{p+1}{2}|\widetilde{R}_k|^{p-1}\widetilde{Y}_k^\pm\bar{\varepsilon}dx+\text{Re}\int_{|\widetilde{R}_k|>2|\varepsilon|}\frac{p-1}{2}|\widetilde{R}_k|^{p-3}\bar{\widetilde{R}}_k^{\scalebox{0.7}{\raisebox{-1.5ex}{2}}}\widetilde{Y}_k^{\pm}\varepsilon dx+\mathcal{O}(\|\varepsilon\|_{H^1}^p).
			\label{I6}
		\end{align}
		
		Combining \eqref{I4} and \eqref{I6} together,
		we thus obtain that when $p<2$,
		\begin{equation}
			|I_4+I_6|\le C(\|\varepsilon\|_{H^1}^p+e^{-\delta_2t}).
			\label{p<2}
		\end{equation}
		
		Therefore, we conclude from \eqref{3.80} 
		and \eqref{p<2} that 
		for $p\in (1+\frac 4d, 1+\frac{4}{(d-2)_+})$, 
		\begin{equation}  \label{I46-esti}
			|I_4+I_6|\le C(\|\varepsilon\|_{H^1}^{p\wedge 2}+e^{-\delta_2t}).
		\end{equation}

		$(iv)$ Estimate of $I_5$. It remains to treat the $I_5$ term involving the random coefficients. 
		By  H\"older's inequality,
		expressions \eqref{1.30} and \eqref{1.31},
		and the change of variables,
		$I_5$ can be bounded by
		\begin{align}
			|I_5|
			\le&CB_*(t)\sum\limits_{l=1}^{N}\int|\nabla\phi_l(y+v_kt+\alpha_k(t))|(|\nabla Q_{w_k}(y)|+|Q_{w_k}(y)|)dy\nonumber\\
			&+CB_*(t)\sum\limits_{l=1}^{N}\int\left(|\nabla\phi_l(y+v_kt+\alpha_k(t))|^2+|\Delta\phi_l(y+v_kt+\alpha_k(t))|\right)|Q_{w_k}(y)|dy\nonumber\\
			&+CB_*(t)\Vert\nabla\varepsilon\Vert_{L^2}\left(\sum\limits_{l=1}^{N}\int|\nabla\phi_l(y+v_kt+\alpha_k(t))|^2|Y^\pm(w_k^{-1}y)|^2dy\right)^{\frac{1}{2}}\nonumber\\
			&+CB_*(t)\Vert\varepsilon\Vert_{L^2}\left(\sum\limits_{l=1}^{N}\int\left(|\nabla\phi_l(y\!+\!v_kt\!+\!\alpha_k(t))|^2\!+\!|\Delta\phi_l(y\!+\!v_kt\!+\!\alpha_k(t))|\right)^2|Y^\pm(w_k^{-1}y)|^2dy\right)^{\frac{1}{2}}\!\!+\!Ce^{-\delta_2t}.
			\label{3.81}
		\end{align}

We note that the spatial functions of the noise 
travel with the speeds $\{v_k\}$. 
Hence, intuitively, 
after a large time, 
they shall be separated sufficiently far away 
from the ground state. 

In order to capture this fact, 
we split the integration into two regimes $|y|\leq {|v_k|}t/2$ 
        and $|y| > {|v_k|}t/2$. 
        The key observation is that, 
        for $|y|\leq {|v_k|}t/2$,
		where $t$ is large enough such that $t\geq 8 |\alpha_k^0|/|v_k|$,
		by \eqref{3.9},  $|y+v_kt+\alpha_k(t)|\ge|v_k|t-|y|-|\alpha_k(t)|\geq |v_k|t/4$. 
Thus, 
in view of Assumption ${\rm (A_1)}$, 
the integration in this regime  
can be controlled by the spatial decay 
rate of $\na \phi_l$ in the  noise. 
Moreover, 
in the outer large regime $|y| > {|v_k|}t/2$, 
the integration can be bounded by the exponential decay of the ground state $Q$. 
As a result,  the first and second 
integrations on the right-hand side of \eqref{3.81} above can be bounded by
		\begin{align}
			&CB_*(t)\sum\limits_{l=1}^{N}\left(\int_{|y|\le\frac{|v_k|}{2}t}|\nabla\phi_l(y+v_kt+\alpha_k)|(|\nabla Q_{w_k}(y)|+|Q_{w_k}(y)|)+|\nabla\phi_l(y+v_kt+\alpha_k)|^2|Q_{w_k}(y)|\right.\nonumber\\
			&\quad\quad\quad\quad\quad\quad~\quad\quad+\left.|\Delta \phi_l(y+v_kt+\alpha_k)||Q_{w_k}(y)|dy+\int_{|y|\ge\frac{|v_k|}{2}t}|Q_{w_k}(y)|+|\nabla Q_{w_k}(y)|dy\right)\nonumber\\
			\le&CB_*(t)\phi(\frac{|v_k|}{4}t)\left(\int |Q_{w_k}(y)|+|\nabla Q_{w_k}(y)|dy
			+ 
			\int_{|y|\ge\frac{|v_k|}{2}t}e^{-\frac{2\delta_0}{w_k}|y|}dy\right)\nonumber\\
			\le&CB_*(t)(\phi(\delta_1t)+e^{-\delta_2t})
			\label{3.82}
		\end{align}
		for some positive deterministic constants $C$, $\delta_1$ and $\delta_2$, 
        depending on $w_k$, $\alpha_k^0$, $v_k$ and $\delta_0$ from Lemma \ref{lem6.2}.

		Similarly, 
		in view of the exponential decay of $Y^\pm$ in Lemma \ref{Y} and 
		the spatial decay of noise in Assumption $(A_1)$, 
		the remaining two terms on the right-hand side of \eqref{3.81} also can be bounded by
		\begin{equation}
			C (\|\varepsilon\|^2_{H^1}+ B_*(t)\phi(\delta_1t)+ e^{-\delta_2t}).
			\label{3.83}
		\end{equation}
		
		Finally, combining \eqref{3.71}-\eqref{3.73}, \eqref{I46-esti}, \eqref{3.82} and \eqref{3.83} 
        altogether we obtain \eqref{3.61} and finish the proof.
	\end{proof}

	\subsection{Control of the remainder} \label{subsec cr}
	
	In this subsection,
	we control the remainder $\varepsilon$ in the geometrical decomposition.
	The key role is played by the Lyapunov functional $ \mathcal{G}$
	defined in \eqref{4.103} below.

	Because the velocities $\{v_k\}$ of soliton profiles are different,
	without loss of generality,
	we may assume that $v_{1,1}<v_{2,1}<\cdots<v_{K,1}$.
	Set
	$A_0:=\frac{1}{4} {\min}_{2\le k \le K}\left\{v_{k,1}-v_{k-1,1}\right\}$,
	$A_k:=\frac{1}{2}\left(v_{k-1,1}+v_{k,1}\right)$, $2\le k \le K$. Let $\Psi(x)$ be a smooth non-decreasing function on $\mathbb{R}$ such that $0\le \Psi \le 1$, $\Psi(x)=0$ for $x\le -A_0$, $\Psi(x)=1$ for $x\ge A_0$,
	and for some $C>0$,
	\begin{equation}
		\left(\Psi'(x)\right)^2\le C\Psi(x),\quad\left(\Psi''(x)\right)^2\le C\Psi'(x).
		\label{4.1}
	\end{equation}
	Define the localization functions by
	\begin{equation}
		\begin{aligned}
			&\varphi_1(t,x)=1-\Psi\left(\frac{x_1-A_2t}{t}\right),\quad\varphi_K(t,x)=\Psi\left(\frac{x_1-A_Kt}{t}\right),\\
			&\varphi_k(t,x)=\Psi\left(\frac{x_1-A_kt}{t}\right)-\Psi\left(\frac{x_1-A_{k+1}t}{t}\right),\quad2\le k \le K-1
		\end{aligned}
		\label{4.2}
	\end{equation}
	for any $x=(x_1,x_2,\cdots,x_d)\in\mathbb{R}^d$. 
	Note that
	$\sum_{k=1}^{K}\varphi_k (t,x)=1$,
	and
	\begin{align} \label{4.3}  |\partial_x\varphi_k(t,x)|+|\partial_x^3\varphi_k(t,x)|+|\partial_t\varphi_k(t,x)|\le \frac{C}{t}.
	\end{align}
	
	Proposition \ref{coe} below 
    provides the
	main control of the remainder 
    in the geometrical decomposition.

	\begin{proposition}[Coercivity type control of remainder]\label{coe}
     There exist a deterministic constant $r(>0)$ and  
      a positive random variable $T$,     
		 such that for any $\boldsymbol{a}^-\in{B}_{\mathbb{R}^K}\left(r\right)$ 
         and $T^*$ close to $T$,  
		there exist universal deterministic constants $C,\delta_1,\delta_2$,
		depending on $w_k$, $v_k$, $\alpha_k^0$ and $\delta_0$,
		\begin{align}
			\|\varepsilon(t)\|_{H^1}^2\le &C\int_{t}^{\infty}\Big(\frac{1}{s}+B_*(s)\Big)\Vert\varepsilon(s)\Vert_{H^1}^2ds+C\bigg(\int_t^{\infty}\|\varepsilon(s)\|_{H^1}^{p\wedge2}ds\bigg)^2+C\int_{t}^{\infty}B_*(s)\phi(\delta_1s)ds\nonumber\\
			&+C\bigg(\int_{t}^{\infty}B_*(s)\phi(\delta_1s)ds\bigg)^2+C \left(|\mathbf{a}^-(t)|^2+ |\boldsymbol{a}^-|^2+ e^{-\delta_2t} \right)+\beta\|\varepsilon(t)\|^2_{H^1},\ \forall t \in [T^*,T],
			\label{4.103}
		\end{align}
		where $\beta\to0$ as $\|\varepsilon(t)\|_{H^1}\to0$.
	\end{proposition}	
	
	The key role in the proof of Proposition \ref{coe}
	is played  by the following Lyapunov type functional
	\begin{align}
		\mathcal{G}(t):=&\|\nabla u\|_{L^2}^2-\frac{2}{p+1}\|u\|_{L^{p+1}}^{p+1}+\sum\limits_{k=1}^K\bigg\{\left(\left(w_k\right)^{-2}+\frac{|v_k|^2}{4}\right)\int|u(t,x)|^2\varphi_k(t,x)dx\nonumber\\
		&\quad -v_k\cdot\text{Im}\int\nabla u(t,x)\bar{u}(t,x)\varphi_k(t,x)dx\bigg\}.
		\label{4.21}
	\end{align}
	where $u$ is the solution to the rescaled random NLS \eqref{3.1}.
	As in the proof of Propositions 4.1, 4.3 and 4.4 in \cite{RSZ23}, we have the following control of the Lyapunov functional.
	\begin{lemma}[Control of Lyapunov functional]\label{prop4.4}
		For any $t\in [T^*,T]$ one has
		\begin{equation}
			\left|\frac{d}{dt}\mathcal{G}(t)\right|\le \frac{C}{t}\left(\Vert\varepsilon(t)\Vert_{H^1}^2+e^{-\delta_2t}\right)+CB_*(t)\left(\Vert\varepsilon(t)\Vert_{H^1}^2+\phi(\delta_1t)+e^{-\delta_2t}\right),
			\label{4.211}
		\end{equation}
		where $C,\delta_1,\delta_2$ are deterministic constants,  depending on $w_k$, $v_k$, $\alpha_k^0$ and $\delta_0$.
		Moreover, the following  expansion holds:
		\begin{align}
			\mathcal{G}(t)=&\sum\limits_{k=1}^{K}\left(\|\nabla Q_{w_k}\|_{L^2}^2-\frac{2}{p+1}\|Q_{w_k}\|_{L^{p+1}}^{p+1}+\left(w_k\right)^{-2}\Vert Q_{w_k}\Vert_{L^2}^2\right)+H(\varepsilon(t))\nonumber\\
			&+\mathcal{O}(e^{-\delta_2t})+ 
			\beta \Vert\varepsilon(t)\Vert_{H^1}^2,
			\label{4.22}
		\end{align}
		where $\beta\to0$ as $\|\varepsilon(t)\|_{H^1}\to 0$, 
		and $H(\varepsilon)$ contains the quadratic terms of the remainder $\varepsilon$, i.e.,
		\begin{equation}
			\begin{aligned}
				H(\varepsilon)=&\int|\nabla\varepsilon|^2dx-\sum\limits_{k=1}^{K}\int |\widetilde{R}_k|^{p-1}|\varepsilon|^2+(p-1)|\widetilde{R}_k|^{p-3}\left({\rm Re}\widetilde{R}_k\bar{\varepsilon}\right)^2dx\\
				&+\sum\limits_{k=1}^{K}\left\{\left(\left(w_k\right)^{-2}+\frac{|v_k|^2}{4}\right)\int|\varepsilon|^2\varphi_k dx-v_k\cdot
				{\rm Im}\int\nabla\varepsilon\bar{\varepsilon}\varphi_kdx\right\}.
			\end{aligned}
			\label{4.23}
		\end{equation}
	\end{lemma}

	\begin{remark}
		The quadratic term $H(\varepsilon(t))$ has the crucial coercivity type estimate
		\begin{equation}
			\Vert\varepsilon(t)\Vert_{H^1}^2\le CH(\varepsilon(t))
			+C\big(|\mathbf{a}^+(t)|^2+|\mathbf{a}^-(t)|^2\big),~~t\in \left[T^*,T\right]
			\label{4.41}
		\end{equation}
		for some $C>0$.
		This can be proved by using the coercivity of the linearized operator in \eqref{6.1},
		the orthogonal conditions in \eqref{3.5},
		and analogous arguments as in the proof of the 1D case in \cite[Appendix B]{MMT06}.
	\end{remark}
	
	We are now ready to prove Proposition \ref{coe}.
	\begin{proof}[Proof of Proposition \ref{coe}]
		  In view of Proposition \ref{prop3.5}, 
        we can take a deterministic small 
        constant $r(>0)$ 
        such that for any $\boldsymbol{a}^-\in{B}_{\mathbb{R}^K}\left(r\right)$, there exists a unique $\boldsymbol{b}\in \mathbb{R}^{2K}$ such that  
		\begin{equation}\label{esti-aT-b}
			\mathbf{a}^+(T)={\bf 0}, \quad \mathbf{a}^-(T)=\boldsymbol{a}^-,\quad
			{\rm and}\ \ |\boldsymbol{b}|\le{C}|\boldsymbol{a}^-|\le Cr.
		\end{equation} 
        Then, we take $r$ possibly smaller,  
        a random time $T$ large enough 
		and $T^*$ close to $T$,  
        such that the geometrical decomposition 
        in Proposition \ref{prop3.2}, 
        \eqref{3.8} and \eqref{3.9} hold.

        Using the coercivity estimate \eqref{4.41}, for any $t\in[T^*,T]$, one has
		\begin{align}
			\Vert\varepsilon(t)\Vert_{H^1}^2\le C|H(\varepsilon(T))|+C|H(\varepsilon(t))-H(\varepsilon(T))|
			+C \big(|\mathbf{a}^+(t)|^2+|\mathbf{a}^-(t)|^2\big).
			\label{3.100}
		\end{align}
		
		We shall estimate the right-hand side of \eqref{3.100}.
		By Proposition \ref{prop3.6},
			\begin{equation}\label{esti-ak+}
			\left|\dot a_{k}^+(t)- e_0\left(w_k\right)^{-2}a_{k}^+(t)\right|
			\leq C(\Vert\varepsilon(t)\Vert_{H^1}^{p\wedge2}
			+ B_*(t)\phi(\delta_1t)
			+ e^{-\delta_2t}). 
		\end{equation}
		Then,
		it follows from \eqref{esti-ak+}, $\mathbf{a}^+(T)={\bf 0}$ and Gronwall's inequality that for any $t\in[T^*,T]$,
		\begin{equation}
			|a_{k}^+(t)|\leq C\int_{t}^{\infty}\Big(\|\varepsilon(s)\|_{H^1}^{p\wedge2}+B_*(s)\phi(\delta_1s)\Big)ds+Ce^{-\delta_2t}.   \label{4.101}
		\end{equation} 
		
		 Moreover,
		by \eqref{time T} and \eqref{esti-aT-b},
		\begin{equation}
			\left|H(\varepsilon(T))\right|\le C\Vert\varepsilon(T)\Vert_{H^1}^2\le C|\boldsymbol{b}|^2\le C|\boldsymbol{a}^-|^2.
			\label{5.15}
		\end{equation} 
		
		 Regarding the $|H(\varepsilon(t))-H(\varepsilon(T))|$ term, by the expansion \eqref{4.22}, for any $t\in[T^*,T]$,
		\begin{align}
			\left|H(\varepsilon(t))-H(\varepsilon(T))\right|
			\le
			Ce^{-\delta_2t}+\beta\Vert\varepsilon(t)\Vert_{H^1}^2+\left|\mathcal{G}(t)-\mathcal{G}(T)\right|,
			\label{5.14}
		\end{align}
		where $\beta\to 0$ as $\Vert\varepsilon\Vert_{H^1}^2\to 0$.
		Integrating both sides of \eqref{4.211} on $[t,T]$ we have
		\begin{equation}
			\left|\mathcal{G}(t)-\mathcal{G}(T)\right|\leq
			C \int_{t}^{T} \frac{1}{s}\left(\Vert\varepsilon(s)\Vert_{H^1}^2+e^{-\delta_2s}\right)ds
			+ C \int_{t}^{T}B_*(s)\left(\Vert\varepsilon(s)\Vert_{H^1}^2+\phi(\delta_1s)+e^{-\delta_2s}\right)ds.\label{5.13}
		\end{equation}
		Hence, combining \eqref{5.14} and \eqref{5.13} together and using \eqref{3.7} we  derive that
		\begin{align}
			\left|H(\varepsilon(t))\!-\!H(\varepsilon(T))\right|\le C\int_{t}^{\infty}\Big(\frac{1}{s}\!+\!B_*(s)\Big)\Vert\varepsilon(s)\Vert_{H^1}^2ds+C\int_{t}^{\infty}B_*(s)\phi(\delta_1s)ds
			+Ce^{-\delta_2t}+\beta\Vert\varepsilon(t)\Vert_{H^1}^2.
			\label{4.102}
		\end{align} 
		
	 Therefore, plugging \eqref{4.101}, \eqref{5.15} and \eqref{4.102} into \eqref{3.100} we obtain \eqref{4.103}. 
	\end{proof}

	\section{Uniform estimates of approximating solutions}\label{sec4}

	For every $n\in \mathbb{N}$, 
    we consider the approximating equation
	\begin{equation}
		\begin{cases}
			&i\partial_tu_n+(\Delta+b_*\cdot\nabla+c_*)u_n+|u_n|^{p-1}u_n=0,\\
			&u_n(n)=R(n)+i\sum\limits_{k,\pm}b_{k,n}^\pm Y_k^\pm(n).
		\end{cases}
		\label{5.1}
	\end{equation}

	Let $\delta_1$, $\delta_2$ be as in Proposition \ref{prop3.4}, 
    and 
	 \begin{equation}
			\delta_3:=\frac{1}{2} \min_{1\le k\le K}\{e_0(w_k)^{-2}\},
            \label{def-delta3}
		\end{equation}
		with $e_0(>0)$ being the eigenvalue of the linearized Schr\"odinger operator in \eqref{1.10},
		and $w_k$ the frequency of the soliton \eqref{1.11}. Set 
	\begin{equation}
		\widetilde{\delta}:=
		\begin{cases}
			\frac{1}{2}(\delta_1\wedge\delta_2\wedge\delta_3),&\text{in Case (I)};\\
			\delta_1,&\text{in Case (II)}.
		\end{cases}
		\label{3.51}
	\end{equation}
	
	The main result of this section is the following crucial uniform estimate, 
    which shows that the approximating  equation \eqref{5.1} can be solved backward up to a universal time $T_0$, uniformly in $n$.
	
	\begin{theorem}[Uniform estimate]\label{th5.1}
		Let $\widetilde{\delta}$ be as in \eqref{3.51} and $\nu_0$ as in \eqref{def-v0} below.
		 Assume $(A_0)$ and $(A_1)$ with $\nu_*\ge\nu_0$ in Case (II). 
		Then, 
		there exist a positive random time $T_0$ 
        and  a deterministic constant $C>0$, 
        such that for 
        $\mathbb{P}$-a.e. $\omega\in \Omega$ 
         there exists $N_0(\omega)\ge1$ 
		such that for any $n\ge N_0(\omega)$,  
        $T_0(\omega)<n$ and 
        the following holds:
		
		There exists $\boldsymbol{b}_n(\omega)\in B_{\mathbb{R}^{2K}}(C\phi^{\frac{1}{2}+\frac{1}{4d}}(\widetilde{\delta}n))$
		such that the solution $u_n(\omega)$ to equation (\ref{5.1}) admits the geometrical decomposition
		\eqref{3.3} on $[T_0(\omega),n]$ and satisfies
		\begin{equation}
			\Vert u_n(t,\omega)-R(t,\omega)\Vert_{H^1}\le Ct\phi^{\frac{1}{2}}(\widetilde{\delta}t),
            \quad \forall  t\in[T_0(\omega),n],  \label{5.2}
		\end{equation}
		where $\phi$ is the spatial decay function of noise  given by (\ref{1.17}).
	\end{theorem}

	\begin{remark}
		We note that
the temporal convergence rate of the approximating solution $u_n$, 
as well as the smallness of 
the modulated parameter $\mathbf{b}$, 
are dictated by the spatial decay function of 
the noise. 
	\end{remark}
	
	The proof of Theorem \ref{th5.1} mainly proceeds in two steps. First in Subsection \ref{Subsec-Bootstrap},
	we prove the bootstrap estimates of the reminder $\varepsilon_n$, 
    the modulation parameters $(\alpha_n,\theta_n)$,  and the unstable direction 
    $\mathbf{a}_n^+$ under 
    an a-priori control of $\mathbf{a}_n^-$.
	Then, in Subsection \ref{Subsection-an-},
	we control the remaining parameter $\mathbf{a}_n^-$ by using topological arguments. 

    Let us mention that, 
    in the sequel, 
    we mainly consider $\boldsymbol{a}_n^-\in B_{\mathbb{R}^K}(\phi^{\frac{1}{2}+\frac{1}{4d}}(\widetilde{\delta}n))$ 
    with $n$ large enough 
such that 
$\phi^{\frac{1}{2}+\frac{1}{4d}} (\widetilde{\delta}n) \leq r$  
and $\boldsymbol{b}\in B_{\mathbb{R}^{2K}}(\eta)$, 
where $r,\eta$ are small deterministic constants from Propositions \ref{prop3.5} 
and \ref{prop3.2}, respectively, 
   so that 
    the geometrical decomposition 
    and the final condition in 
    two propositions 
    hold.  
   In particular, 
         for any
		$\boldsymbol{a}_n^- \in B_{\mathbb{R}^K}(\phi^{\frac 12+\frac 1{4d}}(\wt \delta n)))$,
		there exists a unique small vector  $\boldsymbol{b}_n\in\mathbb{R}^{2K}$ such that
		\begin{align}  
			\label{an-final} 
            & \mathbf{a}_n^+(n) = {\bf 0},\quad
			\mathbf{a}_n^-(n) = \boldsymbol{a}_n^-,  \\ 
            \label{bn-esti} 
            & |\boldsymbol{b}_n|
			\leq C |\boldsymbol{a}_n^-|
			\leq C \phi^{\frac{1}{2}+\frac{1}{4d}}(\wt \delta n)
			\leq C \phi^{\frac{1}{2}+\frac{1}{4d}}(\wt \delta t),\ \forall t\le n,
		\end{align} 
		   where $C>0$ is a deterministic positive constant independent of $n$.

	\subsection{Bootstrap estimates} \label{Subsec-Bootstrap}
	
	The main bootstrap estimates of this subsection
	are stated as follows.
	
	\begin{proposition}[Bootstrap estimates of $\varepsilon_n$, $\alpha_n$, $\theta_n$ and $\mathbf{a}_n^+$ ]\label{prop5.2}
		Assume the conditions of Theorem \ref{th5.1} to hold. 
		Then, 
        there exists a random time 
        $\tau^* (>0)$,  
        such that $\mathbb{P}$-a.e. $\omega\in \Omega$  
        there exists $N_0(\omega)\ge1$ 
		such that for any $n\ge N_0(\omega)$,  
        $\tau^*(\omega)<n$ 
        and  the following holds:
		
		Let $t^*(\omega)\in(\tau^*(\omega), n]$ be such that for any $t\in[t^*, n]$, $u_n(\omega)$ satisfies
		\begin{equation}
			\Vert u_n(t,\omega)-R(t,\omega)\Vert_{H^1}\le \frac{1}{2}\delta_* 
			\label{4p}
		\end{equation} 
		with $\delta_*$ being 
		the small number from Lemma \ref{lem6.3}, 
		and the following estimates hold:
		\begin{equation}
			\Vert \varepsilon_n(t,\omega)\Vert_{H^1}
            \le\phi^{\frac{1}{2}}(\widetilde{\delta}t),~~~ |\mathbf{a}_n^+(t,\omega)|\le \phi^{\frac{1}{2}}(\widetilde{\delta}t),~~~ |\mathbf{a}_n^-(t,\omega)|\le \phi^{\frac{1}{2}+\frac{1}{4d}}(\widetilde{\delta}t),
			\label{5.3}
		\end{equation}
		\begin{equation}
			\sum\limits_{k=1}^{K}(|\alpha_{n,k}(t,\omega)-\alpha_k^0|+|\theta_{n,k}(t,\omega)-\theta_k^0|)\le t\phi^{\frac{1}{2}}(\widetilde{\delta}t).
			\label{5.4}
		\end{equation}
		Then,
		there exists a smaller time $t_*\in(\tau^*,t^*)$,
		such that
		$u_n$ admits the geometrical decomposition \eqref{3.3} on the larger time interval $[t_*, n]$, and the improved estimates hold:
		\begin{equation}
			\Vert u_n(t,\omega)-R(t,\omega)\Vert_{H^1}\le \frac{1}{4}\delta_*, \label{5.5}
		\end{equation}
		\begin{equation}
			\Vert \varepsilon_n(t,\omega)\Vert_{H^1}\le\frac{1}{2}\phi^{\frac{1}{2}}(\widetilde{\delta}t),~~~ |\mathbf{a}_n^+(t,\omega)|\le \frac{1}{2}\phi^{\frac{1}{2}}(\widetilde{\delta}t),
			\label{5.6}
		\end{equation}
		\begin{equation}
			\sum\limits_{k=1}^{K}(|\alpha_{n,k}(t,\omega)-\alpha_k^0|
            +|\theta_{n,k}(t,\omega)-\theta_k^0|)\le\frac{1}{2}t\phi^{\frac{1}{2}}(\widetilde{\delta}t),  \quad \forall t\in [t_*, n].\label{5.7}
		\end{equation}
	\end{proposition}

\begin{remark}
We remark that the constants in estimates   
 \eqref{5.5}-\eqref{5.7} 
are smaller than those in 
\eqref{4p}-\eqref{5.4}. 
This is because 
in the following analysis one 
can gain small factors $o(1)$ 
before the deterministic constants $C$. 
The small factors are contributed by 
the exponential and polynomial decay of time, 
the tail of noise $B_*(t)$ 
and $\nu_*^{-1}$ with $\nu_*$ large enough. 

We also note that 
the bootstrap estimates above require a-priori 
control of $\mathbf{a}_n^-(t)$, 
which cannot be improved in estimates 
\eqref{5.5}-\eqref{5.7}. 
Later in Proposition \ref{prop5.4}, 
we shall use topological arguments to 
choose a suitable
final data $\boldsymbol{a}_n^-$ 
to obtain the required a-priori control.

\end{remark}

	\begin{proof}[Proof of Proposition \ref{prop5.2}]
		We define the 
        random time $\tau^*$ by 
       \begin{equation}\label{def-Ta-}
		\tau^*:=\begin{cases}
			\max\{M, \sigma_1, \tau_j, 
            j=1,2,3\},&\text{in Case (I)};\\
			\max\{M, \sigma_1, 
            \sigma_2, \tau_j, 
            j=1,4,5\},&\text{in Case (II)},
		\end{cases}
	\end{equation} 
    and let $N_0=[\tau^*]+1$, where 
    $M$ is the deterministic large time 
    from the geometrical decomposition in   Proposition \ref{prop3.2},  
    $\sigma_1$ 
    and $\sigma_2$ are the random times in 
    Subsection \ref{Subsec-Noise},   
    and the random times 
    $\tau_j$, $1\leq j\leq 5$, 
    will be determined below. 
    To ease notations, 
    we omit the dependence of $\omega$ 
    in the sequel. 
    
    Note that, 
    under the above bootstrap estimates,  
    the estimates in Section \ref{subsec cm} 
    are all valid after $\tau^*$ 
    and the corresponding constants are deterministic.  
To be precise, in view of  
		\eqref{5.3}, \eqref{5.4}, and the continuity of solutions in $H^1$,
		one can take 
        $t_* \in (\tau^*,t^*)$ 
        close to $t^*$, 
        such that the geometrical decomposition \eqref{3.3}
		and the following estimates hold on $[t_*,n]$:
		\begin{equation}
			\Vert \varepsilon_n(t)\Vert_{H^1}\le2\phi^{\frac{1}{2}}(\widetilde{\delta}t),\ \
			|\mathbf{a}_n^+(t)|\le 2\phi^{\frac{1}{2}}(\widetilde{\delta}t), \ \
			|\mathbf{a}_n^-(t)|\le 2\phi^{\frac{1}{2}+\frac{1}{4d}}(\widetilde{\delta}t) \label{5.8}
		\end{equation}
		\begin{equation}
			\sum\limits_{k=1}^{K}(|\alpha_{n,k}(t)-\alpha_k^0|+|\theta_{n,k}(t)-\theta_k^0|)\le 2t\phi^{\frac{1}{2}}(\widetilde{\delta}t).
			\label{5.9}
		\end{equation}  
        Then, let 
        \begin{equation} \label{tau0}
			\tau_1:=\inf\Big\{t>0: 2\phi^{\frac{1}{2}}(\widetilde{\delta}t)\le1, 
            \ B_*(t)+2t\phi^{\frac{1}{2}}(\widetilde{\delta}t)\le\frac{1}{10}{\rm min}\{1,w_k,\alpha_k^0\}\Big\}. 
		\end{equation}  
       By the definition \eqref{def-Ta-}, 
       $t_*> \tau^* \geq \tau_1$. 
       We thus  infer from estimates \eqref{5.8} 
    and \eqref{5.9} that  
    the upper bounds in \eqref{3.8} and \eqref{3.9} 
    are valid on $[t_*,n]$. 
     One also can apply Proposition \ref{prop3.4} to obtain
		\begin{equation}
			\sum\limits_{k=1}^{K}(|\dot\alpha_{n,k}(t)|+|\dot\theta_{n,k}(t)|)\le{C}\left(\Vert\varepsilon_n(t)\Vert_{H^1}+B_*(t)\phi(\delta_1t)+e^{-\delta_2t}\right),
			\ \ \forall t\in [t_*, n].
			\label{5.10}
		\end{equation}

		Below we consider Case (I) and Case (II) separately to derive the bootstrap estimates \eqref{5.5}-\eqref{5.7}.
		
		\medskip
		{\bf Case (I):}  
		First, by \eqref{5.8} and \eqref{5.9}, for any $t\in[t_*,n]$,
		\begin{align}
				\Vert u_n(t)-R(t)\Vert_{H^1}\le&\Vert R(t)-\sum\limits_{k=1}^K\widetilde{R}_{n,k}(t)\Vert_{H^1}+\Vert\varepsilon_n\Vert_{H^1}\nonumber\\
				\le&C\sum\limits_{k=1}^K(|\alpha_{n,k}(t)-\alpha_k^0|+|\theta_{n,k}(t)-\theta_k^0|)+\Vert\varepsilon_n\Vert_{H^1}\nonumber\\
				\le&2Cte^{-\frac{1}{2}\widetilde{\delta}t}+2e^{-\frac{1}{2}\widetilde{\delta}t},
				\label{5.21}
		\end{align}
		where $\wt R_{n,k}$ is the approximating soliton profile from the geometrical decomposition \eqref{3.4} with $\alpha_k(t)$, $\theta_k(t)$ replaced by $\alpha_{n,k}(t)$, $\theta_{n,k}(t)$.
		Setting
			 \begin{equation}\label{tau1}
				\tau_2:=\inf \Big\{t\geq 
                \frac{2}{\widetilde{\delta}}: 
                \ 2Cte^{-\frac{1}{2}\widetilde{\delta}t}+2e^{-\frac{1}{2}\widetilde{\delta}t}\le\frac{1}{4}\delta_*\Big\},
			\end{equation}  
        where $\delta_*$ the small number from Lemma \ref{lem6.3}. 
        Then, 
        since $\tau^* \geq \tau_1\vee\tau_2$,  
         estimate \eqref{5.5} is verified on $[t_*,n]$. 
		
		Regarding the remaining estimates \eqref{5.6} and \eqref{5.7}, 
		using \eqref{4.101}, \eqref{5.8} and the inequality 
        \begin{equation}\label{p2>}
            p\wedge2>1+\frac{1}{2d},\ \forall d\ge1,
        \end{equation}
		we have that for every $1\le k\le K$ and any $t\in[t_*,n]$,
		 \begin{align}
				|a_{n,k}^+(t) |\le&C\int_t^{+\infty}e^{-\frac{p\wedge2}{2}\widetilde{\delta}s}+B_*(s)e^{-\delta_1s}ds+Ce^{-\delta_2t} \notag \\ 
				\le& C\bigg(\frac{2}{(p\wedge2)\widetilde{\delta}}e^{\frac{1-p\wedge2}{2}\widetilde{\delta}t}+\frac{1}{\widetilde{\delta}}e^{-\frac{1}{2}\widetilde{\delta}t}+e^{-\frac{1}{2}\widetilde{\delta}t}\bigg)
				e^{-\frac{1}{2}\widetilde{\delta}t}\notag\\
                \le&C\bigg(\frac{3}{\widetilde{\delta}}e^{-\frac{1}{4d}\widetilde{\delta}t}+e^{-\frac{1}{4d}\widetilde{\delta}t}\bigg)e^{-\frac{1}{2}\widetilde{\delta}t}.   \label{5.17}
		\end{align} 
		
		Moreover,  for any $t\in[t_*,n]$, integrating \eqref{5.10} on $[t,n]$  and using \eqref{time T},
			\eqref{3.7} \eqref{bn-esti}, \eqref{5.8} we derive 
	 \begin{align}
				&\sum\limits_{k=1}^K(|\alpha_{n,k}(t)-\alpha_k^0|+|\theta_{n,k}(t)-\theta_k^0|)\nonumber\\
				\le&\sum\limits_{k=1}^K(|\alpha_{n,k}(t)-\alpha_{n,k}(n)|+|\alpha_{n,k}(n)-\alpha_k^0|+|\theta_{n,k}(t)-\theta_{n,k}(n)|+|\theta_{n,k}(n)-\theta_k^0|)\nonumber\\
				\le&C\int_{t}^{+\infty}e^{-\frac{1}{2}\widetilde{\delta}s}+B_*(s)e^{-\delta_1s}+e^{-\delta_2s}ds+C|\boldsymbol{b}_n|\nonumber\\
				\le&C\Big(\frac{2}{\widetilde{\delta}t}+\frac{1}{t}e^{-\frac{1}{4d}\widetilde{\delta}t}\Big) te^{-\frac{1}{2}\widetilde{\delta}t}.\label{5.18}
		\end{align} 
		
		At last, in view of Proposition \ref{coe},  
		estimates \eqref{bn-esti}, \eqref{5.8} and \eqref{p2>}, 
		we derive that for any $t\in[t_*,n]$,
		 \begin{align}
				\Vert\varepsilon_n(t)\Vert_{H^1}^2
				\le&C\int_t^{+\infty}\Big(\frac{1}{s}+B_*(s)\Big)e^{-\widetilde{\delta}s}ds+C\bigg(\int_t^{+\infty}e^{-\frac{p\wedge2}{2}\widetilde{\delta}s}ds\bigg)^2+C\bigg(\int_{t}^{+\infty}B_*(s)e^{-\widetilde{\delta}s}ds\bigg)^2\nonumber\\
				&+Ce^{-(1+\frac{1}{2d})\widetilde{\delta}t}+Ce^{-\delta_2t}+\beta\|\varepsilon_n(t)\|^2_{H^1}\nonumber\\[5pt]
				\le&C
				\bigg(\frac{1}{\widetilde{\delta}t}+\frac{1}{\widetilde{\delta}}B_*(t)+\frac{4}{(p\wedge2)^2\widetilde{\delta}^2}e^{(1-p\wedge2)\widetilde{\delta}t}+\frac{1}{\widetilde{\delta}^2}e^{-\widetilde{\delta}t}+e^{-\frac{1}{2d}\widetilde{\delta}t}\bigg)e^{-\widetilde{\delta}t}+\beta\|\varepsilon_n(t)\|^2_{H^1}\notag\\
                \le&C
				\bigg(\frac{1}{\widetilde{\delta}t}+\frac{1}{\widetilde{\delta}}B_*(t)+\frac{5}{\widetilde{\delta}^2}e^{-\frac{1}{4d}\widetilde{\delta}t}+e^{-\frac{1}{4d}\widetilde{\delta}t}\bigg)e^{-\widetilde{\delta}t}+\beta\|\varepsilon_n(t)\|^2_{H^1}.\label{5.19}
		\end{align} 
		 Since $\beta\to0$ as $\|\varepsilon_n(t)\|_{H^1}\to0$,
			there exists a deterministic small constant $0<\varepsilon^*<1$, such that  $\beta\le\frac{1}{2}$ 
            if $\|\varepsilon_n(t)\|_{H^1}\le\varepsilon^*$. 
			
            Thus, let 
			 \begin{equation}\label{tau2}
				\tau_3:=\inf \bigg\{ 
                t\geq  1+\frac{2}{\widetilde{\delta}}{\rm ln}\big(\frac{2}{\varepsilon^*}\big):
                \ C\bigg(\frac{2}{\widetilde{\delta}t}+\frac{1}{\widetilde{\delta}}B_*(t)+\Big(\frac{5}{\widetilde{\delta}^2}+\frac{3}{\widetilde{\delta}}+1\Big)e^{-\frac{\widetilde{\delta}t}{4d}}
				\bigg)\le\frac{1}{8}\bigg\}. 
			\end{equation} 
		Note that $\tau_3$ is finite almost surely, 
        due to the vanishing of the noise
			$B_*$ at infinity in \eqref{limB*}. 
		In view of 
        the above estimates \eqref{5.17}-\eqref{5.19} and 
        the definition of $\tau^*$, 
        we consequently get the improved estimates
			\eqref{5.6} and \eqref{5.7} on  $[t_*,n]$.

		\medskip
		{\bf Case (II):}
	  Let us first choose a deterministic large constant $\nu_0$ such that 
			\begin{equation}\label{def-v0}
				\nu_0\ge 4d,\ \ {\rm and}\ \  C\Big(\frac{2}{\nu_0-2}+\frac{\widetilde{\delta}}{\nu_0-1}+\frac{2}{(\nu_0-2)\widetilde{\delta}}+\frac{4}{(\nu_0-2)^2\widetilde{\delta}^2}\Big)\le\frac{1}{16},
			\end{equation} 
        where $C$ is the deterministic constant in estimates \eqref{5.22}-\eqref{5.24} below. In the following we consider any $\nu_*\geq \nu_0$ fixed.

        We  derive from \eqref{5.8} and \eqref{5.9} that,
		as in \eqref{5.21},  for any $t\in[t_*,n]$,
		\begin{align}
				\Vert u_n(t)-R(t)\Vert_{H^1}\le\Vert R(t)-\sum\limits_{k=1}^K\widetilde{R}_{n,k}(t)\Vert_{H^1}+\Vert\varepsilon_n(t)\Vert_{H^1}\le C(t+1)(\widetilde{\delta}t)^{-\frac{\nu_*}{2}}. 
				\label{5.28}
		\end{align}
		Then, let 
		 \begin{equation}\label{tau3}
				\tau_4:=\inf \Big\{t>0: 
                C(t+1)(\widetilde{\delta}t)^{-\frac{\nu_*}{2}}\le\frac{1}{4}\delta_*\Big\},
			\end{equation} 
			where $\delta_*$ is as in Lemma \ref{lem6.3}. 
        Since by the definition \eqref{def-Ta-}, 
       $\tau^*\ge\tau_1\vee \tau_4$, 
         we infer that 
         estimate \eqref{5.5} holds on $[t_*,n]$.

		Moreover, one has, via \eqref{4.101}, \eqref{5.8} and $p\wedge2>1$,
		 \begin{align}
				|a_{n,k}^+(t) |\le &C\bigg(\int_t^{+\infty}(\widetilde{\delta}s)^{-\frac{p\wedge2}{2}\nu_*}+B_*(s)(\widetilde{\delta}s)^{-\nu_*}ds+e^{-\delta_2t}\bigg)\nonumber\\
                \le&C\bigg(\frac{2}{(\nu_*-2)\widetilde{\delta}}(\widetilde{\delta}t)^{\frac{1-p\wedge2}{2}\nu_*+1}+\frac{1}{\nu_*}(\widetilde{\delta}t)^{-\frac{\nu_*}{2}}+(\widetilde{\delta}t)^{\frac{\nu_*}{2}}e^{-\delta_2t}\bigg)(\widetilde{\delta}t)^{-\frac{\nu_*}{2}}.
				\label{5.22}
		\end{align} 
		
		Estimating as in the proof of \eqref{5.18}, for any $t\in[t_*,n]$, integrating \eqref{5.10} on $[t,n]$ we get
		 \begin{align}
				\sum\limits_{k=1}^K(|\alpha_{n,k}(t)\!-\!\alpha_k^0|\!+\!|\theta_{n,k}(t)\!-\!\theta_k^0|)
				\le&C\Big(\int_{t}^{+\infty}(\widetilde{\delta}s)^{-\frac{\nu_*}{2}}+B_*(s)(\widetilde{\delta}s)^{-\nu_*}+e^{-\delta_2s}ds+(\widetilde{\delta}t)^{-(\frac{1}{2}+\frac{1}{4d})\nu_*}\Big) \label{5.23} \\
				\le&C\Big(\frac{2}{\nu_*-2}\!+\!\frac{\widetilde{\delta}}{\nu_*-1}(\widetilde{\delta}t)^{-\frac{\nu_*}{2}-1}\!+\!(\delta_2t)^{-1}(\widetilde{\delta}t)^{\frac{\nu_*}{2}}e^{-\delta_2t}\!+\!t^{-1}(\widetilde{\delta}t)^{-\frac{\nu_*}{4d}}\Big)t(\widetilde{\delta}t)^{-\frac{\nu_*}{2}}. \notag 				
		\end{align} 
		
		An application of Proposition \ref{coe}, \eqref{B*-bdd}, \eqref{bn-esti}, \eqref{5.8} and \eqref{p2>} also gives that for any $t\in[t_*,n]$,
	   \begin{align}
				\Vert\varepsilon_n(t)\Vert_{H^1}^2
				\le&C\int_t^{+\infty}s^{-1}(\widetilde{\delta}s)^{-\nu_*}ds+C\bigg(\int_t^{+\infty}(\widetilde{\delta}s)^{-\frac{p\wedge2}{2}\nu_*}ds\bigg)^2+C\bigg(\int_t^{+\infty}(\widetilde{\delta}s)^{-\nu_*-1}ds\bigg)^2\nonumber\\
				&+C(\widetilde{\delta}t)^{-(1+\frac{1}{2d})\nu_*}+Ce^{-\delta_2t}+\beta\|\varepsilon_n(t)\|^2_{H^1}  \label{5.24} \\
				\le&C\bigg(\frac{1}{\nu_*}\!+\!\frac{1}{\nu_*^2}(\widetilde{\delta}t)^{-\nu_*}\!+\!\frac{4}{(\nu_*\!-\!2)^2\widetilde{\delta}^2}(\widetilde{\delta}t)^{(1-(p\wedge2))\nu_*+2}\!+\!(\widetilde{\delta}t)^{-\frac{1}{2d}\nu_*}\!+\!(\widetilde{\delta}t)^{\nu_*}e^{-\delta_2t}\bigg)(\widetilde{\delta}t)^{-\nu_*}\!+\!\beta\|\varepsilon_n(t)\|^2_{H^1}. \notag 				
		\end{align}

			Thus, let 
			\begin{equation}\label{tau4}
				\tau_5:=\inf\bigg\{t\geq 1+\frac{\nu_*}{\delta_2}+\frac{1}{\widetilde{\delta}}\Big(\frac{2}{\varepsilon^*}\Big)^{\frac{2}{\nu_*}}:\   C\Big(\big(\widetilde{\delta}t\big)^{\nu_*}e^{-\delta_2t}+\big(\widetilde{\delta}t\big)^{-\frac{\nu_*}{4d}}\Big)\le\frac{1}{16}\bigg\},
		\end{equation}
        where $C$ is the deterministic constant from the above estimates \eqref{5.22}-\eqref{5.24}, and $\widetilde{\delta}$ is given by \eqref{3.51}. In view of estimates \eqref{5.22}-\eqref{5.24}, the definition of $\tau^*$ in \eqref{def-Ta-} and the choice of $\nu_*$ in \eqref{def-v0}, 
        we verify estimates \eqref{5.6} and \eqref{5.7} on $[t_*,n]$. 
        The proof is consequently complete. 
	\end{proof}

	\subsection{Topological arguments}
	\label{Subsection-an-}
	
	As mentioned above,  
	Proposition \ref{prop5.2} requires the a-priori control of $\mathbf{a}_n^-$ 
    so that the estimates of $\varepsilon_n$, $\alpha_n$, $\theta_n$ and $\mathbf{a}_n^+$ 
    can be improved.

	In order to control the unstable direction $\mathbf{a}_n^-$, 
	we use a topological argument 
	as in \cite{CMM11}. 
	It is necessary to keep valid 
	the estimate \eqref{5.3} of 
	$\mathbf{a}_n^-$, 
	this motivates the following definition 
	\begin{align}   \label{T0an-def}
		T_0(\boldsymbol{a}_n^-)
		:= \inf \{T\ge\tau^*:  \
		{\rm estimates}\  \eqref{4p}-\eqref{5.4}\
		{\rm hold}\
		{\rm on}\ [T,n]\}.
	\end{align}

	 Let $\nu_*\ge\nu_0$ with $\nu_0$ satisfying  \eqref{def-v0}.  
  We define a universal random time  $T_0$, 
    independent of $n$,  by
 \begin{equation}\label{def-T0}
		T_0:=\begin{cases}
			\max\{M, 
            \sigma_1, 
            \tau_j, j=1,2,3,6\},&\text{in Case (I)};\\
			\max\{M, 
            \sigma_1,\sigma_2, 
            \tau_j, 
            j=1,4,5,7\},&\text{in Case (II)},
		\end{cases}
	\end{equation} 
	where $M$ is the large deterministic time 
    from the geometrical decomposition in Proposition \ref{prop3.2}, 
$\sigma_1$ and $\sigma_2$ 
are the random times 
in Subsection \ref{Subsec-Noise}, 
$\tau_j$, $1\leq j\leq 5$, 
    are defined as in the previous subsection, 
	 \begin{equation}\label{def-S5}
		\tau_6:=\inf\Big\{t>0: 
        C_1e^{(\frac{1}{2}+\frac{1}{4d}-\frac{p\wedge2}{2})\widetilde{\delta}t}\le\frac{\widetilde{\delta}}{4}\Big\},
	\end{equation} 
	and
	 \begin{equation}\label{def-S6}
		\tau_7:=\inf\Big\{t\geq 
        \frac{3\nu_*}{2\delta_3}:
       C_2\Big((\widetilde{\delta}t)^{(\frac{1}{2}+\frac{1}{4d}-\frac{p\wedge2}{2})\nu_*}+(\widetilde{\delta}t)^{(\frac{1}{2}+\frac{1}{4d})\nu_*}e^{-\delta_2t}\Big)\le\frac{\delta_3}{4} \Big\} 
	\end{equation} 
    with $C_1$, $C_2$ being the deterministic constants in \eqref{5.31} and \eqref{N-C-II} below, respectively.

     \medskip 
     
     We aim to show that there exists an appropriate vector 
	$\boldsymbol{a}_n^- \in B_{\mathbb{R}^K}(\phi^{\frac{1}{2}+\frac{1}{4d}}(\widetilde{\delta}n))$
	such that
	$T_0(\boldsymbol{a}_n^-)$ is 
    less than the universal time $T_0$, 
    i.e., $T_0(\boldsymbol{a}_n^-) \leq T_0$.  
    This is important in the next section 
    to pass to the limit 
    of the approximating solutions 
    to construct the desired 
    stochastic multi-solitons.

	\begin{proposition} [Uniform backward time]  \label{prop5.4}
	Let $\widetilde{\delta}$, $v_*$, $\tau^*$ be as in Proposition \ref{prop5.2}, and $T_0(\boldsymbol{a}_n^-)$, $T_0$ defined
		as above. 
		Then, 
        for $\mathbb{P}$-a.e. $\omega\in \Omega$  there exists $N_0(\omega)\ge1$ 
		such that for any $n\ge N_0(\omega)$,
		there exists $\boldsymbol{a}_n^-(\omega)\in B_{\mathbb{R}^K}(\phi^{\frac{1}{2}+\frac{1}{4d}}(\widetilde{\delta}n))$
		such that 
        $T_0(\boldsymbol{a}_n^-(\omega))
        \le T_0(\omega)$. 
	\end{proposition}

	\begin{proof}
		We shall prove this by contradiction.
		Suppose that 
        there exists a measurable set $\Omega'\subseteq \Omega$ 
        with positive probability, 
        such that for every $\omega\in \Omega'$ 
        and 
        for any $N_0$ large enough, there exists $n_0(\omega) \ge N_0$, such that
		$T_0(\boldsymbol{a}_{n_0}^-)(\omega)
        >T_0(\omega)$
		for all $\boldsymbol{a}_{n_0}^-\in B_{\mathbb{R}^K}(\phi^{\frac{1}{2}+\frac{1}{4d}}(\widetilde{\delta}{n_0}(\omega)))$. 
        Below we fix $\omega\in \Omega'$ 
         and take $N_0(\omega)>T_0(\omega)$. We omit the dependence of $\omega$ 
        to simplify the notations. 
		
		In view of Proposition \ref{prop3.5},
		for $N_0$ possibly larger such that $\phi^{\frac{1}{2}+\frac{1}{4d}}(\widetilde{\delta}{N_0})\le r$, 
		we infer that for any $\boldsymbol{a}_{n_0}^-\in B_{\mathbb{R}^K}\big(\phi^{\frac{1}{2}+\frac{1}{4d}}(\widetilde{\delta}{n_0})\big)$,
		there exists a unique vector $\boldsymbol{b}_{n_0}\in\mathbb{R}^{2K}$ such that
		\begin{equation}\label{b-a-n0^-}
			\mathbf{a}_{n_0}^+({n_0}) = {\bf 0},\quad
			\mathbf{a}_{n_0}^-({n_0}) = \boldsymbol{a}_{n_0}^-,\quad {\rm and}\quad  |\boldsymbol{b}_{n_0}|\le C\phi^{\frac{1}{2}+\frac{1}{4d}}(\widetilde{\delta}{n_0}).
		\end{equation}

		By the definition of $T_0(\boldsymbol{a}_{n_0}^-)$,
		$u_{n_0}$ admits the geometrical decomposition \eqref{3.3} 
		and estimates \eqref{4p}-\eqref{5.4} on $(T_0(\boldsymbol{a}_{n_0}^-), {n_0}]$.
		 It is still valid on the closed interval $[T_0(\boldsymbol{a}_{n_0}^-), {n_0}]$, 
		due to the continuity of solutions 
		and modulation parameters.
		
		 Moreover, by Proposition \ref{prop5.2} and the definition of $T_0$ in \eqref{def-T0},
		one has the improved estimates (\ref{5.5})-(\ref{5.7}) on $[T_0(\boldsymbol{a}_{n_0}^-), {n_0}]$.
		In particular,
		\begin{align*}
			&\Vert u_{n_0}(T_0(\boldsymbol{a}_{n_0}^-))-R(T_0(\boldsymbol{a}_{n_0}^-))\Vert_{H^1}\le\frac{\delta_*}{4},\quad \Vert \varepsilon_{n_0}(T_0(\boldsymbol{a}_{n_0}^-))\Vert_{H^1}\le\frac{1}{2}\phi^{\frac{1}{2}}(\widetilde{\delta}T_0(\boldsymbol{a}_{n_0}^-)),\\
			&\sum\limits_{k=1}^K(|\alpha_{n_0,k}(T_0(\boldsymbol{a}_{n_0}^-))-\alpha_k^0|+|\theta_{n_0,k}(T_0(\boldsymbol{a}_{n_0}^-))-\theta_k^0|)\le\frac{1}{2}T_0(\boldsymbol{a}_{n_0}^-)\phi^{\frac{1}{2}}(\widetilde{\delta}T_0(\boldsymbol{a}_{n_0}^-)),\\
			&|\mathbf{a}_{n_0}^+(T_0(\boldsymbol{a}_{n_0}^-))|\le \frac{1}{2}\phi^{\frac{1}{2}}(\widetilde{\delta}T_0(\boldsymbol{a}_{n_0}^-)).
		\end{align*}
		
		Note that
		\begin{equation}
			|\mathbf{a}_{n_0}^-(T_0(\boldsymbol{a}_{n_0}^-))|=\phi^{\frac{1}{2}+\frac{1}{4d}}(\widetilde{\delta}T_0(\boldsymbol{a}_{n_0}^-)).
			\label{5.29}
		\end{equation}
		That is,
		$\mathbf{a}_{n_0}^-(T_0(\boldsymbol{a}_{n_0}^-))$ is in 
		the sphere
		$S_{\mathbb{R}^K}\big(\phi^{\frac{1}{2}+\frac{1}{4d}}(\widetilde{\delta}T_0
		(\boldsymbol{a}_{n_0}^-))\big)$. 
		In fact, if $|\mathbf{a}_{n_0}^-(T_0(\boldsymbol{a}_{n_0}^-))|<\phi^{\frac{1}{2}+\frac{1}{4d}}(\widetilde{\delta}T_0(\boldsymbol{a}_{n_0}^-))$, 
        then by the above estimates and the continuity of modulation parameters and remainder,
		one can find a small time $\eta>0$ such that estimates \eqref{4p}-\eqref{5.4} hold on $[T_0(\boldsymbol{a}_{n_0}^-)-\eta,{n_0}]$,
		which contradicts the definition of $T_0(\boldsymbol{a}_{n_0}^-)$.

		\medskip
		Now,
		we define the  map by
		\begin{align*}
			\Lambda:\quad B_{\mathbb{R}^K}\big(\phi^{\frac{1}{2}+\frac{1}{4d}}(\widetilde{\delta}{n_0})\big)
			&\to S_{\mathbb{R}^K}\big(\phi^{\frac{1}{2}+\frac{1}{4d}}(\widetilde{\delta}{n_0})\big),\\
			\boldsymbol{a}^-_{n_0}
			&\mapsto \phi^{\frac{1}{2}+\frac{1}{4d}}(\widetilde{\delta}{n_0})\phi^{-(\frac{1}{2}+\frac{1}{4d})}(\widetilde{\delta}T_0(\boldsymbol{a}_{n_0}^-))\mathbf{a}_{n_0}^-(T_0(\boldsymbol{a}_{n_0}^-)).
		\end{align*}
		It follows from \eqref{5.29} that
		$\Lambda$ is well defined.
		
		Below we show that
		$\Lambda$ is continuous,
		and it is  the identity map
		when restricted to $S_{\mathbb{R}^K}(\phi^{\frac{1}{2}+\frac{1}{4d}}(\widetilde{\delta}{n_0}))$.
		Assuming these to hold, 
        we then infer that the continuous map $\widetilde{\Lambda}:=-\Lambda$ from $B_{\mathbb{R}^K}\big(\phi^{\frac{1}{2}+\frac{1}{4d}}(\widetilde{\delta}{n_0})\big)$ to $S_{\mathbb{R}^K}\big(\phi^{\frac{1}{2}+\frac{1}{4d}}(\widetilde{\delta}{n_0})\big)$ has no fixed point, 
        which contradicts the Brouwer fixed point theorem. 
        In fact, 
        since the image of $\widetilde{\Lambda}$ 
        is in $S_{\mathbb{R}^K}\big(\phi^{\frac{1}{2}+\frac{1}{4d}}(\widetilde{\delta}{n_0})\big)$, 
        it is clear that  
        \begin{equation*} 
            \widetilde{\Lambda}(\boldsymbol{a}^-_{n_0})\neq\boldsymbol{a}^-_{n_0},\ \forall \boldsymbol{a}^-_{n_0}\in \mathring B_{\mathbb{R}^K}\big(\phi^{\frac{1}{2}+\frac{1}{4d}}(\widetilde{\delta}{n_0})\big).
        \end{equation*}
        Moreover,
        since $\Lambda$ is the identity map
		when restricted to $S_{\mathbb{R}^K}(\phi^{\frac{1}{2}+\frac{1}{4d}}(\widetilde{\delta}{n_0}))$, we have
        \begin{equation*} 
            \widetilde{\Lambda}(\boldsymbol{a}^-_{n_0})=-\Lambda(\boldsymbol{a}^-_{n_0})=-\boldsymbol{a}^-_{n_0}\neq\boldsymbol{a}^-_{n_0},\ \forall \boldsymbol{a}^-_{n_0}\in S_{\mathbb{R}^K}\big(\phi^{\frac{1}{2}+\frac{1}{4d}}(\widetilde{\delta}{n_0})\big).
        \end{equation*}
        The above two facts  together then yield 
        that   $\widetilde{\Lambda}$ has no fixed point, 
        as claimed.

\medskip 
        Below let us start with the proof of the continuity of the map $\Lambda$.

		{\bf $(i)$. Continuity of 
			$\Lambda$}: It suffices to prove that 	
		the map $\boldsymbol{a}_{n_0}^-\mapsto T_0(\boldsymbol{a}_{n_0}^-)$ is continuous. For this purpose,
		we fix  $\boldsymbol{a}_{n_0}^-\in{B}_{\mathbb{R}^K}(\phi^{(\frac{1}{2}+\frac{1}{4d})}(\widetilde{\delta}{n_0}))$ and let $\widehat{T}(\boldsymbol{a}_{n_0}^-)\in(T_0,T_0(\boldsymbol{a}_{n_0}^-))$ be close enough to $T_0(\boldsymbol{a}_{n_0}^-)$ such that estimates \eqref{5.8} and \eqref{5.9} with $n$ replaced by $n_0$ hold  on $[\widehat{T}(\boldsymbol{a}_{n_0}^-),{n_0}]$.
		Let
		\begin{equation}
			\mathcal{N}(t,\boldsymbol{a}_{n_0}^-):=|\phi^{-(\frac{1}{2}+\frac{1}{4d})}(\widetilde{\delta}t)\mathbf{a}_{n_0}^-(t)|^2,
			\quad
			t\in[\widehat{T}(\boldsymbol{a}_{n_0}^-),{n_0}].
			\label{5.30}
		\end{equation}
		
		Below we mainly consider the polynomial decay rate in the stochastic {\rm Case (II)},
		as {\rm Case (I)} is easier and can be treated in an analogous manner.

		\medskip
		{\bf Case (II):} 
		Recall that $\phi(x)=|x|^{-\nu_*}$ 
		is the spatial decay function of noise given by \eqref{1.17} in {\rm Case (II)}. 
		By straightforward computations, \eqref{3.61} and \eqref{def-delta3}, we have 
		for any  $t\in[\widehat{T}(\boldsymbol{a}_{n_0}^-),{n_0}]$,
		\begin{align}\label{N-C-II}
			\frac{d\mathcal{N}}{dt}(t,\boldsymbol{a}_{n_0}^-)=&\frac{d}{dt}\bigg((\widetilde{\delta}t)^{(1+\frac{1}{2d})\nu_*}|\mathbf{a}_{n_0}^-(t)|^2\bigg)\nonumber\\
			=&\big(1+(2d)^{-1}\big)\widetilde{\delta}\nu_*\big(\widetilde{\delta}t\big)^{(1+\frac{1}{2d})\nu_*-1}|\mathbf{a}_{n_0}^-(t)|^2+2\big(\widetilde{\delta}t\big)^{(1+\frac{1}{2d})\nu_*}\sum\limits_{k=1}^K\Big(a_{n_0,k}^-(t)\frac{d}{dt}a_{n_0,k}^-(t)\Big)\nonumber\\
			\le&\big(1+(2d)^{-1}\big)\nu_*t^{-1}\mathcal{N}(t,\boldsymbol{a}_{n_0}^-)+2(\widetilde{\delta}t)^{(1+\frac{1}{2d})\nu_*}\sum\limits_{k=1}^{K}\Big(-e_0(w_k)^{-2}\big(a_{n_0,k}^-(t)\big)^2\nonumber\\
			&\quad +C\big|a_{n_0,k}^-(t)\big|\big((\widetilde{\delta}t)^{-\frac{p\wedge2}{2}\nu_*}+t^{-1}(\widetilde{\delta}t)^{-\nu_*}+e^{-\delta_2t}\big)\Big)\nonumber\\
			\le&\Big(\frac{3}{2}\nu_*t^{-1}-2\delta_3\Big)\mathcal{N}(t,\boldsymbol{a}_{n_0}^-)+C_2\big((\widetilde{\delta}t)^{(\frac{1}{2}+\frac{1}{4d}-\frac{p\wedge2}{2})\nu_*}+(\widetilde{\delta}t)^{(\frac{1}{2}+\frac{1}{4d})\nu_*}e^{-\delta_2t}\big)\sqrt{\mathcal{N}(t,\boldsymbol{a}_{n_0}^-)},
		\end{align}
		which along with \eqref{def-T0} and \eqref{def-S6} yields that for any $t\in[\widehat{T}(\boldsymbol{a}_{n_0}^-),{n_0}]$,
		\begin{equation} \label{dN-N-esti}
			\frac{d\mathcal{N}}{dt}(t,\boldsymbol{a}_{n_0}^-)\le-\delta_3\mathcal{N}(t,\boldsymbol{a}_{n_0}^-)+\frac{\delta_3}{4}\sqrt{\mathcal{N}(t,\boldsymbol{a}_{n_0}^-)}.
		\end{equation}
		
		\medskip 
		We next claim that for $\eta>0$ small enough,
		there exists $\delta>0$ such that
		\begin{equation}
			\mathcal{N}(t,\boldsymbol{a}_{n_0}^-)<1-\delta,\quad \forall t\in[T_0(\boldsymbol{a}_{n_0}^-)+\eta,{n_0}],
			\label{5.332}
		\end{equation}
		and
		\begin{equation}
			\mathcal{N}(t,\boldsymbol{a}_{n_0}^-)>1+\delta,\quad \forall t\in[\widehat{T}(\boldsymbol{a}_{n_0}^-),T_0(\boldsymbol{a}_{n_0}^-)-\eta].
			\label{5.331}
		\end{equation}
		In the case where $T_0(\boldsymbol{a}_{n_0}^-)+\eta>{n_0}$  
		we only consider \eqref{5.331}, which may happen if 
		$\boldsymbol{a}_{n_0}^-\in S_{\mathbb{R}^K} ((\widetilde{\delta}{n_0})^{-(\frac{1}{2}+\frac{1}{4d})\nu_*}
		)$, see the next step $(ii)$ below.  
		
		Below 
		we prove \eqref{5.332}, and the proof of \eqref{5.331} is similar. 
		To this end, 
		we argue by contradiction and assume that there exists $\eta_*>0$ such that for any $m\geq 1$,
		there exists $t_m\in[T_0(\boldsymbol{a}_{n_0}^-)+\eta_*,{n_0}]$,
		such that
		\begin{equation}\label{Nt>1-}
			\mathcal{N}(t_m,\boldsymbol{a}_{n_0}^-)\ge 1-\frac 1m.
		\end{equation}
		By compactness,
		there exists a subsequence (still denoted by $\{m\}$) such that $t_m\to t_0$ as $m\to+\infty$  and $t_0\in[T_0(\boldsymbol{a}_{n_0}^-)+\eta_*,{n_0}]$. 
        Then passing to the limit 
        $m\to +\infty$ in \eqref{Nt>1-}  we get
		$
		    \mathcal{N}(t_0,\boldsymbol{a}_{n_0}^-)\ge1.
		$
		But 
        by the definition of $T_0(\boldsymbol{a}_{n_0}^-)$ in \eqref{T0an-def}, 
        $\mathcal{N}(t,\boldsymbol{a}_{n_0}^-)\le1$  
        for any $t\ge T_0(\boldsymbol{a}_{n_0}^-)$. 
		It thus follows that $\mathcal{N}(t_0,\boldsymbol{a}_{n_0}^-)=1$. 
		Then, taking $t=t_0$ in \eqref{dN-N-esti} we obtain
		\begin{equation*}
			\frac{d\mathcal{N}}{dt}(t_0,\boldsymbol{a}_{n_0}^-)\le-\frac{3}{4}\delta_3<0.
		\end{equation*}
		This yields that  
        $\mathcal{N}(\widetilde{t}_0,\boldsymbol{a}_{n_0}^-)>1$ 
        for some $\widetilde{t}_0\in [T_0(\boldsymbol{a}_{n_0}^-),t_0]$,  which 
        however contradicts 
        the fact that 
        $\mathcal{N}(\widetilde{t}_0,\boldsymbol{a}_{n_0}^-)\leq 1$, thereby proving \eqref{5.332}, as claimed.

		\medskip 
		Now, since by \eqref{5.30}, 
		$\mathcal{N}(t,\boldsymbol{a}_{n_0}^-)$ is continuous in $\mathbf{a}_{n_0}^-(t)$,
		and for all $t\in[\widehat{T}(\boldsymbol{a}_{n_0}^-),{n_0}]$, $\mathbf{a}_{n_0}^-(t)$ is continuous in $\boldsymbol{a}_{n_0}^-$ due to Proposition \ref{prop3.5} and the continuity of the flow of \eqref{5.1}.
		It follows that $\mathcal{N}(t,\boldsymbol{a}_{n_0}^-)$ is continuous in $\boldsymbol{a}_{n_0}^-$.
		Moreover,
		using the contradiction assumption that $T_0(\boldsymbol{a}_{n_0}^-)
        >T_0$, Proposition \ref{prop3.5} and the continuity of the flow of \eqref{5.1} again we get that there exists $\zeta(>0)$ 
        small enough, 
        such that for any $\widetilde{\boldsymbol{a}}_{n_0}^-\in B_{\mathbb{R}^K}(\boldsymbol{a}_{n_0}^-,\zeta)$,
		one has $\widehat{T}(\widetilde{\boldsymbol{a}}_{n_0}^-)\le T_0(\boldsymbol{a}_{n_0}^-)$. Thus, $\mathcal{N}(T_0(\boldsymbol{a}_{n_0}^-),\widetilde{\boldsymbol{a}}_{n_0}^-)$ is well-defined.
		Then, by the continuity of  
        the map 
        $\widetilde{\boldsymbol{a}}_{n_0}^- 
        \mapsto \mathcal{N}(T_0(\boldsymbol{a}_{n_0}^-),\widetilde{\boldsymbol{a}}_{n_0}^-)$, 
		there exists $\zeta=\zeta(\delta,T_0(\boldsymbol{a}_{n_0}^-))$ possibly smaller with $\delta$ as in \eqref{5.332} and \eqref{5.331}, 
		such that for any $\widetilde{\boldsymbol{a}}_{n_0}^-\in B_{\mathbb{R}^K}(\boldsymbol{a}_{n_0}^-,\zeta)$, one has 
        $|\mathcal{N}(T_0(\boldsymbol{a}_{n_0}^-),\widetilde{\boldsymbol{a}}_{n_0}^-)-\mathcal{N}(T_0(\boldsymbol{a}_{n_0}^-),\boldsymbol{a}_{n_0}^-)|=|\mathcal{N}(T_0(\boldsymbol{a}_{n_0}^-),\widetilde{\boldsymbol{a}}_{n_0}^-) - 1|\le\frac{\delta}{2}$.
		Taking into account \eqref{5.332} and \eqref{5.331} 
		with $\boldsymbol{a}_{n_0}^-$ replaced by $\widetilde{\boldsymbol{a}}_{n_0}^-$, we thus derive that $|T_0(\widetilde{\boldsymbol{a}}_{n_0}^-)-T_0(\boldsymbol{a}_{n_0}^-)|<\eta$ 
        for any $\widetilde{\boldsymbol{a}}_{n_0}^-\in B_{\mathbb{R}^K}(\boldsymbol{a}_{n_0}^-,\zeta)$. 
		This gives the continuity of the map
		$\boldsymbol{a}_{n_0}^-\mapsto T_0(\boldsymbol{a}_{n_0}^-)$ in {\rm Case (II)}.

		\medskip
		{\bf Case (I)}: 
		In this case, 
		we have from  \eqref{1.17}  that 
		$\phi(x)=e^{-|x|}$. 
		For any {$t\in[\widehat{T}(\boldsymbol{a}_{n_0}^-),{n_0}]$},
		one can replace \eqref{N-C-II} by the following estimate
		\begin{align}
			\frac{d\mathcal{N}}{dt}(t,\boldsymbol{a}_{n_0}^-)=&\frac{d}{dt}\bigg(e^{(1+\frac{1}{2d})\widetilde{\delta}t}|\mathbf{a}_{n_0}^-(t)|^2\bigg)\nonumber\\
			=&\Big(1+\frac{1}{2d}\Big)\widetilde{\delta}e^{(1+\frac{1}{2d})\widetilde{\delta}t}|\mathbf{a}_{n_0}^-(t)|^2+2e^{(1+\frac{1}{2d})\widetilde{\delta}t}\sum\limits_{k=1}^{K}\Big(a_{n_0,k}^-(t)\frac{d}{dt}a_{n_0,k}^-(t)\Big)\nonumber\\
			\le&\Big(1+\frac{1}{2d}\Big)\widetilde{\delta}\mathcal{N}(t,\boldsymbol{a}_{n_0}^-)+2e^{(1+\frac{1}{2d})\widetilde{\delta}t}\sum\limits_{k=1}^{K}\Big(-e_0\left(w_k\right)^{-2}\big(a_{n_0,k}^-(t)\big)^2\nonumber\\
			&\qquad +C\big|a_{n_0,k}^-(t)\big|\big(e^{-\frac{p\wedge2}{2}\widetilde{\delta}t}+e^{-\delta_1t}+e^{-\delta_2t}\big)\Big)\nonumber\\
			\le&-\widetilde{\delta}\mathcal{N}(t,\boldsymbol{a}_{n_0}^-)+C_1e^{(\frac{1}{2}+\frac{1}{4d}-\frac{p\wedge2}{2})\widetilde{\delta}t}\sqrt{\mathcal{N}(t,\boldsymbol{a}_{n_0}^-)}\nonumber\\
			\le&-\widetilde{\delta}\mathcal{N}(t,\boldsymbol{a}_{n_0}^-)+\frac{\widetilde{\delta}}{4}
			\sqrt{\mathcal{N}(t,\boldsymbol{a}_{n_0}^-)}.\label{5.31}
		\end{align}  
		Then, similar arguments as in {\rm Case (II)}
		lead to the continuity of  the map
		$\boldsymbol{a}_{n_0}^-\mapsto T_0(\boldsymbol{a}_{n_0}^-)$.

		\medskip
		
		{\bf $(ii)$. Identity of $\Lambda$ 
			when restricted to sphere}:  It remains to prove that
		$\Lambda$ is the identity map
		when restricted to $S_{\mathbb{R}^K}(\phi^{\frac{1}{2}+\frac{1}{4d}}(\widetilde{\delta}{n_0}))$.
		
		To this end,
		for any $\boldsymbol{a}_{n_0}^-\in S_{\mathbb{R}^K}(\phi^{\frac{1}{2}+\frac{1}{4d}}(\widetilde{\delta}{n_0}))$, 
		using \eqref{b-a-n0^-} we have
		\begin{equation} \label{N-1}
			\mathcal{N}({n_0},\boldsymbol{a}_{n_0}^-)
			=| \phi^{-(\frac{1}{2}+\frac{1}{4d})}(\widetilde{\delta}{n_0})\mathbf{a}_{n_0}^-({n_0})|^2
			=| \phi^{-(\frac{1}{2}+\frac{1}{4d})}(\widetilde{\delta}{n_0})
			\boldsymbol{a}_{n_0}^-|^2
			=1.
		\end{equation}
	  Moreover, 
		 letting $t=n_0$ in \eqref{5.31} and \eqref{dN-N-esti} we obtain  
		\begin{equation} \label{dN-0}
			\frac{d\mathcal{N}}{dt}({n_0},\boldsymbol{a}_{n_0}^-)<0
		\end{equation} 
        in both Case (I) and Case (II). 
		Suppose that $T_0(\boldsymbol{a}_{n_0}^-)<{n_0}$, then 
		\eqref{N-1} and \eqref{dN-0} 
		imply that there exists $t\in(T_0(\boldsymbol{a}_{n_0}^-),{n_0})$ such that $\mathcal{N}(t,\boldsymbol{a}_{n_0}^-)>1$. 
		But by the definition of $T_0(\boldsymbol{a}_{n_0}^-)$ 
		in \eqref{T0an-def}, 
		$\mathcal{N}(t,\boldsymbol{a}_{n_0}^-)\leq 1$ 
		for any $t\in[T_0(\boldsymbol{a}_{n_0}^-),{n_0}]$. 
		This leads to a contradiction. 
		Thus, we get $T_0(\boldsymbol{a}_{n_0}^-)={n_0}$.
		
		Consequently, by the definition of $\Lambda$ and \eqref{b-a-n0^-}, 
$\Lambda(\boldsymbol{a}_{n_0}^-)=\boldsymbol{a}_{n_0}^-$
		for all $\boldsymbol{a}_{n_0}^-\in S_{\mathbb{R}^K}(\phi^{\frac{1}{2}+\frac{1}{4d}}(\widetilde{\delta}{n_0}))$,
		which shows that
		the map $\Lambda$ restricted to $S_{\mathbb{R}^K}(\phi^{\frac{1}{2}+\frac{1}{4d}}(\widetilde{\delta}{n_0}))$ is the identity map.

		Therefore,
		the proof of Proposition \ref{prop5.4}
		is complete.
	\end{proof}

	\subsection{Proof of uniform estimate}
	Now, we are in position to prove the uniform estimate in Theorem \ref{th5.1}.
	
	\begin{proof}[Proof of Theorem \ref{th5.1}]
		By Propositions \ref{prop3.5} and \ref{prop5.4},  
		for $\mathbb{P}$-a.e. $\omega\in \Omega$ and for $n=n(\omega)$ large enough,
		there exist $\boldsymbol{a}_n^-(\omega)\in B_{\mathbb{R}^K}(\phi^{\frac{1}{2}+\frac{1}{4d}}(\widetilde{\delta}n))$ and
		a unique $\boldsymbol{b}_n(\omega)
        =\boldsymbol{b}_n(\boldsymbol{a}_n^-(\omega))$,
		such that 
        $T_0(\boldsymbol{a}_n^-(\omega))\le T_0(\omega)$, 
       	where $T_0(\omega)$ is independent of $n$. 
        Thus, by the definition of $T_0(\boldsymbol{a}_n^-(\omega))$ in \eqref{T0an-def}, $u_n(\omega)$ admits the geometrical decomposition \eqref{3.3} and estimates \eqref{5.3}, \eqref{5.4} on $[T_0(\omega), n]$.

		Regarding the error estimate \eqref{5.2}, we see that in Case (I), for any $t\in[T_0(\omega), n]$, by \eqref{5.3} and \eqref{5.4},
		\begin{align*}
			\Vert u_n(t,\omega)
            -R(t,\omega)\Vert_{H^1}
            \le&\Vert R(t,\omega)-\widetilde{R}_n(t,\omega)\Vert_{H^1}+\Vert\varepsilon_n(t,\omega)\Vert_{H^1}\nonumber\\
			\le&C\sum\limits_{k=1}^K(|\alpha_{n,k}(t,\omega)-\alpha_k^0|
            +|\theta_{n,k}(t,\omega)-\theta_k^0|)
            +\Vert\varepsilon_n(t,\omega)\Vert_{H^1}\nonumber\\
			\le&Cte^{-\frac{1}{2}\widetilde{\delta}t}.
		\end{align*}
		Moreover, in Case (II), for any $t\in[T_0(\omega), n]$,
		\begin{align*}
			\Vert u_n(t,\omega)-R(t,\omega)\Vert_{H^1}
			\le C\sum\limits_{k=1}^K(|\alpha_{n,k}(t,\omega)-\alpha_k^0|
            +|\theta_{n,k}(t,\omega)-\theta_k^0|)
            +\Vert\varepsilon_n(t,\omega)\Vert_{H^1}
			\le C t(\widetilde{\delta}t)^{-\frac{\nu_*}{2}}.
		\end{align*}
		Thus, estimate \eqref{5.2} is verified in both Case(I) and Case (II).
		The proof of Theorem \ref{th5.1} is complete.
	\end{proof}

	\section{Proof of main results}
	\label{Sec-Proof-Main}
	
	We are now ready to prove the main results in
	Theorems  \ref{th1.3} and \ref{th1.6}.
	
	\begin{proof}[Proof of Theorem \ref{th1.6}]
		Let us fix $\omega\in \Omega$ as in Theorem \ref{th5.1} 
        and omit it in the following arguments 
        to ease notations. 
        Let $\{u_n\}$ be the approximating solutions to \eqref{5.1}. 
		By virtue of Theorem \ref{th5.1} and the expressions of 
		$R(t)$ and $R_k(t)$ in \eqref{sumRk} and \eqref{1.11},  
        respectively, 
		we derive that  
		\begin{equation}\label{bound-un}
			\|u_{n}(t)\|_{H^1}\le
			\|u_n(t)-R(t)\|_{H^1}+\|R(t)\|_{H^1}\le Ct\phi^{\frac{1}{2}}(\widetilde{\delta}t)+\sum\limits_{k=1}^K\|R_k(t)\|_{H^1}
			\le C,\quad \forall t\in[T_0,n],
		\end{equation} 
		where 
		$C$ depends on $w_k$ and $\widetilde{\delta}$ 
		and is independent of $n$ and $t$. 
		It follows that $\{u_{n}(t)\}$ is uniformly bounded in $H^1(\mathbb{R}^d)$.
		In particular, letting $t=T_0$ 
        we obtain a subsequence (still denoted by $\{u_n(T_0)\}$) such that
		\begin{equation}
			u_n(T_0)\rightharpoonup u_0~~{\rm in}~~H^1(\mathbb{R}^d), 
			\quad {\rm as}~~n\to\infty,
			\label{wH}
		\end{equation}
		for some $u_0\in H^1(\mathbb{R}^d)$.

		\medskip 
		{\bf Claim:} 
		One has the strong convergence of 
		$u_n(T_0)$ in $L^2(\mathbb{R}^d)$, i.e.,
		\begin{equation}
			u_n(T_0)\to u_0~~{\rm in}~~L^2(\mathbb{R}^d),
			\quad {\rm as}~~n\to\infty.
			\label{sL}
		\end{equation}
		
		To this end,
		for any $\eta>0$,
		since $u_0\in H^1(\mathbb{R}^d)$,
		we can take $A_0$ large enough such that
		\begin{equation}
			\int_{|x|\ge A_0}|u_0(x)|^2dx\le\frac{\eta}{8}.
			\label{small1}
		\end{equation}
		
		By Proposition \ref{prop5.4}, estimates \eqref{4p}-\eqref{5.4} hold on the time interval $[T_0,n]$.
		For $n$ large enough, we fix $\widetilde{T}_0\in(T_0,n]$ independent of $n$ such that
		\begin{equation}
			\|\varepsilon_n(\widetilde{T}_0)\|_{H^1}^2\le\phi(\widetilde{\delta}\widetilde{T}_0)\le\frac{\eta}{16}.
			\label{small2}
		\end{equation}
		Then, we take $A_1$ large enough such that for $|x|\ge A_1$ and every $1\le k\le K$,
		\begin{equation*}
			\inf\limits_{n\ge1}|x-v_k\widetilde{T}_0-\alpha_{n,k}(\widetilde{T}_0)|\ge|x|-|v_k|\widetilde{T}_0-\sup\limits_{n\ge1,t\ge T_0}|\alpha_{n,k}(t)|\ge A_1-\frac{1}{2}A_1=\frac{1}{2}A_1,
		\end{equation*}
		and, via the exponential decay of the ground state in \eqref{1.7}, we may take $A_1$ larger such that
		\begin{equation}
			\sup\limits_{n\ge1}\int_{|x|\ge A_1}|\widetilde{R}_n(\widetilde{T}_0)|^2dx\le C\sum\limits_{k=1}^K\int_{|x|\ge\frac{A_1}{2w_k}}e^{-2\delta|x|}dx\le C\sum\limits_{k=1}^Ke^{-\frac{\delta A_1}{w_k}}\le\frac{\eta}{16},
			\label{small3}
		\end{equation}
		where $C$ depends on $A_1$, $w_k$ and $\delta$.
		
		Combining \eqref{small2} and \eqref{small3} we obtain
		\begin{equation}
			\sup\limits_{n\geq 1}
			\int_{|x|\ge A_1}|u_n(\widetilde{T}_0)|^2dx\le2\sup\limits_{n\ge1}\int_{|x|\ge A_1}|\widetilde{R}_n(\widetilde{T}_0)|^2dx
			+2\|\varepsilon_n(\widetilde{T}_0)\|^2_{H^1}\le\frac{\eta}{4}.
			\label{small4}
		\end{equation}
		
		Moreover, let $g(x)\in C^1(\mathbb{R})$ be such that $0\le g(x)\le1$,
		$g(x)=0$ for $|x|\le\frac{1}{2}$,
		$g(x)=1$ for $|x|\ge1$,
		and $|g'(x)|\le2$ for $x\in\mathbb{R}$.
		Let $g_{A(\eta)}:=g(\frac{|x|}{A(\eta)})$,
		where $A(\eta)$ is a constant to be determined later.
		
		By the integration-by-parts formula and the boundness of $B_*(t)$, $\|u_n\|_{H^1}$ in \eqref{3.7} 
		and \eqref{bound-un},
		there exists a positive constant $C_1$ such that for any $t\in[T_0,n]$,
		\begin{align*}
			\bigg|\frac{d}{dt}\!\int g_{A(\eta)}|u_n(t)|^2dx\bigg|\!=\!\bigg|2\text{Im}\!\int\!g'_{A(\eta)}(\partial_1u_n)\bar{u}_ndx+2\text{Re}\!\sum\limits_{l=1}^N\int_t^{+\infty}\!\!g_l(s)dB_l(s)\int\!g'_{A(\eta)}(\partial_1\phi_l)|u_n|^2dx\bigg|
			\le\frac{C_1}{A(\eta)}.
		\end{align*}
		This implies that 
		\begin{align}
			\int_{|x|\ge A(\eta)}|u_n(T_0)|^2dx\le&\int_{\mathbb{R}^d}|u_n(T_0)|^2g_{A(\eta)}dx\nonumber\\
			\le&\int_{\mathbb{R}^d}|u_n(\widetilde{T}_0)|^2g_{A(\eta)}dx+\int_{T_0}^{\widetilde{T}_0}\bigg|\frac{d}{dt}\int_{\mathbb{R}^d}|u_n(t)|^2g_{A(\eta)}dx\bigg|dt\nonumber\\
			\le&\int_{|x|\ge\frac{1}{2}A(\eta)}|u_n(\widetilde{T}_0)|^2dx+\frac{C_1}{A(\eta)}(\widetilde{T}_0-T_0).
			\label{small5}
		\end{align}
		
		Thus, setting $A(\eta)=\max\big\{A_0,~2A_1,~\frac{8C_1(\widetilde{T}_0-T_0)}{\eta}\big\}$ and combining \eqref{small1}, \eqref{small4} and \eqref{small5} together we obtain
		\begin{align}\label{<A eta}
			\int|u_n(T_0)-u_0|^2dx=&\int_{|x|\le A(\eta)}|u_n(T_0)-u_0|^2dx+\int_{|x|>A(\eta)}|u_n(T_0)-u_0|^2dx\nonumber\\
			\le&\int_{|x|\le A(\eta)}|u_n(T_0)-u_0|^2dx+2\int_{|x|>A(\eta)}|u_n(T_0)|^2dx+2\int_{|x|>A(\eta)}|u_0|^2dx\nonumber\\
			\le&\int_{|x|\le A(\eta)}|u_n(T_0)-u_0|^2dx+\eta.
		\end{align}
		By the compact embedding $H^1(B_{\mathbb{R}^d}(A(\eta)))\hookrightarrow L^2(B_{\mathbb{R}^d}(A(\eta)))$,
		$$
		\lim\limits_{n\to+\infty}\int_{|x|\le A(\eta)}|u_n(T_0)-u_0|^2dx=0,
		$$
		which along with \eqref{<A eta} yields that $\lim_{n\to+\infty}\int|u_n(T_0)-u_0|^2dx\le\eta$, 
		thereby proving \eqref{sL} 
		due to the arbitrariness of  $\eta>0$.

		\medskip 
		Now, we consider equation
		\begin{equation}
			\begin{cases}
				i\partial_tu+(\Delta+b_*\cdot\nabla+c_*)u+|u|^{p-1}u=0,\\
				u(T_0)=u_0.
			\end{cases}
			\label{m2}
		\end{equation}
		By \eqref{5.2} and \eqref{sL}, the standard well-posedness theory shows that there exists a unique $L^2$-solution $u$ to \eqref{m2} on $[T_0,+\infty)$,
		where $T_0$ is the universal time in Proposition \ref{prop5.4},
		such that
		\begin{equation}
			\lim_{n\to+\infty}\|u_n(t)-u(t)\|_{L^2}=0,
			\ \ \forall t\in [T_0, \infty).
			\label{m3}
		\end{equation}
		The preservation of $H^1$-regularity also yields $u(t)\in H^1$ for $t\in [T_0,+\infty)$.
		
		Moreover, by the uniform estimates in Theorem \ref{th5.1},
		for any $t\in [T_0, \infty)$
		and for $n$ large enough,
		\begin{align}\label{esti-un-R}
			\|u_n(t)-R(t)\|_{H^1}
			\leq C  t \phi^{\frac{1}{2}}(\wt \delta t).
		\end{align}
		Together with \eqref{m3},
		this yields that up to a subsequence (still denoted by \{n\} which may depend on $t$),
		\begin{equation}
			u_n(t)-R(t)\rightharpoonup u(t)-R(t)\ \ \text{weakly in} \ \ H^1,\ \text{as}\ n\to\infty.
		\end{equation}
		Letting $n\to+\infty$ in \eqref{esti-un-R} we thus obtain \eqref{1.32} and finish the proof of Theorem \ref{th1.6}.
	\end{proof}
	
	\begin{proof}[Proof of Theorem \ref{th1.3}]
		\eqref{1.20} follows directly from \eqref{1.28} and \eqref{1.32}. To prove \eqref{1.22}, it suffices to prove that
		\begin{equation}
			\|X(t)-e^{-W_*(t)}X(t)\|_{H^1}\le C\sum\limits_{k=1}^{N}\left(\int_{t}^{\infty}g_k^2ds\log\left(\int_{t}^{\infty}g_k^2ds\right)^{-1}
			\right)^{\frac{1}{2}}=:CL(t).
			\label{m4}
		\end{equation}
		
		To this end, using the inequality $|1-e^{ix}|\le2|x|$ for any $x\in\mathbb{R}$, the explicit formula \eqref{1.21} and the mass conservation of $X(t)$, we compute that
		\begin{align}
			\Vert X(t)-e^{-W_*(t)}X(t)\Vert_{H^1}
			\le&\Vert(1-e^{-W_*(t)})X(t)\Vert_{L^2}+\Vert\nabla(1-e^{-W_*(t)})X(t)\Vert_{L^2}+\Vert(1-e^{-W_*(t)})\nabla X(t)\Vert_{L^2}\nonumber\\
			\le&C\| W_*(t)\|_{W^{1,\infty}}(\|X(t)\|_{L^2}+\|\nabla X(t)\|_{L^2})\nonumber\\
			\le&CL(t)(1+\|\nabla X(t)\|_{L^2}),
			\label{m5}
		\end{align}
		where we also used the Levy H\"older continuity
		estimate of Brownian motions in the last step.
		
		Note that, by \eqref{1.28}, \eqref{1.32} and the mass conservation law of $X(t)$,
		\begin{align}
			\|\nabla X(t)\|_{L^2}\le&\|e^{W_*(t)}\nabla u(t)\|_{L^2}+\|\nabla W_*(t)X(t)\|_{L^2}\nonumber\\
			\le&\|u(t)-R(t)\|_{H^1}+\|R(t)\|_{H^1}+\|X(t)\|_{L^2}\le C,
			\label{m6}
		\end{align}
		where $C$ is independent of $t$,
		It follows that $\{X(t)\}$ is uniformly bounded in the energy space.
		
		Therefore,
		plugging \eqref{m6} into \eqref{m5} we obtain \eqref{m4} and finish the proof of Theorem \ref{th1.3}.
	\end{proof}

	\section*{Acknowledgments}
	Y. Sun and D. Zhang would like to thank Professor Yingchao Xie for many valuable discussions to improve this paper. 
    We gartefully acknowledge the funding by the Deutsche Forschungsgemeinschaft
      (DFG, German Research Foundation) – Project-ID 317210226 –
SFB 1283. 
    Y. Su is supported by NSFC grant (No. 12371122).
	D. Zhang is also grateful for the NSFC grants (No. 12271352, 12322108) and Shanghai Frontiers Science Center of Modern Analysis.

\end{document}